\documentclass[11pt,leqno]{article}
\usepackage{amssymb,amsfonts}
\usepackage{amsmath,amsthm,amsxtra}
\usepackage{epsfig}
\usepackage{slashed}
\usepackage{color}
\usepackage{verbatim}
\usepackage[notcite,notref]{showkeys}

\newcommand\bigcheck[1]{#1 \raise1ex\hbox{$\hspace{-1ex}{}^\vee$}}
\newcommand\sucheck[1]{#1 \raise0.5ex\hbox{$\hspace{-1ex}{}^\vee$}}

\newcommand{\add}{{\rm add}}

\newcommand{\ch}{{\rm ch}}

\newcommand{\even}{\mathop{\rm even \, }}
\newcommand{\End}{\mathop{\rm End }}

\newcommand{\im}{\mathop{\rm im  \, }}
\renewcommand{\Im}{\mathop{\rm Im  \, }}

\newcommand{\odd}{{\rm odd}}

\newcommand{\re}{\mathop{\rm re  \, }}

\newcommand{\Res}{\mathop{\rm Res  \, }}
\newcommand{\rank}{\rm rank \, }

\renewcommand{\sl}{s\ell}
\newcommand{\str}{{\rm str}}

\newcommand{\sdim}{\mathop{\rm sdim \, }}
\newcommand{\sign}{\rm sign \, }

\newcommand{\tr}{\mathrm{tr} \, }
\newcommand{\tw}{\rm tw \, }

\newcommand{\vac}{|0\rangle}

\newcommand{\bo}{\bar{1}}
\newcommand{\bz}{\bar{0}}

\renewcommand{\Re}{\mathop{\rm Re  \, }}

\newcommand{\CC}{\mathbb{C}}

\newcommand{\NN}{\mathbb{N}}
\newcommand{\QQ}{\mathbb{Q}}
\newcommand{\RR}{\mathbb{R}}
\newcommand{\ZZ}{\mathbb{Z}}

\newcommand{\fg}{\mathfrak{g}}
\newcommand{\fh}{\mathfrak{h}}

\newcommand{\fn}{\mathfrak{n}}

\newcommand{\so}{\mathfrak{so}}

\newcommand\bl{(\, . \, | \, . \, )}

\renewcommand{\tilde}{\widetilde}
\renewcommand{\hat}{\widehat}

\makeatletter
\renewcommand\section{\@startsection {section}{1}{\z@}%
                                   {-3.5ex \@plus -1ex \@minus -.2ex}%
                                   {2.3ex \@plus.2ex}%
                                   {\normalfont\large\bfseries}}
\renewcommand\subsection{\@startsection{subsection}{2}{\z@}%
                                     {-3.25ex\@plus -1ex \@minus -.2ex}%
                                     {0ex \@plus .0ex}%
                                     {\normalfont\normalsize\bfseries}}

\@addtoreset{equation}{section}
\makeatother

\newtheorem{theorem}{Theorem}[section]

\newtheorem{lemma}[theorem]{Lemma}
\newtheorem{corollary}[theorem]{Corollary}
\newtheorem{proposition}[theorem]{Proposition}
\newtheorem{conjecture}[theorem]{Conjecture}
\newtheorem*{lemma*}{Lemma}

\theoremstyle{remark}
\newtheorem{remark}[theorem]{Remark}
\newtheorem{example}[theorem]{Example}

\makeatletter
\def\@maketitle{\newpage
 \null
 \vskip 2em
 \begin{center}%
 \vskip 3em
  {\Large\bf \@title \par}%
  \vskip 1.5em
  {\normalsize
   \lineskip .5em
   \begin{tabular}[t]{c}\@author
   \end{tabular}\par}%
  \vskip 2em

 \end{center}%
 \par
 \vskip 2.5em}
\makeatother

\renewcommand{\epsilon}{\varepsilon}

\definecolor{light}{gray}{.9}

\newcommand{\half}{\frac{1}{2}}
\newcommand{\thalf}{\tfrac{1}{2}}
\newcommand{\zp}{\ZZ_{>0}}
\newcommand{\spr}{s^{\prime}}
\newcommand{\la}{\lambda}
\newcommand{\La}{\Lambda}
\newcommand{\al}{\alpha}
\newcommand{\osp}{\mathrm{osp} \,}
\newcommand{\wg}{\widehat{\fg}}

\begin{document}

\title{Representations of affine superalgebras and mock theta functions III}


\author{Victor G.~Kac\thanks{Department of Mathematics, M.I.T, 
Cambridge, MA 02139, USA. Email:  kac@math.mit.edu~~~~Supported in part by an NSF grant.} \ 
and Minoru Wakimoto\thanks{Email: ~~wakimoto@r6.dion.ne.jp~~~~.
Supported in part by Department of Mathematics, M.I.T.}}

\maketitle

\noindent Dedicated to Jean-Pierre Serre on his 90th birthday.

\section*{Abstract}
We study modular invariance of normalized supercharacters of tame 
integrable modules over an affine Lie superalgebra, associated to an arbitrary basic Lie superalgebra $ \fg. $ For this we develop a several step modification process of multivariable mock theta functions, where at each step a Zwegers' type ``modifier'' is used. We show that the span of the resulting modified normalized supercharacters is $ SL_2(\ZZ) $-invariant, with the transformation matrix equal, in the case the Killing form on $\fg$ is non-degenerate, to that for the subalgebra $ \fg^! $ of $ \fg, $ orthogonal to a maximal isotropic set of roots of $ \fg. $ 

\section{Introduction}
This paper is the third in the series of our papers on modular invariance of 
modified normalized characters of irreducible highest weight representations 
$L(\Lambda)$ over an affine Lie superalgebra $\hat\fg$, associated to a simple finite-dimensional Lie superalgebra $\fg$. 

We assume that the Lie superalgebra $\fg$ is {\it basic}, i.e. it is endowed with a non-degenerate invariant supersymmetric bilinear form $(.|.)$ and  its even part $\fg_{\bar{0}}$ is a 
reductive subalgebra. (These properties hold if the Killing form $\kappa$
on $\fg$ is non-degenerate.) 
The associated affine Lie superalgebra is
$\hat{\fg} = \fg [t,t^{-1}] \oplus \CC K \oplus \CC d$, where $K$ is a central element, $\fg [t, t^{-1}] \oplus \CC K$ is a central extension of the loop algebra $\fg [t, t^{-1}]$:
\[ [at^m , bt^n] = [a,b]t^{m+n} + m \delta_{m, -n} (a | b)K,\, a,b \in \fg,\,
 m, n \in \ZZ,  \]
and $ d = t \frac{d}{dt}$.

Choosing a Cartan subalgebra $\fh$ of $\fg_{\bar{0}}$, one defines the Cartan subalgebra of $\hat{\fg}$:

\[\hat{\fh} = \CC d + \fh + \CC K .\]
The restriction of the bilinear form $(.|.)$ to $\fh$ is symmetric 
non-degenerate, and one extends it from $\fh$ to $\hat{\fh}$, letting

\[(\fh | \CC K + \CC d) = 0,\quad  (K | K) = (d | d) = 0,\quad  (K | d) = 1.\] 
One identifies $\hat{\fh}^*$ with $\hat{\fh}$ via this bilinear form. 
Traditionally, the elements of $\hat{\fh}^*$ corresponding to $K$ and $d$ are denoted by $\delta$ and $\Lambda_0$, respectively. One uses the following coordinates on $\hat{\fh}$: 
\begin{equation}
\label{eq:0.1} 
h= 2\pi i(-\tau d + z + tK) = : (\tau, z, t),\,\,\hbox{where}\,\,\tau, 
t \in \CC,\, z \in \fh.
\end{equation}

Given a set of simple roots $\hat{\Pi} =
\{\alpha_0, \alpha_1, \ldots, \alpha_\ell \}$ of $\hat{\fg}$ and 
$\Lambda \in \hat{\fh}^*$, 
one defines the \textit{highest weight module} $L(\Lambda)$ over $\hat{\fg}$ as the irreducible module, which admits a non-zero vector $v_\Lambda$, such that 
\[h v_\Lambda = \Lambda(h)v_\Lambda \,\, \hbox{for}\,\, 
 h \in \hat{\fh}\,\, \hbox{and}\,\, \hat{\fg}_{\alpha_i}v_\Lambda = 0\, 
\,\hbox{for}\,\,\, 
i = 0, \ldots , \ell, \] 
where $\hat{\fg}_{\alpha_i}$ denotes the root subspace of $\hat{\fg}$, 
attached to the simple root $\alpha_i$. 
Since $K$ is a central element of $\hat{\fg}$, it is represented by a scalar 
$k=\Lambda(K)$, called the \textit{level} of $L(\Lambda)$ (or $\Lambda$). 

The \textit{character} $\mathrm{ch}^+$ and the \textit{supercharacter} $ch^-$ of $L(\Lambda)$ are defined as the following series, corresponding to the weight space decomposition of $L(\Lambda)$ with respect to $\hat{\fh}$:
\[  ch^\pm_{L(\Lambda)} (\tau, z, t) = tr^\pm_{L(\Lambda)}e^{2\pi i (-\tau d + z + tK)},        
\]
where $tr^+$ (resp. $tr^-$) denotes the trace (resp. supertrace). 
It is easy to see (as in \cite{K2}, Chapter 10) that these series converge 
absolutely in the 
domain $\{h \in \hat{\fh} |\, \mathrm{Re}\, \alpha_i (h) > 0,\, i = 0, 1, \ldots, \ell \}$ to holomorphic functions. In all examples these functions extend to meromorphic functions in the domain
\begin{equation}
\label{eq:0.2}
X =\left\lbrace h \in \hat{\fh} |\, \mathrm{Re}(K | h) > 0\right\rbrace  = \left\lbrace (\tau, z, t) |\, \mathrm{Im}\tau > 0\right\rbrace. 
\end{equation} 

Note that, as a $\fg[t,t^{-1}] \oplus \CC K$-module, $L(\Lambda)$ remains 
irreducible and it is unchanged if we replace $\Lambda$ by $\Lambda + a\delta,\, a \in \CC$, and the character of the $\hat{\fg}$-module gets multiplied by $q^a$. Throughout the paper $ q = e^{2\pi i \tau} = e^{-\delta}$.

In the case when $\fg$ is a simple Lie algebra, there exists an important collection of integrable  $\hat{\fg}$-modules $L(\Lambda)$, whose normalized characters have modular invariance property 
\cite{K2}. Recall that the \textit{normalized (super)character} 
$\mathrm{ch}_{\Lambda}^\pm$ is defined as 
\begin{equation}
\label{eq:0.3}
 \mathrm{ch}^\pm_{\Lambda} (\tau, z, t) = q^{m_{\Lambda}} \mathrm{ch}^\pm_{L(\Lambda)} (\tau, z, t),
\end{equation}
where $m_\Lambda $ is given by the following formula: 
\[ m_{\La} = \frac{| \La + \widehat{\rho}|^2}{2 (k+h^\vee)} - \frac{\sdim \fg}{24}, \]
where $\hat{\rho}$ is the affine Weyl vector, defined by $2(\hat{\rho}|\alpha_i)
=(\alpha_i|\alpha_i)$ for all $\alpha_i\in \hat{\Pi}$,
and $h^\vee$ is its level, called the dual Coxeter number.
Throughout the paper we assume that the level of $\Lambda$ is non-critical, i.e. 
$ k + h^\vee \neq 0. $
Note that $\mathrm{ch}^\pm_{\Lambda + aK} = 
\mathrm{ch}^\pm_\Lambda,\, a \in \CC$. 

Recall the action of $SL_2(\RR)$ in the domain $X$ in coordinates (\ref{eq:0.1}):
\begin{equation}
\label{eq:0.4}
\left( 
\begin{array}{cc}
a & b\\ 
c & d\\
\end{array}
\right) \cdot (\tau, z, t) = \left( \frac{a \tau + b}{c \tau + d},\: \frac{z}{c \tau + d},\: t - \frac{c (z|z)}{2(c \tau + d)}\right) .
\end{equation}
By definition, modular invariance of 
a finite family of functions on $X$  
means that the $\CC$-span of this family is $SL_2(\ZZ)$-invariant with respect to the action  (\ref{eq:0.4}).

If the Killing form $\kappa$ on $\fg$ is non-degenerate (which is equivalent
to $h^\vee \neq 0$) a 
$\hat{\fg}$-module $L(\Lambda)$ is called {\it integrable} if for any root $\alpha$ of
$\hat{\fg}$, such that $\kappa(\alpha,\alpha)>0$, the elements from the root space $\hat{\fg}_{\alpha}$ act locally nilpotently on $L(\Lambda)$. If $\fg$ is
a Lie algebra, this property is equivalent to integrability, as defined in
\cite{K2}. For the Lie superalgebras $\fg$ with $h^\vee = 0$ 
the definition of integrability is similar. However, for the sake of simplicity, we shall restrict the discussion in the most of the introduction to the 
case of non-degenerate $\kappa$.

Recall that the \textit{defect} of the Lie superalgebra $ \fg $ is the maximal number $ d $ of linearly independent pairwise orthogonal isotropic roots of $ \fg. $ An integrable $ \wg $-module $ L(\La) $ is called \textit{tame} if the set of simple roots $ \widehat{\Pi} $ contains a $ d $-element set $ T,  $ consisting of $ d $ pairwise orthogonal isotropic roots, orthogonal to $ \La. $ The (super)character formula for a tame integrable $ \wg $-module $ L(\La), $ conjectured in \cite{KW3} and proved in \cite{GK} (for ``good'' choices of $ T $) reads:

\begin{equation}
\label{eq:0.5}
\hat{R}^\pm \ch^{\pm}_{L(\Lambda)} = 
\sum_{w \in \hat{W}^\#}^{} \epsilon^\pm (w) w \frac{e^{\Lambda + \hat{\rho}}}{\prod_{\beta \in T}{(1 \pm e^{-\beta})}}.
\end{equation}
Here $\hat{R}^\pm$ is the affine (super)denominator:
\[ \hat{R}^\pm = e^{\hat{\rho}} \frac{\prod_{\alpha \in \hat{\Delta}_{\bar{0}, +}}^{} (1- e^{- \alpha})}{\prod_{\alpha \in \hat{\Delta}_{\bar{1},+}}^{} (1 \pm e^{- \alpha})}, 
\]
$\hat{\Delta}_{\bar{0},+}$ and 
$\hat{\Delta}_{\bar{1},+}$ 
are the sets of 
positive even and odd roots of $\hat{\fg}$ (counting multiplicities); $\hat{W} = W \ltimes t_L$ is the affine Weyl group , where $ W $ is the 
Weyl group of $\fg_{\bar{0}}$, and the subgroup $t_L$ consists of 
\textit{translations} 
$t_\gamma$, $\gamma \in L$, where $L$ is the coroot lattice of 
$\fg_{\bar{0}}$, which are defined by
\begin{equation}
\label{eq:0.6}
t_\gamma (\lambda) = \lambda + \lambda (K) \gamma - ((\lambda|\gamma)+ \frac{1}{2} \lambda(K) (\gamma | \gamma))\delta,\,\,\lambda \in \hat{\fh}^*;
\end{equation} 
$\hat{W}^\# = W^\# \ltimes t_{L^\#}$ (resp. $W^\#$)
is the subgroup of the affine (resp. finite) Weyl group, generated by 
reflections in 
$\alpha \in \hat{\Delta}_+$  (resp. $\alpha \in \Delta_+$) with $\kappa(\alpha,\alpha)>0$,  where $L^\#$ is the sublattice of $L$, 
spanned by the coroots $\alpha$ with $\kappa (\alpha, \alpha) > 0$ of 
$\fg_{\bar{0}}$;
finally, $\epsilon^\pm (w) = (-1)^{s_\pm (w)} $, for a decomposition of $w$ 
in a product of $s_+$ reflections, with respect to non-isotropic even roots, and $s_-$ is the number of those of them, for which the half is not a root. 

Note that formula (\ref{eq:0.5}) for the supercharacter can be rewritten, after multiplying both sides by a suitable power of $q$, as
\begin{equation}
\label{eq:0.7}
q^{\frac{\mathrm{sdim} \fg}{24}} \hat{R}^- \ch^-_{\Lambda} = 
\sum_{w \in W^\#} \epsilon^- (w)\, 
w (\Theta^{L^{\#}}_{\Lambda + \hat{\rho} , T, \epsilon^-}),
\end{equation}
where $\Theta^{L^{\#}}_{\Lambda + \hat{\rho}, T, \epsilon^-}$ is a (signed)  mock theta function of degree $k + h^{\vee}$.
 
Recall that 
for $\lambda \in \hat{\fh}^*$, such that $\lambda(K) > 0$, a \textit{(signed) mock theta function} $\Theta^Q_{\lambda, T, \epsilon}$ of degree $n = \lambda (K)$ and defect $ d = |T|, $ attached to a lattice $ Q, $ is defined by the following series \cite{KW6}, \cite{KW7}:
 \begin{equation}
\label{eq:0.8}
 \Theta^Q_{\lambda, T, \epsilon} = q^{\frac{(\lambda | \lambda)}{2n}} \sum_{\gamma \in Q}^{} \epsilon (\gamma) t_\gamma \frac{e^\lambda}{\prod_{\beta \in T}^{} (1 - e^{- \beta})   }, 
 \end{equation}
 where $Q \subset \fh$ is a positive definite integral lattice, $ \epsilon : Q \rightarrow \{ \pm 1 \} $ is a homomorphism, $t_\gamma$ are the translations, defined by (\ref{eq:0.6}), and $T \subset \widehat{\fh}$ is a finite subset, consisting of pairwise orthogonal isotropic vectors, orthogonal to $\lambda$. This series converges to a meromorphic function in the domain $X$, which in coordinates (\ref{eq:0.1}) takes the form
\begin{equation}
\label{eq:0.9}
\Theta^Q_{\lambda, T, \epsilon} (\tau, z, t) = e^{2 \pi int} \sum_{\gamma \in \frac{\bar{\lambda}}{n}+Q}^{} \epsilon (\gamma - \frac{\bar{\la}}{n}) \frac{q^{n \frac{(\gamma | \gamma)}{2}}  e^{2 \pi i n \gamma (z)}}{\prod_{\beta \in T}^{}(1 - q^{-(\gamma| \beta)} e^{-2\pi i \beta (z)})}.
\end{equation}
\noindent Of course, if $T = \emptyset$, we get the usual (signed) Jacobi form, which is contained in a finite modular invariant family (up to a weight factor). 

The normalized superdenominator $q^{\frac{\mathrm{sdim} \fg}{24}} \widehat{R}^-$ is modular invariant (up to the same weight factor), since it can be expressed as a ratio of products of the standard Jacobi form $\vartheta_{11}$ of degree 2 and powers of the $ \eta $-function, see (\ref{eq4.5}). 
Therefore, modular invariance of supercharacters reduces to that of 
numerators, i.e. the RHS of (\ref{eq:0.7}).
 
Note that the normalized denominator $ q^{\frac{\sdim \fg}{24}} \widehat{R}^+ $ is expressed in terms of the Jacobi forms $ \vartheta_{11} $ and $ \vartheta_{10}, $ and the powers of the $ \eta $-functions, see (\ref{eq4.6}), hence it is not quite modular invariant, but is a member of a modular invariant family of three functions \cite{KW6}. Likewise, in general, if a family of supercharacters is modular invariant, then the family of characters can be included in a modular invariant family (see subsection 6.6). For that reason, in the remainder of the paper we discuss supercharacters, rather than characters. 

If $\fg$ is a simple Lie algebra, or a defect 0 Lie superalgebra ($ = \osp (1 | n) $), and $L(\Lambda)$ is an integrable $\hat{\fg}$-module, then formula (\ref{eq:0.5}) turns into the usual Weyl-Kac character formula, where $\hat{W}^\#=\hat{W}$, and $T = \emptyset$ (in this case, of course, $\mathrm{ch}^+ = \mathrm{ch}^-$ and $\epsilon_+ (w) = \epsilon_- (w) = \mathrm{det} (w)$). Therefore (\ref{eq:0.7}) holds with the usual Jacobi forms, which easily implies that the finite set of normalized characters of integrable $ \wg $-modules of given level is modular invariant, see \cite{K2}, Chapter 13.  

However, modular invariance fails for mock theta functions, but sometimes it can be achieved by adding non-holomorphic real analytic corrections, discovered by Zwegers \cite{Z}.

The key role in the modification procedure of \cite{Z} and our papers \cite{KW6}, \cite{KW7}, has been played by the following rank 1 mock theta functions ($ m \in \zp, s \in \ZZ $):
\begin{equation}
\label{eq0.10}
\Phi^{[m;s]} (\tau, z_1, z_2) = \sum_{n \in \ZZ} \frac{e^{2 \pi i m n (z_1 + z_2) + 2 \pi i s z_1} q^{mn^2 + sn}}{1 - e^{2 \pi i z_1 } q^n}.
\end{equation}
In order to make these functions modular invariant one introduces the following \textit{modifier}, see \cite{Z}, \cite{KW6} and (\ref{eq1.14}) of the present paper:
\begin{equation}
\label{eq0.11}
\Phi^{[m;s]}_{\mathrm{add}} (\tau, z_1, z_2) = \sum_{  \substack{j \in s + \ZZ \\ s \leq j < s + 2m}} R_{j,m} (\tau, \frac{z_1 - z_2}{2}) \ \Theta_{j,m} (\tau, z_1 + z_2),
\end{equation}
\noindent where $ \Theta_{-j,m} (\tau, z) $ are rank one Jacobi forms, given by (\ref{eq1.7}), and $ R_{j,m} (\tau, z) $ are certain real analytic, but not holomorphic, functions, given by (\ref{eq1.12}), see \cite{Z}, \cite{KW7}. Then the \textit{modified} mock theta functions 
\begin{equation}
\label{eq0.12}
\tilde{\Phi}^{[m;s]} = \Phi^{[m;s]} - \half \Phi^{[m;s]}_{\mathrm{add}}, m \in \zp, s \in \ZZ, 
\end{equation}
\noindent satisfy beautiful modular and elliptic transformation properties, 
described by Theorem \ref{th1.1}, cf. \cite{Z}, \cite{KW6}, \cite{KW7}.
In particular, it follows from Theorem \ref{th1.1} that the function $ \tilde{\Phi}^{m;s]} $ is independent of $ s \in \ZZ, $ hence it is denoted by $ \tilde{\Phi}^{[m]}, m \in \zp. $

As we have shown in \cite{KW6}, \cite{KW7}, the functions $ \Phi^{[m;s]} $ appear very naturally in represnetation theory of $ \wg $, where $ \fg = s \ell (2 | 1), $ the simplest basic Lie superalgebra of non-zero defect. In this case, choosing both simple roots $ \al_1, \al_2  $ of $ s\ell (2|1) $ isotropic with $ (\al_1 | \al_2) = 1,  $ and taking $ T = \{ \al_1\}, $ the weights 
\[ \La_{m;s} = (m-s) \La_0 + s \al_1, \mbox{ where } m, s \in \ZZ_{\geq 0}, \ 0 \leq s \leq m, \]
\noindent are highest weights of tame integrable $\wg $-modules of level $m$, and formula (\ref{eq:0.7}) for supercharacters reads:
\begin{equation}
\label{eq0.13}
(\widehat{R}^- \ch^-_{\Lambda_{m;s}}) (\tau, z_1, z_2, t) = e^{2 \pi i (m+1)t} \left(\Phi^{[m;s]} (\tau, z_1, z_2) - \Phi^{[m;s]} (\tau, -z_2, -z_1) \right),
\end{equation}
\noindent where
\[ \widehat{R}^- (\tau, z_1, z_2, t) = i e^{2 \pi i t } \frac{\eta (\tau)^3 \vartheta_{11} (\tau, z_1 + z_2)}{\vartheta_{11} (\tau, z_1) \vartheta_{11} (\tau, z_2)}, \]
\noindent and, in coordinates (\ref{eq:0.1}), $ z = -z_1 \al_2 - z_2 \al_1. $

Consequently, the modified supercharacter $ \tilde{\ch}^-_{L(\La_{m;s})} = \tilde{\ch}^-_{\La_{m;s}}, $ obtained by replacing in the RHS of (\ref{eq0.13}) the function $ \Phi^{[m;s]} $ by its modification $ \tilde{\Phi}^{[m;s]} = \tilde{\Phi}^{[m]}, $ is independent of $ s $ and is modular invariant. Thus after the modification, the non-modular invariant family of $ m+1  $ supercharacters $ \{ \ch^-_{\La_{m;s}} |\, 0 \leq s \leq m \} $  turns in one modular invariant modified supercharacter 
\[ \tilde{\ch}^-_{m \La_0} (\tau, z_1, z_2, t) = e^{2 \pi i m t} \widehat{R}^- (\tau, z_1, z_2, 0)^{-1} \left( \tilde{\Phi}^{[m]} (\tau, z_1, z_2) - \tilde{\Phi}^{[m]} (\tau, -z_2, -z_1)  \right). \]

The only basic Lie superalgebras of rank 2 and non-zero defect are $ s \ell (2|1) $ and $ \osp (3|2) $ (both have defect 1). The essential difference between these two cases is that in the first case the function $ \epsilon_- (t_{\al}) $ is identically 1 for $ \al \in L^{\#}, $ while in the second case it is not. As we have shown in \cite{KW7}, in order to treat the second case along the same lines as the first one, one needs to consider the following rank 1 signed mock theta functions ($ m \in \half \zp, s \in \half \ZZ $)
\begin{equation}
\label{eq0.14}
\Phi^{-[m;s]} (\tau, z_1, z_2) = \sum_{n \in \ZZ} (-1)^n \frac{e^{2 \pi i m n (z_1 + z_2) + 2 \pi i s z_1} q^{mn^2 + sn}}{1 - e^{2 \pi i z_1} q^n},
\end{equation}
\noindent and the signed analogues $ R^-_{j,m} (\tau, z) $ of the real analytic functions $ R_{j,m}, $ given by (\ref{eq1.12}). Then the signed modifier $ \Phi^{-[m;s]}_{\mathrm{add}} $ is defined by a formula, similar to (\ref{eq0.11}), where the $ R_{j,m} $ are replaced by the $ R^-_{j,m} $ and the Jacobi forms $ \Theta_{j,m} $ by signed Jacobi-forms $ \Theta^-_{j,m}, $ defined by (\ref{eq1.8}); the modified signed mock theta functions $ \tilde{\Phi}^{-[m;s]} $ are defined by (\ref{eq0.12}), where $ \Phi $ is replaced by $ \Phi^-. $ The functions $ \tilde{\Phi}^{-[m;s]} $ satisfy beautiful modular and elliptic transformation properties, described by Theorem \ref{th1.3}. In particular, it follows from Theorem \ref{th1.3} that the functions $ \tilde{\Phi}^{-[m;s]} $ and $ \tilde{\Phi}^{-[m;s^{\prime}]} $ are equal if $ s - s^{\prime} \in \ZZ. $

The above two modifications are called type A and type B modifications. 

The main goal of the present paper is to obtain similar results for arbitrary basic Lie superalgebra $ \fg $ with non-zero defect, different from $ p s \ell (n|n). $ (The case $ \fg = p s \ell (2|2) $ was treated in \cite{KW6}.) For this we develop in Section 3 a $ d $-step modification procedure of a mock (resp. signed mock) theta function of defect $ d, $ attached to a lattice of rank $\ell$. At each step we use a modification of type A (resp. B), the end result being a product of a theta (resp. signed theta) function, attached to a lattice of rank $ \ell - d, $ and a product of $ d $ functions of the form $ \tilde{\Phi}^{[m;s]} $ (resp. $ \tilde{\Phi}^{-[m;s]} $).

This modification procedure is used in Sections 5 and 6 to construct finite families of tame integrable $ \wg $-modules, whose modified normalized supercharacters form a modular invariant family. 
Namely, we choose a set of simple roots $ \widehat{\Pi} $ of $ \wg $ and a $ d $-element isotropic subset $ T \subset \widehat{\Pi} $ if $ \fg \neq \osp(2m|2n) $ with $ m > n+2 $ (resp. two $ d $-element isotropic subsets $ T, T^{\prime} \subset \widehat{\Pi} $ otherwise), such that the set of all level $ k $ integrable $ \wg $-modules $ L(\La) $ with $ (\La | T) = 0 $ (resp. $ (\La | T) = 0 $ or $ (\La | T^{\prime}) = 0 $) has the property that the modified normalized supercharacters form a modular invariant family. It turns out that these families are naturally parameterized by dominant integral level $ k $ weights of the Lie superalgebra $ \widehat{\fg^!}, $ where $ \fg^! $ is isomorphic to the defect 0 basic superalgebra of $ \fg, $ orthogonal to $ T $. Moreover, it turns out that the transformation matrix for the action of $ SL_2 (\ZZ) $ on these modified normalized supercharacters coincides with the well known matrices \cite{K2} for the action of $ SL_2 (\ZZ) $ on normalized supercharacters of integrable level $ k $ modules over $ \widehat{\fg^!}. $

At the end of the paper we treat the case of basic $ \fg \neq p s \ell (n|n) $ with $ h^{\vee} = 0, $ namely $ \fg = \osp (2n + 2 | 2n) $ and $ D(2,1;a), $ and the case of ``subprincipal'' integrable tame $ \osp (3|2)^{\hat {}} $-modules. We also construct explicitly all level 1 integrable 
$ \osp (M|N)^{\hat {}} $-modules, which allows us to establish modular invariance of their normalized supercharacters without modification.

The results of this paper were reported at a spring school in Bonn in March 2015. 

\section{Mock theta functions}
Fix an $ \ell $-dimensional vector space $ \fh $ over $ \CC $ with a 
non-degenerate symmetric bilinear form $ \bl. $ Fix a free abelian subgroup $ L $ in $ \fh $ of rank $ m > 0, $ such that the restriction of the bilinear form $ \bl $ to $ L $ is real valued positive definite. 

Let $ \hat{\fh} $ be the direct sum of $ \fh $ and the 2-dimensional space 
$ \CC K \oplus \CC d $, with the symmetric bilinear form $ \bl $ extended from $ \fh $ by 
\[ \left( \fh | \CC K \oplus \CC d \right) = 0, \ \left( K | K \right) = 0, \ \left( d | d \right) = 0, \ \left( K | d \right) = 1. \]
Since the bilinear form $ \bl $ is non-degenerate on $ \fh $ and $ \hat{\fh}$,
using this form we may identify $ \hat{\fh} $ with $ \hat{\fh}^{\ast}$, 
so that $ \fh $ is identified with $ \fh^{\ast}$. This form induces a bilinear 
form on $ \hat{\fh}^{\ast}$, also denoted by $ \bl$. The elements, corresponding to $ K $ and $ d $ in $ \hat{\fh}^{\ast} $ under this identification are traditionally denoted by $ \delta $ and $ \Lambda_0$.

For each $ \gamma \in \fh = \fh^{\ast} $ define the \textit{translation} $ t_{\gamma} \in \End \hat{\fh}^{\ast} $ by the formula
\begin{equation}
\label{eq1.1}
t_{\gamma} (\lambda) = \lambda + \lambda (K) \gamma - (\lambda (\gamma) + \tfrac{1}{2} \lambda (K) |\gamma|^2) \, \delta, \ \lambda \in \hat{\fh}^{\ast}.
\end{equation}
Here and further $ |\gamma|^2 $ stands for $ (\gamma |  \gamma ) $. 
It is easy to see that $ t_{\beta} t_{\gamma} = t_{\beta + \gamma} $ for $ \beta, \gamma \in \fh $, and that $ t_{\gamma} $ leaves the bilinear form $ \bl $ on $ \hat{\fh}^{\ast} $ invariant and fixes the element $ K. $

Let $ \lambda \in \hat{\fh}^{\ast} $ be such that its \textit{level} $ k:= \lambda (K) $ is a positive real number. Let $ T = \{ \beta_1, \ldots, \beta_n \} 
\subset
\hat{ \fh}^{\ast} $ 
be such that $ (\beta_i  | \beta_j) = 0 $ for all $i, j$,
$ (\lambda  |  T) = 0$, and $ n \leq m $ .  
Following  \cite{KW4}, define the \textit{mock theta function} of degree $ k $ 
and \textit{defect} $ n $ by the series 
\begin{equation}
\label{eq1.2}
\Theta_{\lambda, T} = q^{\frac{|\lambda|^2}{2k}}  \ \sum_{\gamma \in L} t_{\gamma} \ \frac{e^{\lambda}}{\prod_{\beta \in T} (1 - e^{- \beta})}.
\end{equation}
Here and further $ e^{\lambda} $ is a function on $ \hat{\fh}, $ defined by $ e^{\lambda} (h) = e^{\lambda(h)}, 
q = e^{- \delta}$, and $ t_{\gamma} e^{\lambda} = e^{t_{\gamma}(\lambda)}. $

Given a homomorphism $\epsilon:L\mapsto \{\pm 1\}$, depending only on $k$,
we define the \textit{signed} mock theta function of degree $ k $ by the series
\begin{equation}
\label{eq1.3}
\Theta_{\lambda, T, \epsilon} = q^{\frac{|\lambda|^2}{2k}}  \ \sum_{\gamma \in L}\epsilon (\gamma) 
t_{\gamma} \ \frac{e^{\lambda}}{\prod_{\beta \in T} (1 - e^{- \beta})}.
\end{equation}
Both series (\ref{eq1.2}) and (\ref{eq1.3}) remain unchanged if we replace $ \lambda $ by $ \lambda + a \delta, a \in \CC. $ Both series converge to a meromorphic function in the domain
\begin{equation}
\label{eq1.4}
X = \left\{ h \in  \hat{\fh} | \Re \delta (h) > 0\right\}.
\end{equation}

One uses the following coordinates on $ \hat{\fh}: $
\begin{equation}
\label{eq1.5}
h = 2 \pi i (z + tK - \tau d),
\end{equation}
\noindent where $ t, \tau \in \CC, z \in \fh $. 
In these coordinates, $ q = e^{2 \pi i \tau}, X = \left\{ (\tau, z, t) | \Im \tau > 0 \right\} $ and (\ref{eq1.2}) takes the form
\begin{equation}
\label{eq1.6}
\Theta_{\lambda, T} (\tau, z, t) = e^{2 \pi i k t } \ \sum_{\gamma \in L} \ \frac{q^{\frac{|\bar{\lambda} + k \gamma |^2}{2k}} e^{2 \pi i (\bar{\lambda} + k \gamma ) (z)}}{\prod_{\beta \in T} 
(1-q^{-(\gamma | \beta)} e^{-2 \pi i \beta (z)})}\,,
\end{equation}
\noindent and similarly for $ \Theta_{\lambda, T, \epsilon}$. Here and further $ \bar{\lambda} $ denotes the projection of $ \lambda  $ on $ \fh^{\ast}. $

Note that $ \Theta_{\lambda} : = \Theta_{\lambda, \emptyset} $ and 
$ \Theta^\epsilon_{\lambda} : = \Theta_{\lambda, \emptyset, \epsilon} $ are ordinary theta functions and signed theta functions of degree $ k$. Both converge to holomorphic functions in $ X $.

Sometimes we shall write $ \Theta^L_{\lambda, T} $ and 
$ \Theta^{L}_{\lambda, T, \epsilon}  $ in order to emphasize the dependence on the lattice $ L $.

We shall also need the following theta functions and signed theta function, associated to a lattice $ L $ of rank 1 (where we put $ t = 0 $):
\begin{equation}
\label{eq1.7}
\Theta_{j,m} (\tau, z) = \sum_{n \in \ZZ} \  e^{2 \pi i m z (n + \frac{j}{2m})} \, q^{m (n+ \frac{j}{2m})^2}, \ m \in  \ZZ_{>0}, \ j \in \ZZ \,, 
\end{equation}
\begin{equation}
\label{eq1.8}
\Theta^{\pm}_{j,m} (\tau, z) =  \sum_{n \in \ZZ} (\pm 1)^n \ e^{2 \pi i m z (n 
+ \frac{j}{2m})} \, q^{m (n+ \frac{j}{2m})^2}, \ m \in \tfrac{1}{4} \ZZ_{>0}, \ j \in \tfrac{1}{2} \ZZ \, .
\end{equation}
\noindent Note that $ \Theta_{j,m} $ and $ \Theta^+_{j,m} $ depend only on $ j \mod{2m}, $ and $ \Theta^-_{j,m} $ depends only on $ j \mod{4m}. $ As before and further on we always assume that $ |q| <1  $ (i.e. $ \Im \tau > 0 $) in order to guarantee the convergence. Note also an unfortunate clash of notation: 
$m$ in (\ref{eq1.7}) - (\ref{eq1.14}) has nothing to do with $m=\rank L$. 

Finally, we shall need the following (signed) mock theta functions, associated to a lattice $ L $ of rank 1 and $ |T| =1 $:
\begin{equation}
\label{eq1.9}
\Phi^{[m;s]} (\tau, z_1, z_2) = \sum_{n \in \ZZ} \frac{e^{2 \pi i m n (z_1 + z_2) + 2 \pi i s z_1} q^{mn^2 + sn}}{1-e^{2 \pi i z_1} q^n},\,\, m\in \ZZ_{>0}, s \in \ZZ, 
\end{equation}
\begin{equation}
\label{eq1.10}
\Phi^{\pm [m;s]} (\tau, z_1, z_2) = \sum_{n \in \ZZ} (\pm1)^n \frac{e^{2 \pi i m n (z_1 + z_2) + 2 \pi i s z_1} q^{mn^2 + sn}}{1-e^{2 \pi i z_1} q^n},\,\, m \in \tfrac{1}{2} \ZZ_{>0}, s \in \tfrac{1}{2} \ZZ, 
\end{equation}

It is well known that (signed) theta functions have nice modular and elliptic transformation properties (see e.g. Propositions A2 and A3 in the Appendix to \cite{KW7}).
This is not the case for (signed) mock theta functions. It was Zwegers who 
first found their modifications which do have these properties \cite{Z}.
Following Zwegers, introduce the following (slightly changed) real analytic, but not meromorphic, functions for $ m \in \ZZ_{> 0}, j \in \ZZ$:  
\begin{equation}
\label{eq1.11}
R_{j,m} (\tau, z) = \sum_{n \in \ZZ  \atop n\equiv j \! \mod{2m} } 
\left\{ \sign (n - \tfrac{1}{2} -j+2m) - E (\psi_{m,n} (\tau, z))\right\} 
e^{\frac{- \pi i n^2}{2m} \tau + 2 \pi i n z},
\end{equation}
\noindent where $ E(x) = 2 \int_{0}^{x} e^{-\pi u^2} du,\,\, x \in \RR,  $ and $ \psi_{m,n} (\tau, z) = (n - 2m \frac{\Im z}{\Im \tau}) \sqrt{\frac{\Im \tau}{m}}. $

We shall also need the following ``signed'' analogues of these functions, defined for $  m \in \tfrac{1}{2} \ZZ_{> 0}, j \in \tfrac{1}{2} \ZZ $ (cf. \cite{KW7}):
\begin{equation}
\label{eq1.12}
R^{\pm}_{j,m} (\tau, z) = \sum_{n \in \tfrac{1}{2}\ZZ \atop n \equiv j \, \mod{2m}} (\pm 1)^{\frac{n-j}{2m}}
\left\{ \sign (n - \tfrac{1}{2} -j+2m) - E (\psi_{m,n} (\tau, z)) \right\} e^{\frac{- \pi i n^2}{2m} \tau + 2 \pi i n z}. 
\end{equation}

Now, following \cite{Z}, \cite{KW7}, define the \textit{modifier} for $ m \in \ZZ_{>0}, s \in \ZZ: $
\begin{equation}
\label{eq1.13}
\Phi^{[m;s]}_{\add} (\tau, z_1, z_2) =  \sum_{j \in \ZZ \atop s \leq j < s + 2m} R_{j,m} \left( \tau, \frac{z_1 - z_2}{2} \right) \Theta_{j,m} (\tau, z_1 + z_2), 
\end{equation}
\noindent and the \textit{signed modifier} for $ m \in \tfrac{1}{2} \ZZ_{>0}, s \in \tfrac{1}{2} \ZZ: $
\begin{equation}
\label{eq1.14}
\Phi^{\pm [m;s]}_{\add} (\tau, z_1, z_2) =  \sum_{j \in s+\ZZ \atop s \leq j < s + 2m} R^{\pm}_{j,m} \left( \tau, \frac{z_1 - z_2}{2} \right) \Theta^{\pm}_{j,m} (\tau, z_1 + z_2).
\end{equation}
Define the 
\textit{modified} mock theta functions and \textit{modified signed} mock theta functions: 
\[ \tilde{\Phi}^{[m;s]}: = \Phi^{[m;s]} - \tfrac{1}{2} \Phi^{[m;s]}_{\add} 
\hbox{ and}\,\,
 \tilde{\Phi}^{\pm [m;s]}: = \Phi^{\pm [m;s]} - \tfrac{1}{2} \Phi^{\pm [m;s]}_{\add} \ .\]
The functions $ \tilde{\Phi}^{[m;s]} $ and $ \tilde{\Phi}^{\pm [m;s]} $ have beautiful modular and elliptic invariance properties, cf. \cite{Z}, \cite{KW6}, \cite{KW7}.

Before stating the results, let us recall the basic definitions on modular transformations. The group $ G = SL_2 (\RR) $ acts in the domain $ X $ in coordinates (\ref{eq1.5}) as follows:
\begin{equation}
\label{eq1.15}
\left( \begin{matrix}
a & b \\
c & d \\
\end{matrix}
\right) \cdot (\tau, z, t) = \left( \frac{a \tau + b}{c \tau + d}, \ \frac{z}{c \tau + d}, \  t - \frac{c (z | z)}{2 ( c \tau + d)} \right).
\end{equation}
This action induces the right action of \textit{weight} $ w \in \frac{1}{2} \ZZ $ of $ G $ on functions in $ X $:
\begin{equation}
\label{eq1.16}
F {{w \!\!\!\!\!} \atop \bigg{|}}_{\left( \begin{matrix}
a & b \\
c & d \\
\end{matrix}
\right)} (\tau, z, t) = (c \tau + d)^{-w} \ F \left( \left( \begin{matrix}
a & b \\
c & d \\
\end{matrix}
\right) \cdot (\tau, z, t) \right).
\end{equation}
\noindent (Actually,
this is an action of the double cover of $ G $ if $ w \in \frac{1}{2} +\ZZ $.)
We shall drop $w$ in this notation if $w= 0$.

\begin{remark}
\label{rem1.1}
Action (\ref{eq1.16}) induces the following action of the standard basis elements $ e, h,  $ and $ f $ of $ s \ell_2 (\CC) $ on the space of holomorphic functions 
$ \mathcal{F} $ on $ X $ ($F\in \mathcal{F} $): 
\begin{eqnarray*}
e\,F & =& \frac{\partial}{\partial \tau} F, \\
h\,F & = &\left( w + 2 \tau \frac{\partial}{\partial \tau} + \sum_{i = 1}^{\ell} z_i \frac{\partial}{\partial z_i} \right) F, \\
f\,F &=& \left( -w \tau - \tau^2 \frac{\partial}{\partial \tau} - \tau \sum_{i = 1}^{\ell} z_i \frac{\partial}{\partial z_i} - \frac{(z|z)}{2} \frac{\partial}{\partial t}\right) F. \\
\end{eqnarray*}
This action commutes in the following way with the Laplace operator $ D $ (defined in the Appendix to \cite{KW6}):
\[ \left[ e , D\right] = 0, \quad \left[ h, D\right] = - 2D, \quad \left[ f, D\right] = 2 \tau D. \]
We thus obtain a representation of the Lie algebra $ s \ell_2 (\CC) $ on $ \mathcal{F} D. $
\end{remark}

Let $ k \in \CC. $ If $ F $ is a function on $ X, $ such that 
\[ F(h + aK) = e^{ka} \, F(h), \quad h \in X, \]
\noindent we say that F has \textit{degree} $ k. $ In this case we can write:
\[ F(\tau, z, t) = e^{2 \pi i kt } \, F(\tau, z), \mbox{ where } F(\tau, z) = F (\tau, z, 0). \]
The space of function of degree $ k $ on $ X $ is $ G $-invariant with respect to the right action of (\ref{eq1.16}), which obviously induces the following right action of $ G $ of \textit{weight} w and \textit{degree} k on functions in $ \tau $ and $ z $:
\begin{equation}
\label{eq1.17}
F {{w,k \!\!\!\!\!} \atop \bigg{|}}_{\left( \begin{matrix}
a & b \\
c & d \\
\end{matrix}
\right)}  (\tau, z) = (c \tau + d)^{-w} \, e^{- \frac{\pi i kc (z|z)}{c \tau + d}} \ F \left( \frac{a \tau + b }{c \tau + d}, \ \frac{z}{c \tau + d}\right).
\end{equation}

The following example is especially important for this and the next sections.
\begin{example}
\label{ex1.1}
Let $ \fh $ be a 2-dimensional vector space over $ \CC $ with the basis $ z_1, z_2 $ and the symmetric bilinear form defined by 
$ (z_i  |  z_i) = 0, \ i = 1,2, \ (z_1  |  z_2) = 2z_1z_2$.  Let 
$  w =1  $ and $  k =m. $ Then the action (\ref{eq1.17}) becomes:
\begin{equation}
\label{eq1.18}
F {{1,m \!\!\!\!\!} \atop \bigg{|}}_{\left( \begin{matrix}
a & b \\
c & d \\
\end{matrix}
\right)}  (\tau, z_1, z_2) = (c \tau + d)^{-1} \, e^{- \frac{2\pi i mc z_1 z_2}{c \tau + d}} \ F \left( \frac{a \tau + b }{c \tau + d}, \ \frac{z_1}{c \tau + d}, \ \frac{z_2}{c \tau + d}\right).
\end{equation}
\end{example}
\begin{theorem}
\label{th1.1}
Let $ m \in \ZZ_{>0} $ and $s, s_1 \in \ZZ$.
\begin{enumerate}
\item[(a)] One has (cf. (\ref{eq1.18})) 
 \[ \tilde{\Phi}^{[m;s]} \left(-\frac{1}{\tau}, \frac{z_1}{\tau}, \frac{z_2}{\tau}\right) = \tau e^{\frac{2 \pi i m}{\tau} z_1 z_2} \  \tilde{\Phi}^{[m;s_1]} (\tau, z_1, z_2), \]
\[ \tilde{\Phi}^{[m;s]} (\tau + 1, z_1, z_2) = \tilde{\Phi}^{[m;s]} (\tau, z_1, z_2).\]
\item[(b)] If $ a, b \in \ZZ, $ then
\[ \tilde{\Phi}^{[m;s]} (\tau, z_1 + a \tau, z_2 + b \tau) = e^{-2 \pi i m (bz_1 + az_2)} \, q^{-mab} \, \tilde{\Phi}^{[m;s]} (\tau, z_1, z_2), \]
\[ \tilde{\Phi}^{[m;s]} (\tau, z_1 +a, z_2 +b) = \tilde{\Phi}^{[m;s]} (\tau, z_1, z_2). \]
\end{enumerate}
\end{theorem}

In the case $s=s_1$, Theorem \ref{th1.1}(a) coincides with 
Theorem 1.12(1),(2) in 
\cite{KW7}. Also, Theorem \ref{th1.1}(b) coincides with Theorem 1.12(3),(4) 
from \cite{KW7}. Note that notation in \cite{KW7} is slightly different: in the present paper we denote $ \Phi_1^{[m;s]} $ from \cite{KW7} by $ \Phi^{[m-1;-s]}$. The proof of Theorem \ref{th1.1}(a) for distinct $s$ and $s_1$ 
is along the 
same lines as that of Theorem \ref{th1.3}(a) below. 

Theorem \ref{th1.1}(a) implies the following corollary.
\begin{corollary} 
\label{cor1.2}
Let $ m \in \ZZ_{>0}$ and  $s, s_1 \in \ZZ$. Then 
$\tilde{\Phi}^{[m;s]} = \tilde{\Phi}^{[m;s_1]}$, 
and these functions are fixed by the action 
(\ref{eq1.18}) of the group $SL_2(\ZZ)$.
\end{corollary}
Due to Corollary \ref{cor1.2}, for $ m \in \ZZ_{>0}, \,s \in \ZZ,  $ the function $ \tilde{\Phi}^{ [m;s]} $ is independent on $ s, $ hence we shall denote it by $ \tilde{\Phi}^{[m]}$. 

\begin{theorem}
\label{th1.3} 
Let $ m \in \tfrac{1}{2} \ZZ_{>0}$. 
\begin{enumerate}
\item[(a)] If 
$s, s_1 \in \ZZ$ and $s^{\prime}, s^{\prime}_1 \in \tfrac{1}{2} + \ZZ$, then
\[ \tilde{\Phi}^{+[m;s]} \left(-\frac{1}{\tau}, \frac{z_1}{\tau}, \frac{z_2}{\tau}\right) = \tau e^{\frac{2 \pi i m}{\tau } z_1 z_2} \, \tilde{\Phi}^{+[m; s_1]} (\tau, z_1, z_2), \]
\[ \tilde{\Phi}^{-[m;s]} \left(-\frac{1}{\tau}, \frac{z_1}{\tau}, \frac{z_2}{\tau}\right) = \tau e ^{\frac{2 \pi i m}{\tau } z_1 z_2} \, \tilde{\Phi}^{+[m; s^{\prime}]} (\tau, z_1, z_2), \]
\[ \tilde{\Phi}^{+[m;s^{\prime}]} \left(-\frac{1}{\tau}, \frac{z_1}{\tau}, \frac{z_2}{\tau}\right) = \tau e^{\frac{2 \pi i m}{\tau }z_1z_2} \, \tilde{\Phi}^{-[m; s]} (\tau, z_1, z_2), \]
\[ \tilde{\Phi}^{-[m;s^{\prime}]} \left(-\frac{1}{\tau}, \frac{z_1}{\tau}, \frac{z_2}{\tau}\right) = \tau e^{\frac{2 \pi i m}{\tau }z_1z_2} \, \tilde{\Phi}^{-[m; s^{\prime}_1]} (\tau, z_1, z_2). \]
\item[(b)]  If $s, s^{\prime} \in \tfrac{1}{2} \ZZ  $ are such that $ m + s \in \ZZ, \, m +s^{\prime} \in \tfrac{1}{2} + \ZZ,  $ then
\[ \tilde{\Phi}^{\pm [m;s]}(\tau +1, z_1, z_2) = \tilde{\Phi}^{\pm [m;s+\half]}(\tau, z_1, z_2), \]
\[  \tilde{\Phi}^{\pm [m;s^{\prime}]} (\tau +1, z_1, z_2) = \tilde{\Phi}^{\mp [m;s^{\prime}]} (\tau, z_1, z_2). \]
\item[(c)] If $s \in \tfrac{1}{2} \ZZ, $ and $ a, b \in \ZZ $ have the same parity, or if $m\in \ZZ_{>0}$ and $a,b \in \ZZ$, then
\[ \tilde{\Phi}^{\pm [m;s]} (\tau, z_1 + a \tau, z_2 + b \tau) = (\pm 1)^a \, e^{-2 \pi i m (bz_1 + az_2)} \, q^{-mab} \, \tilde{\Phi}^{\pm [m;s]} (\tau, z_1, z_2), \]
\[ \tilde{\Phi}^{\pm [m;s]} (\tau, z_1 +a, z_2 +b) = e^{2 \pi i s a } \, \tilde{\Phi}^{\pm [m;s]} (\tau, z_1, z_2).\]
\item[(d)] If $ m $ is not an integer, $s \in \tfrac{1}{2} \ZZ$, and $a, b \in \ZZ$ have different parity, then
\[ \tilde{\Phi}^{\pm [m;s]} (\tau, z_1 + a \tau, z_2 + b \tau) = (\pm 1)^a \, e^{-2 \pi i m (bz_1 + az_2)} \, q^{-mab} \, \tilde{\Phi}^
{\pm [m;s+\half]} (\tau, z_1, z_2), \]
\[ \tilde{\Phi}^{\pm [m;s]} (\tau, z_1 +a, z_2 +b) = e^{2 \pi i s a } \tilde{\Phi}^{\mp [m;s]} (\tau, z_1, z_2). \]
\end{enumerate}
\end{theorem}

A special case of Theorem \ref{th1.3} is Theorem 4.14 in \cite{KW7} for the 
functions $ \tilde{\Phi}^{[B;m]}$, which coincide with the functions
$ \tilde{\Phi}^{-[m + \tfrac{1}{2}, \tfrac{1}{2}]}$ of the present paper.
A proof of Theorem \ref{th1.3}(a) is given in the next Section. The proof
of claims (b), (c) and (d) of this theorem is then straightforward, cf. the proof of Theorem 1.11(2),(3) in \cite{KW7}.
\begin{corollary} 
\label{cor1.4}
\begin{enumerate}
\item[(a)]  
Let $ m \in \tfrac{1}{2} \ZZ_{>0}$, and let $s, s_1 \in \tfrac{1}{2} \ZZ $ be 
such that $s-s_1\in \ZZ$. Then $\tilde{\Phi}^{\pm [m;s]} = \tilde{\Phi}^
{\pm[m;s_1]}.$ 
\item[(b)] For each  $ m \in \frac{1}{2} \ZZ_{>0} $ the four functions $ \tilde{\Phi}^{+[m;s]}, \ \tilde{\Phi}^{+[m;s^{\prime}]}, \ \tilde{\Phi}^{-[m;s]}, \ \tilde{\Phi}^{-[m;s^{\prime}]},  $ where $ s \in \ZZ,\, s^{\prime} \in \frac{1}{2} + \ZZ, $ are permuted by the action (\ref{eq1.18}) of $ SL_2(\ZZ).  $ 
The functions $ \tilde{\Phi}^{+[m;s]} $ (resp. $ \tilde{\Phi}^{-[m;s^{\prime}]} $) are fixed by $ SL_2 (\ZZ) $ if $ m \in \ZZ_{>0} $ (resp. $ \in \half + \ZZ_{\geq 0} $). If $ m \in \zp, $ then 
$\left( \begin{matrix}
1 & 1 \\
0 & 1 \\
\end{matrix} \right) $ 
fixes $ \tilde{\Phi}^{-[m;s]} $ and $ S = \left( \begin{matrix}
0 & -1 \\
1 & 0 \\
\end{matrix}
\right) $ 
fixes $ \tilde{\Phi}^{-[m; s^{\prime}]}. $ If $ m \in \half + \ZZ_{\geq 0}, $ 
then 
$\left( \begin{matrix}
1 & 1 \\
0 & 1 \\
\end{matrix} \right) $ fixes $ \tilde\Phi^{+[m;s^{\prime}]} $ and $ S $ fixes $ \tilde{\Phi}^{+[m;s]}. $ 
The remaining functions are transposed by these transformations. 
\end{enumerate}
\end{corollary}

\begin{remark}
\label{Rem1.6}
Recall the Zwegers' function $ \mu, $  given by
\[ \vartheta_{11} (\tau, z_2) \, \mu (\tau, z_1, z_2) = \sum_{n \in \ZZ} (-1)^n \frac{q^{\half (n^2 +n)} e^{2 \pi i n z_2}}{1 - e^{2 \pi i z_1 q^n} }. \]
Note that the RHS coincides with the function $ \Phi^{-[\half; \half]} (\tau, z_1, z_2),$ see (\ref{eq1.10}).
Due to identity (5.16) from \cite{KW6}, we deduce the following identity:
\begin{equation}
\label{eq1.19}
\vartheta_{11} (\tau, z_1 + z_2) \, \Phi^{-[\half; \half]} (\tau, z_1, 2z_2-z_1) = \vartheta_{11}(\tau, z_2) \, \Phi^{[1]} (\tau, z_1, z_2),
\end{equation}
\noindent and from \cite{KW6}, (5.18), we deduce a similar identity for the modified functions:
\begin{equation}
\label{eq1.20}
\vartheta_{11} (\tau, z_1 +z_2) \, \tilde{\Phi}^{-[\half; \half]} (\tau, z_1, 2z_2 - z_1) = \vartheta_{11} (\tau, z_2) \, \tilde{\Phi}^{[1]} (\tau, z_1, z_2).
\end{equation}
\end{remark}

\section{Proof of Theorem \ref{th1.3}(a)}
The proof of Theorem \ref{th1.3}(a) is based on a series of lemmas, which are generalizations of the corresponding statements in \cite{KW7}. Since the proofs are similar, we are omitting them.
In all lemmas of this Section, $ m \in \half \zp $ and $ s, s^{\prime}, s_1, s_1^{\prime} \in \half \ZZ. $
\begin{lemma}
\label{lem2.1}
$  $
\begin{enumerate}
\item[(a)]  If $ s \in \ZZ, $ then
\begin{enumerate}
\item[(i)] $ \Res_{z_1 = n+j \tau} \Phi^{\pm [m;s]} = (\pm 1)^j \frac{-1}{2 \pi i } e^{-2 \pi i j m z_2} \ e^{2 \pi i j n m}, $
\item[(ii)] $ \Res_{z_1 = n+j \tau} \Phi^{\pm [m;s]} \big{|}_S = (\pm 1)^n \frac{-1}{2 \pi i } e^{-2 \pi i j m z_2} \ e^{2 \pi i j n m}. $
\end{enumerate}
\item[(b)] If $ s^{\prime} \in \half + \ZZ, $ then
\begin{enumerate}
\item[(i)] $ \Res_{z_1 = n + j \tau} \Phi^{+[m;s^{\prime}]} = (-1)^n \ \frac{-1}{2 \pi i} \ e^{-2 \pi i j m z_2} \, e^{2 \pi i j n m}, $ \\
$ \Res_{z_1 = n + j \tau} \Phi^{-[m;s^{\prime}]} = (-1)^{j+n} \ \frac{-1}{2 \pi i} \ e^{-2 \pi i j m z_2} \, e^{2 \pi i j n m},$
\item[(ii)] $ \Res_{z_1 = n + j \tau} \Phi^{+[m;s^{\prime}]} \big{|}_S = (-1)^j \ \frac{-1}{2 \pi i} \ e^{-2 \pi i j m z_2} \, e^{2 \pi i j n m},$ \\
$ \Res_{z_1 = n + j \tau}  \Phi^{-[m;s^{\prime}]} \big{|}_S = (-1)^{j+n} \ \frac{-1}{2 \pi i} \ e^{-2 \pi i j m z_2} \, e^{2 \pi i j n m}. $
\end{enumerate}
\end{enumerate}
\qed
\end{lemma}
\begin{corollary}
\label{cor2.2}
Let $ s, s_1 \in \ZZ, s^{\prime}, s^{\prime}_1 \in \half + \ZZ. $ Then the functions $ \Phi^{+ [m;s]} - \Phi^{+ [m;s_1]}\big{|}_S, \ \Phi^{- [m;s]} - \Phi^{+ [m;s^{\prime}]}\big{|}_S, \ \Phi^{+ [m;s^{\prime}]} - \Phi^{- [m;s]}\big{|}_S, \  $ and $ \Phi^{- [m;s^{\prime}]} - \Phi^{- [m;s_1^{\prime}]}\big{|}_S $ are holomorphic in the domain $ X_0  = \left\{ (\tau, z_1, z_2) \in \CC^3| \ \Im \tau > 0 \right\}. $
\end{corollary}
\begin{lemma}
\label{lem2.2}
$  $
\begin{enumerate}
\item[(a)] $ \Phi^{\pm [m;s]} (\tau, -z_1, -z_2) = \Phi^{\pm [m; 1-s]} (\tau, z_1, z_2). $
\item[(b)] If $ a, b \in \ZZ $ have the same parity, one has
\[ \Phi^{\pm [m;s]} (\tau, z_1 +a, z_2 + b) = e^{2 \pi i s a} \Phi^{\pm[m;s]} (\tau, z_1, z_2). \]
\item[(c)] $  $
\begin{enumerate}
\item[(i)] If $m$ is an integer and $ a, b \in \ZZ $, one has
\[ \Phi^{\pm [m;s]} (\tau, z_1 + a, z_2 + b) = e^{2 \pi i s a } \Phi^{\pm [m;s]} (\tau, z_1, z_2). \]
\item[(ii)] If $ m $ is not an integer and $ a, b \in \ZZ $ have different parity, one has
\[ \Phi^{\pm [m;s]} (\tau, z_1 + a, z_2 + b) = e^{2 \pi i sa} \Phi^{\mp[m;s]} (\tau, z_1, z_2). \]
\end{enumerate}
\end{enumerate}

\qed
\end{lemma}
\begin{lemma}
\label{lem2.3}
\begin{enumerate}
\item[(a)] $ \Phi^{\pm [m;s]} (\tau, z_1, z_2) - e^{4 \pi i m z_1} \Phi^{\pm [m;s]} (\tau, z_1, z_2 + 2 \tau )  $
\[= \sum_{0 \leq k < 2m \atop k \in \ZZ} e^{\pi i (k+s) (z_1 - z_2)} q^{- \frac{(k+s)^2}{4m}} \Theta^{\pm}_{k+s, m} (\tau, z_1 + z_2). \]
\item[(b)]$ \Phi^{\pm [m;s]} (\tau, z_1, z_2) - e^{-4 \pi i m z_2} \Phi^{\pm [m;s]} (\tau, z_1-2 \tau, z_2 )  $
\[= \sum_{0 \leq k < 2m \atop k \in \ZZ} e^{\pi i (k+s) (z_1 - z_2)} q^{- \frac{(k+s)^2}{4m}} \Theta^{\pm}_{k+s, m} (\tau, z_1 + z_2). \]
\end{enumerate}
\qed
\end{lemma}
\begin{lemma}
\label{lem2.4}
If $ j \in \ZZ,  $ then
\[ \Phi^{\pm [m;s]} (\tau, z_1 + j \tau, z_2 + jz) = (\pm 1)^j e^{- 2 \pi i j m (z_1 + z_2)} q^{-mj^2} \Phi^{\pm [m;s]} (\tau, z_1, z_2). \]
\qed
\end{lemma}

As in \cite{KW6}, \cite{KW7}, consider the following change of coordinates: 
\begin{equation}
\label{eq2.1}
 z_1 = v -u, \ z_2 = -v - u, \mbox{ i.e. } u = -\frac{z_1 + z_2}{2},  \ v = \frac{z_1 - z_2}{2},
\end{equation}
\noindent and let
\begin{equation}
\label{eq2.2}
 \varphi^{\pm [m;s]} (\tau, u, v) := \Phi^{\pm [m;s]} (\tau, z_1, z_2).
\end{equation}
\noindent Recall the action 
(\ref{eq1.18}) of $SL(2,\ZZ)$. Throughout this section we shall denote this 
action by $F\big{|}_A$ to simplify notation.  Note that in the new coordinates we have:
\begin{equation}
\label{eq2.3}
 \varphi \big{|}_S (\tau, u, v) = \frac{1}{\tau} e^{-\frac{2 \pi i m}{\tau} (u^2 - v^2)} \varphi \left( -\frac{1}{\tau}, \frac{u}{\tau}, \frac{v}{\tau}\right), \mbox{ where } S = \left( \begin{matrix}
0 & -1 \\
1 & 0 \\
\end{matrix}
\right). 
\end{equation}
\begin{lemma}
\label{lem2.5}
$  $
\begin{enumerate}
\item[(a)] $ \varphi^{\pm[m;s]} (\tau, -u, v) = \varphi^{\pm [m;s]} (\tau, u, v), $
\item[(b)] For $ a, b \in \ZZ $ one has
\[ \varphi^{\pm[m;s]} (\tau, u + a, v+ b) = e^{2 \pi i s (b-a)} \varphi^{\pm [m;s]} (\tau, u, v), \]
\item[(c)] $  $
\begin{enumerate}
\item[(i)] If $ m $ is an integer, and $ a, b \in \half \ZZ, $ such that $ a + b \in \ZZ, $ one has
\[ \varphi^{\pm [m;s]} (\tau, u+a, v+b) = e^{2 \pi i s (b-a)} \varphi^{\pm [m;s]} (\tau, u, v). \]
\item[(ii)] If $ m $ is not an integer, and $ a, b \in  \half + \ZZ,$ one has
\[ \varphi^{\pm [m;s]} (\tau, u+a, v+b) = e^{2 \pi i s (b-a)} \varphi^{\mp[m;s]} (\tau, u, v). \]
\qed
\end{enumerate}
\end{enumerate}
\end{lemma}
\begin{lemma}
\label{lem2.6}
$  $
\begin{enumerate}
\item[(a)] $ \varphi^{\pm [m;s]} (\tau, u, v) - e^{4 \pi i m (v-u)} \varphi^{\pm [m;s] }(\tau, u- \tau, v - \tau) $
\[ = \sum_{0 \leq j < 2m \atop j \in \ZZ} e^{2 \pi i (j + s) v} q^{-\frac{(j+s)^2}{4m}} \ \Theta^{\pm}_{-(j+s), m} (\tau, 2 u).\]
\item[(b)] $ \varphi^{\pm [m;s]} (\tau, u, v) - e^{4 \pi i m (v+u)} \varphi^{\pm [m;s]} (\tau, u + \tau, v - \tau). $
\[ = \sum_{0 \leq j < 2m \atop j \in \ZZ} e^{2 \pi i (j + s) v} q^{-\frac{(j+s)^2}{4m}} \ \Theta^{\pm}_{-(j+s), m} (\tau, 2 u). \]
\end{enumerate}
\qed
\end{lemma}
\begin{lemma}
\label{lem2.7}
$ \varphi^{\pm [m;s]} (\tau, u, v) - e^{8 \pi i m (v-\tau)} \varphi^{\pm [m;s]} (\tau, u, v - 2 \tau) $
 \[ = \sum_{0 \leq k < 4m \atop k \in \ZZ} e^{2 \pi i (k+s) v} q^{- \frac{(k+s)^2}{4m}} \ \Theta^{\pm}_{-(k+s), m} (\tau, 2u). \]
\qed
\end{lemma}
\begin{lemma}
\label{lem2.8A}
For $ s, s_1 \in \ZZ, $ we put
\[ G^{++
[m; s, s_1]}(\tau,u,v) := \varphi^{+[m;s]} (\tau, u, v) - \varphi^{+[m;s_1]} \big{|}_S (\tau, u, v).\]
\noindent Then
\begin{enumerate}
\item[(a)] $ G^{++[m; s, s_1]} (\tau, u, v) - e^{8 \pi i m (v - \tau)} G^{++ [m; s, s_1]} (\tau, u, v - 2 \tau)$
\[ = \sum_{k \in \ZZ \atop s \leq k < s + 4m} e^{2 \pi i k v} q^{- \frac{k^2}{4m}} \ \Theta^+_{-k,m} (\tau, 2u). \]
\item[(b)] $ G^{++[m; s, s_1]} (\tau, u, v+2) - G^{++[m;s, s_1]} (\tau, u , v)$
\[ = \frac{-i}{\sqrt{2m}} (-i \tau)^{-\half} \sum_{k \in \ZZ \atop s_1 \leq k < s_1 + 4m} \ \sum_{j \in \ZZ / 2m \ZZ} e^{\frac{\pi i j k }{m}} e^{\frac{2 \pi i m }{\tau} (v + \frac{k}{2m})^2} \Theta^+_{j,m} (\tau, 2u). \]
\item[(c)] $ G^{++[m; s, s_1]} $ is the unique holomorphic function in the domain $ X_0, $ satisfying conditions (a) and (b). 
\end{enumerate}
\qed
\end{lemma}
\begin{lemma}
\label{lem2.8B}
For $ s \in \ZZ, \ s^{\prime} \in \half + \ZZ, $ we put
\[ G^{-+ [m;s,s^{\prime}]} (\tau, u, v) :=  \varphi^{-[m;s]} (\tau, u, v) - \varphi^{+[m;s^{\prime}]} \big{|}_S (\tau, u, v).\]
\noindent Then
\begin{enumerate}
\item[(a)] $ G^{-+ [m;s,s^{\prime}]} (\tau, u, v) - e^{8 \pi i m (v-\tau)} G^{-+ [m; s, s^{\prime}]} (\tau, u, v - 2 \tau) $
\[ = \sum_{k \in \ZZ \atop s \leq k < s + 4m} e^{2 \pi i kv } q^{-\frac{k^2}{4m}} \ \Theta^-_{-k,m} (\tau, 2u). \]
\item[(b)] $ G^{-+ [m;s,s^{\prime}]} (\tau, u, v+2) - G^{-+[m;s,s^{\prime}]} (\tau, u, v)  $
\[ = \frac{-i}{\sqrt{2m}} (-i \tau)^{-\half} \sum_{k \in \half + \ZZ \atop s^{\prime} \leq k < s^{\prime}+ 4m} \sum_{j \in \ZZ / 2m \ZZ} e^{\frac{\pi i j k }{m}} e^{\frac{2 \pi i m}{\tau}(v + \frac{k}{2m})^2} \ \Theta^-_{j,m} (\tau, 2u). \]
\item[(c)] $ G^{-+ [m;s,s^{\prime}]}  $ is the unique holomorphic function in the domain $ X_0, $ satisfying conditions (a) and (b). 
\end{enumerate}
\qed
\end{lemma}
\begin{lemma}
\label{lem2.8C}
For $ s^{\prime}, s^{\prime}_1 \in \half + \ZZ, $ we put
\[ G^{--[m;s^{\prime}, s^{\prime}_1]} (\tau, u, v) := \varphi^{-[m;s^{\prime}]} (\tau, u, v) - \varphi^{-[m;s^{\prime}_1]} \big{|}_S (\tau, u, v). \]
\noindent Then
\begin{enumerate}
\item[(a)] $ G^{--[m;s^{\prime}, s^{\prime}_1]} (\tau, u, v) - e^{8 \pi i m (v- \tau)} G^{--[m; s^{\prime}, s^{\prime}_1]} (\tau, u, v -2 \tau) $
\[ = \sum_{k \in \half + \ZZ \atop s^{\prime} \leq k < s^{\prime} + 4m} e^{2 \pi i k v } q^{-\frac{k^2}{4m}} \ \Theta^-_{-k,m} (\tau, 2 u). \]
\item[(b)] $ G^{--[m;s^{\prime}, s^{\prime}_1]} (\tau, u, v+2) - G^{--[m;s^{\prime}, s^{\prime}_1]} (\tau, u, v) $
\[ = \frac{-i}{\sqrt{2m}} (-i \tau)^{-\half} \sum_{k \in \half + \ZZ \atop s^{\prime}_1 \leq k < s^{\prime}_1 +4m} \sum_{j \in \ZZ / 2m \ZZ} e^{\frac{\pi i j k }{m}} e^{\frac{2 \pi i m }{\tau}(v + \frac{k}{2m})^2} \ \Theta^-_{j+ \half, m} (\tau, 2 u).\]
\item[(c)] $ G^{--[m;s^{\prime}, s^{\prime}_1]} $ is the unique holomorphic function in the domain $ X_0, $ satisfying (a) and (b).
\end{enumerate}
\qed
\end{lemma}
\begin{lemma}
\label{lem2.9}
Let $ m \in \half \NN $ and $ j \in \half \ZZ. $ Let $ v = a \tau - b, $ where $ a, b \in \RR. $
\begin{enumerate}
\item[(a)] If $ j \in \ZZ, $ then
\begin{enumerate}
\item[(i)] $ (\frac{\partial}{\partial a} + \tau \frac{\partial}{\partial b}) \left( (-i \tau )^{-\half} e^{\frac{2 \pi i m}{\tau} v^2} \ R^+_{j,m} (-\frac{1}{\tau}, \frac{v}{\tau}) \right) $
\[  = 4i \sqrt{\frac{y}{2}} \ e^{-4 \pi m a^2 y} \sum_{k \in \ZZ / 2m \ZZ} e^{\frac{\pi i j k }{m}} \ \Theta^+_{k,m} (-\bar{\tau}, 2 \bar{v}), \]
\item[(ii)] $ (\frac{\partial}{\partial a} + \tau \frac{\partial}{\partial b}) \left( (-i \tau )^{-\half} e^{\frac{2 \pi i m}{\tau} v^2} \ R^-_{j,m} (-\frac{1}{\tau}, \frac{v}{\tau}) \right) $
\[ = 4i \sqrt{\frac{y}{2}} \ e^{-4 \pi m a^2 y} \sum_{k \in \ZZ / 2m \ZZ} e^{\frac{\pi i j }{m} (k + \half)} \ \Theta^+_{k+ \half,m} (-\bar{\tau}, 2 \bar{v}). \]
\end{enumerate}
\item[(b)] If $ j \in \half + \ZZ, $ then
\begin{enumerate}
\item[(i)] $ (\frac{\partial}{\partial a} + \tau \frac{\partial}{\partial b}) \left( (-i \tau)^{-\half} e^{\frac{2 \pi i m}{\tau} v^2} \ R^+_{j,m} (-\frac{1}{
\tau}, \frac{v}{\tau}) \right) $
\[  = 4i \sqrt{\frac{y}{2}} \ e^{-4 \pi m a^2 y} \sum_{k \in \ZZ / 2m \ZZ} e^{\frac{\pi i j k }{m}} \ \Theta^-_{k,m} (- \bar{\tau}, 2 \bar{v}), \]
\item[(ii)] $ (\frac{\partial}{\partial a} + \tau \frac{\partial}{\partial b}) \left( (-i \tau)^{-\half} e^{\frac{2 \pi i m}{\tau} v^2} \ R^-_{j,m} (-\frac{1}{\tau}, \frac{v}{\tau}) \right) $
\[ = 4i \sqrt{\frac{y}{2}} \ e^{-4 \pi m a^2 y} \sum_{k \in \ZZ / 2m \ZZ} e^{\frac{\pi i j}{m} (k + \half)} \ \Theta^-_{k + \half, m} (- \bar{\tau}, 2 \bar{v}). \]
\end{enumerate}
\end{enumerate}
\qed
\end{lemma}
\begin{lemma}
\label{lem2.10}
Let $j\in \frac{1}{2}\ZZ$. Then
\begin{enumerate}
\item[(a)] $ R^{\pm}_{j,m} (\tau, v+1) = (-1)^{2j} R^{\pm}_{j,m} (\tau, v). $
\item[(b)] $ R^{\pm}_{j,m} (\tau, v - \tau) = \pm e^{2 \pi i m (\tau - 2v)} \left\{ R^{\pm}_{j,m} (\tau, v) - 2e^{- \frac{\pi i \tau}{2m} j^2 + 2 \pi i j v} \right\} $, i.e.
\begin{enumerate}
\item[(i)] $ R^{+}_{j,m} (\tau, v) - e^{2 \pi i m (2 v - \tau)} R^+_{j,m} (\tau, v - \tau) = 2q^{-\frac{j^2}{4m}} e^{2 \pi i j v}, $
\item[(ii)] $ R^{-}_{j,m} (\tau, v) + e^{2 \pi i m (2v- \tau)} R^-_{j,m} (\tau, v - \tau) = 2q^{-\frac{j^2}{4m}} e^{2 \pi i j v}. $
\end{enumerate}
\item[(c)] $ R^{\pm}_{j,m} (\tau,v ) - e^{8 \pi i m (v - \tau)} \, R^{\pm}_{j,m} (\tau, v - 2 \tau) =2 \left\{q^{-\frac{j^2}{4m}} e^{2 \pi i j v} \pm q^{- \frac{(j+2m)^2}{4m}} e^{2 \pi i (j+2m)v}\right\}. $
\item[(d)] $ R^{\pm}_{j,m} (- \frac{1}{\tau}, \frac{v+2}{\tau}) - e^{-\frac{8 \pi i m }{\tau}(v+1)} R^{\pm}_{j,m} \left( -\frac{1}{\tau}, \frac{v}{\tau}\right)  = -2 e^{-\frac{2 \pi i m}{\tau} (v+2)^2} \left\{ e^{\frac{2 \pi i m }{\tau} (v + \frac{j}{2m})^2} \pm e^{\frac{2 \pi i m}{\tau} (v + \frac{j + 2m}{2m})^2} \right\}. $
\end{enumerate} 
\qed
\end{lemma}

Let 
 \[ \varphi^{\pm[m;s]}_{ \add} (\tau, u, v):= -\half \sum_{k \in s + \ZZ \atop s \leq k < s + 2m} R^{\pm}_{k,m} (\tau, v) \ \Theta^{\pm}_{-k,m} (\tau, 2 u). \]
\begin{lemma}
\label{lem2.11}
$  $
\begin{enumerate}
\item[(a)] If $ a, b \in \ZZ, $ then
\[ \varphi^{\pm[m;s]}_{ \add} (\tau, u + a, v+b) = e^{2 \pi i s (b-a)} \, \varphi^{\pm[m;s]}_{ \add} (\tau, u, v). \]
\item[(b)] If $ m $ is an integer, and $ a,b \in \half \ZZ $ are such that $ a + b \in \ZZ, $ then
\[ \varphi^{\pm[m;s]}_{ \add} 
(\tau, u+a, v+b) = e^{2 \pi i s (b-a)} \, \varphi^{\pm[m;s]}_{ \add} (\tau, u, v). \]
\item[(c)] If $ m $ is not an integer, and $ a,b \in \half + \ZZ, $ then
\[ \varphi^{\pm[m;s]}_{ \add} (\tau, u+a, v+b) = e^{2 \pi i s (b-a)} \, \varphi^{\mp[m;s]}_{ \add} (\tau, u, v). \]
\end{enumerate}
\qed
\end{lemma}
\begin{lemma}
\label{lem2.12}
$  $
\[ \varphi^{\pm[m;s]}_{ \add} (\tau, u, v) - e^{8 \pi i m (v-\tau)} \, \varphi^{\pm[m;s]}_{ \add} (\tau, u, v - 2 \tau)  = - \!\!\!\!\! \sum_{k \in s + \ZZ \atop s \leq k < s + 2m} e^{2 \pi i k v} q^{-\frac{k^2}{4m}} \ \Theta^{\pm}_{-k,m} (\tau, 2u). \]
\qed
\end{lemma}

Next, as in \cite{KW7}, in order to study the functions
\[ G^{++[m; s, s_1]}_{\add} := \varphi^{+[m;s]}_{ \add}- \varphi^{+[m;s_1]}_{ \add}\big{|}_S, \]
\noindent where $ s, s_1 \in \ZZ, $ we consider the functions
\[ a_j^{++[s_1]} (\tau, v) := -R^+_{j;m} (\tau, v) -\frac{i}{\sqrt{2m}} (-i \tau)^{-\half} e^{\frac{2 \pi i m}{\tau} v^2} \sum_{k \in \ZZ \atop s_1 \leq k < s_1 + 2m} \, e^{-\frac{\pi i j k}{m}} R^+_{k,m} \left( -\frac{1}{\tau}, \frac{v}{\tau}\right), j \in \ZZ. \]
\begin{lemma}
\label{lem2.13A}
$  $
\begin{enumerate}
\item[(a)] $ a_j^{++[s_1]} (\tau, v) - e^{8 \pi i m (v-\tau)} a_j^{++[s_1]} 
(\tau, v - 2 \tau)    
= -2 \left\{ q^{- \frac{j^2}{4m}} e^{2 \pi
 i jv } + q^{-\frac{(j+2m)^2}{4m}} e^{2 \pi i (j+2m)v} \right\}.$
\item[(b)] $ a_j^{++[s_1]} (\tau, v+2) - a^{++[s_1]}_j (\tau, v) $
\[ = \frac{2i}{\sqrt{2m}} (-i \tau)^{-\half} \!\!\!\! \sum_{k \in \ZZ \atop s_1 \leq k < s_1 + 2m} e^{- \frac{\pi i j k }{m}}  \left\{ e^{\frac{2 \pi i m}{\tau}(v + \frac{k}{2m})^2} + e^{\frac{2 \pi i m}{\tau}(v + \frac{k+2m}{2m})^2} 
\right\}. \]
\item[(c)] $ a_j^{++[s_1]} (\tau, v) $ is holomorphic in $ v. $
\end{enumerate}
\qed
\end{lemma}
\begin{lemma}
\label{lem2.14A}
$  $
\begin{enumerate}
\item[(a)] $ $
\[  G^{++[m; s, s_1]}_{\add} (\tau, u, v) = \half \sum_{j \in \ZZ \atop s \leq j < s + 2m} a^{++[s_1]}_j (\tau, v) \ \Theta^+_{-j,m} (\tau, 2u).  \]
\item[(b)] $ G^{++[m; s, s_1]}_{\add} $ is holomorphic in the domain $ X_0. $
\end{enumerate}
\qed
\end{lemma}
\begin{lemma}
\label{lem2.15A}
\begin{enumerate}
$  $
\item[(a)]$ G^{++[m; s, s_1]}_{\add} (\tau, u, v) - e^{8 \pi i m (v - \tau)} G^{++[m; s, s_1]}_{\add} (\tau, u, v-2\tau ) $
\[ = - \sum_{k \in \ZZ \atop s \leq j < s + 4m} e^{2 \pi i kv} q^{-\frac{k^2}{4m}} \ \Theta^+_{-k,m} (\tau, 2u).  \]
\item[(b)] $ G^{++[m; s, s_1]}_{\add} (\tau, u, v+2)- G^{++[m; s, s_1]}_{\add} (\tau, u, v) $
\[ = \frac{i}{\sqrt{2m}} (-i \tau)^{- \half} \sum_{k \in \ZZ \atop s_1 \leq k < s_1 + 4m} \sum_{j \in \ZZ / 2 m \ZZ} e^{\frac{\pi i j k }{m}} e^{\frac{2 \pi i m }{\tau}(v + \frac{k}{2m})^2} \ \Theta^+_{j,m} (\tau, 2u). \]
\end{enumerate}
\qed
\end{lemma}
By Lemma \ref{lem2.8A} and Lemma \ref{lem2.15A}, we have
\begin{proposition}
\label{prop2.1A}
\[ G^{++[m; s, s_1]}_{\add} = - G^{++[m; s, s_1]}. \]
\qed
\end{proposition}
Then, putting $ \tilde{\varphi}^{+[m;s]} := \varphi^{+[m;s]} + \varphi^{+[m;s]}_{\add}, $ we obtain
\begin{corollary}
\label{cor2.1A}
$  $
\begin{enumerate}
\item[(a)] $ \tilde{\varphi}^{+[m;s]} = \tilde{\varphi}^{+[m;s_1]} \big{|}_S,$ i.e. $\tilde{\varphi}^{+[m;s]} (\tau, u, v) = \frac{1}{\tau} e^{-\frac{2 \pi i m}{\tau}(u^2 - v^2)} \tilde{\varphi}^{+[m;s_1]} \left( -\frac{1}{\tau}, \frac{u}{\tau}, \frac{v}{\tau}\right).$
\item[(b)] $ \tilde{\varphi}^{+[m;s]} $ does not depend on the choice of $ s \in \ZZ. $
\end{enumerate}
\qed
\end{corollary}

The remaining cases are treated in a similar fashion. 
For $ s \in \ZZ, s^{\prime} \in \half + \ZZ, j \in \ZZ,  $ we put 
\[ G^{-+[m; s,s^{\prime}]}_{\add}:= \varphi^{-[m;s]}_{\add} - \varphi^{+[m;s^{\prime}]}_{\add} \big{|}_S, \]
$  a_j^{-+[s^{\prime}]} (\tau,v ) := -R^-_{j,m} (\tau, v)
 -\frac{i}{\sqrt{2m}} (-i \tau)^{-\half} e^{\frac{2 \pi i m}{\tau} v^2} \sum_{k \in \half + \ZZ \atop s^{\prime} \leq k < s^{\prime}+2m} e^{-\frac{\pi i j k}{m}} R^+_{k;m} \left(- \frac{1}{\tau}, \frac{v}{\tau}\right). $
\begin{lemma}
\label{lem2.13B}
$  $
\begin{enumerate}
\item[(a)] $ a_j^{-+[s^{\prime}]} (\tau, v) - e^{8 \pi i m (v-\tau)} a_j^{-+[s^{\prime}]} (\tau, v-2 \tau)  = -2 \left\{ q^{- \frac{j^2}{4m}} e^{2 \pi i j v} - q^{-\frac{(j + 2m)^2}{4m}} e^{2 \pi i (j + 2m) v}
\right\}. $
\item[(b)] $ a_j^{-+[s^{\prime}]} (\tau, v+2)- a_j^{-+[s^{\prime}]} (\tau, v) $
\[ = \frac{2i}{\sqrt{2m}} (-i \tau)^{-\half} \sum_{k \in \half + \ZZ \atop s^{\prime} \leq k < s^{\prime} +2m  } e^{- \frac{\pi i j k }{m}} \left\{ e^{\frac{2 \pi i m}{\tau} (v + \frac{k}{2m})^2} + e^{\frac{2 \pi i m }{\tau} (v + \frac{k + 2m}{2m})^2}
\right\}. \]
\item[(c)] $ a_j^{-+[s^{\prime}]} (\tau, v) $ is holomorphic in $ v. $
\end{enumerate}
\qed
\end{lemma}
\begin{lemma}
\label{lem2.14B}
$  $
\begin{enumerate}
\item[(a)] $  $
 \[ G^{-+[m;s,s^{\prime}]}_{\add} (\tau, u , v) = \half \sum_{j \in \ZZ \atop s \leq j < s + 2m} a_j^{-+[s^{\prime}]} (\tau, v) \Theta^-_{-j,m} (\tau, 2u).  \]
\item[(b)] $ G^{-+[m;s,s^{\prime}]}_{\add} $ is holomorphic in the domain $ X_0. $
\qed
\end{enumerate}
\end{lemma}
\begin{lemma}
\label{lem2.15B}
$  $
\begin{enumerate}
\item[(a)] $ G^{-+[m;s,s^{\prime}]}_{ \add} (\tau, u, v) - e^{8 \pi i m (v - \tau)} G^{-+[m;s,s^{\prime}]}_{ \add} (\tau, u, v - 2 \tau) $
\[ = - \sum_{k \in \ZZ \atop s \leq k < s + 4m} e^{2 \pi i k v} q^{-\frac{k^2}{4m}} \ \Theta^-_{-k,m} (\tau, 2 u). \]
\item[(b)] $ G^{-+[m;s,s^{\prime}]}_{ \add} (\tau, u, v +2) - G^{-+[m;s,s^{\prime}]}_{ \add} (\tau, u ,v ) $
\[ = \frac{i}{\sqrt{2m}} (-i \tau)^{-\half} \sum_{k \in \half + \ZZ \atop \spr \leq k < \spr + 4m} \sum_{j \in \ZZ / 2m \ZZ} e^{\frac{\pi i j k }{m}} e^{\frac{2 \pi i m}{\tau}(v + \frac{k}{2m})^2} \ \Theta^-_{j,m} (\tau, 2u). \]
\end{enumerate}
\qed
\end{lemma}
By Lemma \ref{lem2.8B} and Lemma \ref{lem2.15B}, we have
\begin{proposition}
\label{prop2.1B}
$ G^{-+[m;s,s^{\prime}]}_{ \add} = - G^{-+[m;s,s^{\prime}]}. $
\qed
\end{proposition}
Then, putting
\[ \tilde{\varphi}^{-[m;s]} : = \varphi^{-[m;s]} + \varphi^{-[m;s]}_{\add}, \quad \tilde{\varphi}^{+[m; \spr]} := \varphi^{+[m;\spr]} + \varphi^{+[m;\spr]}_{\add} \]
\noindent we obtain
\begin{corollary}
\label{cor2.1B}
$  $
\begin{enumerate}
\item[(a)] $ \tilde{\varphi}^{-[m;s]} = \tilde{\varphi}^{+[m;s^{\prime}]} \big{|}_S \, . $
\item[(b)] $ \tilde{\varphi}^{-[m;s]}  $ (resp. $ \tilde{\varphi}^{+[m;s^{\prime}]} $) does not depend on the choise of $ s \in \ZZ $ (resp. $ s^{\prime} \in \half + \ZZ $).
\end{enumerate}
\qed
\end{corollary}

Finally, for $ s^{\prime}, s^{\prime}_1 \in \half + \ZZ $ and $ j \in \half + \ZZ, $ we put
\[ G^{--[m; \spr, \spr_1]}_{\add} : = \varphi^{-[m;\spr]}_{\add} - \varphi^{-[m;\spr_1]}_{\add} \big{|}_S,  \]
\[ a_j^{--[\spr_1]}(\tau,v) := -R^-_{j;m} (\tau, v) - \frac{i}{\sqrt{2m}} (-i \tau)^{-\half} e^{\frac{2 \pi i m}{\tau} v^2} \sum_{k \in \half + \ZZ \atop \spr_1 \leq k < \spr_1 + 2m} e^{-\frac{\pi i j k }{m}} R^-_{k;m} \left( -\frac{1}{\tau}, \frac{v}{\tau}\right). \]
\begin{lemma}
\label{lem2.13C}
$  $
\begin{enumerate}
\item[(a)]$ a_j^{--[\spr_1]} (\tau, v) - e^{8 \pi i m (v-\tau)} a_j^{--[\spr_1]} (\tau, v-2 \tau) $
\[ =-2 \left\{ q^{- \frac{j^2}{4m}} e^{2 \pi i j v} - q^{-\frac{(j+2m)^2}{4m}} e^{2 \pi i (j + 2m)v}
\right\}. \]
\item[(b)] $ a_j^{--[\spr_1]} (\tau, v+2) -  a_j^{--[\spr_1]} (\tau, v) $
\[  =\frac{2i}{\sqrt{2m}} (-i \tau)^{-\half} \sum_{k \in \half + \ZZ \atop \spr_1 \leq k < \spr_1 + 2m} e^{-\frac{\pi i j k }{m}} \left\{ e^{\frac{2 \pi i m }{\tau} (v + \frac{k}{2m})^2} + e^{\frac{2 \pi i m }{\tau}(v + \frac{k+2m}{2m})^2}
\right\}. \]
\item[(c)] $ a_j^{--[\spr_1]} (\tau, v) $ is holomorphic in $ v. $
\end{enumerate}
\qed
\end{lemma}
\begin{lemma}
\label{lem2.14C}
$  $
\begin{enumerate}
\item[(a)] $ $ 
\[ G^{--[m; \spr, \spr_1]}_{\add} (\tau, u, v) = \half  \sum_{j \in \half + \ZZ \atop \spr \leq j < \spr + 2m} a_j^{--[\spr_1]} (\tau, v) \ \Theta^-_{-j,m} (\tau, 2u).  \]
\item[(b)] $ G^{--[m; \spr, \spr_1]}_{\add} $ is holomorphic in the domain $ X_0. $
\end{enumerate}
\qed
\end{lemma}
\begin{lemma}
\label{lem2.15C}
$  $
\begin{enumerate}
\item[(a)] $ G^{--[m; \spr, \spr_1]}_{\add} (\tau, u, v) - e^{8 \pi i m (v-\tau)} G^{--[m; \spr, \spr_1]}_{\add} (\tau, u, v-2 \tau) $
\[ = -\sum_{k \in \half + \ZZ \atop \spr \leq k < \spr + 4m} e^{2 \pi i k v} q^{- \frac{k^2}{4m}} \ \Theta^-_{-k,m} (\tau, 2u).  \]
\item[(b)] $ G^{--[m; \spr, \spr_1]}_{\add} (\tau, u, v+2) -  G^{--[m; \spr, \spr_1]}_{\add} (\tau, u ,v) $
\[ =\frac{i}{\sqrt{2m}} (-i \tau)^{-\half} \sum_{k \in \half + \ZZ \atop \spr_1 \leq k < \spr_1 + 4m} \sum_{j \in \ZZ / 2m\ZZ} e^{\frac{\pi i k}{m}(j + \half)} e^{\frac{2 \pi i m}{\tau}(v + \frac{k}{2m})^2} \ \Theta^-_{j + \half, m} (\tau, 2u). \]
\end{enumerate}
\qed
\end{lemma}
By Lemma \ref{lem2.8C} and Lemma \ref{lem2.15C}, we have
\begin{proposition}
\label{prop2.1C}
$ G^{--[m; \spr, \spr_1]}_{\add} = - G^{--[m; \spr, \spr_1]} .$
\qed \end{proposition}
Hence, putting 
\[ \tilde{\varphi}^{-[m; \spr]} := \varphi^{-[m; \spr]} + \varphi^{-[m; \spr]}_{\add} ,\]
\noindent we obtain
\begin{corollary}
\label{cor2.1C}
$  $
\begin{enumerate}
\item[(a)] $ \tilde{\varphi}^{-[m;s^{\prime}]} = \tilde{\varphi}^{-[m;s^{\prime}_1]} \big{|}_S $
\item[(b)] $ \tilde{\varphi}^{-[m;s^{\prime}]} $ does not depend on the choice of $ s^{\prime} \in \half + \ZZ. $
\end{enumerate}
\qed \end{corollary}

Translating Corollaries \ref{cor2.1A}, \ref{cor2.1B}, \ref{cor2.1C} from $ \tilde{\varphi}^{\pm[m;s]} $ to $ \tilde{\Phi}^{\pm[m;s]}$ via 
(\ref{eq2.1}), (\ref{eq2.2}),
we obtain Theorem \ref{th1.3}(a).

\section{Modification procedure}
Unlike theta functions, mock theta functions do not have good modular and elliptic invariance properties.  In this section we introduce a modification process of a mock theta function $ \Theta^L_{\lambda, T}, $ which, using an $ n $ step passage from $ \Phi^{[a_i]} $ to $ \tilde{\Phi}^{[a_i]} $ for some $ a_i \in \ZZ_{>0},\, i = 1, \ldots, n = |T|, $ makes the obtained \textit{modified} mock theta function $ \tilde{\Theta}^L_{\lambda, T} $ to be the product of $ n $ functions $ \tilde{\Phi}^{[a_i]} $ and a theta function. Since the functions $ \tilde{\Phi}^{[a]},\, a \in \ZZ_{>0}, $ and the theta functions, have good modular and elliptic invariance properties, so does the function  $ \tilde{\Theta}^L_{\lambda, T} $.

Recall that $ L \subset \fh = \fh^{\ast} $ is a positive definite lattice of rank $ m > 0 $ and that $ T = \{\beta_1, \ldots, \beta_n \} \subset \fh = \fh^{\ast}$ is such that $ (\beta_i  |  \beta_j) = 0  $ for all $ i $ and $ j, $ and $ m \geq n. $ We shall assume that $ L $ has a $ \ZZ $-basis $ \gamma_1, \ldots, \gamma_m, $ such that the following two properties hold:
\begin{equation}
\label{eq3.1}
(\gamma_i  |  \beta_j) = -\delta_{ij}, \quad  i = 1, \ldots, m, \ j = 1, \ldots, n, 
\end{equation}
\begin{equation}
\label{eq3.2}
(\gamma_i^{\vee}| \gamma_j) \in \ZZ, \quad 1\leq i\leq n,\quad i\leq j \leq m.
\end{equation}
As usual, here and further, we let $ \gamma^\vee = 2 \gamma / (\gamma  |  \gamma). $

Let $ \lambda \in \fh^{\ast}  $ be of positive level $ k =\lambda (K)$, and assume that $ (\lambda | \beta_i) = 0,\, i = 1, \ldots, n. $
Assume, in addition, that
\begin{equation}
\label{eq3.3}
\frac{k}{2} \, | \gamma_i|^2 \in \ZZ_{>0},\,\,\, (\lambda  |  \gamma_i) \in \ZZ, \ i = 1, \dots, n.
\end{equation}
Introduce the following notation for $ s = 0, \dots, n: $
\begin{equation}
\label{eq3.4}
 L_s = \ZZ \{\gamma_i+\sum_{p=1}^s (\gamma_i|\gamma_p)\beta_p\,|\,s+1\leq i\leq m\}, 
\ T_s = \{\beta_{s+1}, \ldots, \beta_n \}, \ \lambda_s = \lambda + \sum_{i =1}^{s} (\lambda  |  \gamma_i) \beta_i,  
\end{equation}
\noindent so that $  T_0 = T,\ T_n = \emptyset,\ 
L_0=L\,,\lambda_0 = \lambda \, . $
Note that the lattice $L_s$ is isomorphic to the lattice $\ZZ\{\gamma_{s+1},...,\gamma_m\}$, hence it is positive definite.

It is a straightforward calculation, using only assumption (\ref{eq3.1}), to show that, if $ n \geq 1, $ the function $ \Theta^L_{\lambda, T} $ can be written in the following form:
\begin{equation}
\label{eq3.5}
\Theta^L_{\lambda, T} = \sum_{\gamma \in L_1} \ \left( t_{\gamma} \frac{e^{\lambda}}{\prod_{\beta \in T_1} (1-e^{-\beta})} \right) \ \Phi^{[\tfrac{k}{2} | \gamma_1 |^2; (\lambda| \gamma_1)]} 
(\tau,- \beta_1, \beta_1 + \gamma^\vee_1 + \tau (\gamma | \gamma^\vee_1 + 2 \beta_1)).
\end{equation}
The first modification step consists of replacing $ \Phi $ by $ \tilde{\Phi} $ in this formula. 
For this, we need assumptions (\ref{eq3.3}). This modification produces the function, which we denote by $ \tilde{\Theta}^{1;L}_{\lambda, T}. $ Due to assumptions (\ref{eq3.1}) and (\ref{eq3.2}), $ (\gamma  |  \gamma^\vee_1 +2 \beta_1) \in \ZZ, $ hence we may apply the elliptic transformation property, given by 
Theorem \ref{th1.1}(b), and Corollary \ref{cor1.2}, to obtain:
\begin{equation}
\label{eq3.6}
\tilde{\Theta}^{1; L}_{\lambda, T} =\Theta^{L_1}_{\lambda_1, T_1} \ \tilde{\Phi}^{[ \frac{k}{2} |\gamma_1|^2 ]} (\tau, -\beta_1, \beta_1 + \gamma^\vee_1).
\end{equation}

Similarly, if $ n \geq 2, $ we may apply the above modification step to the mock theta function $ \Theta^{L_1}_{\lambda_1, T_1} $ to obtain the second modified mock theta function (assumptions (\ref{eq3.3}) still hold for $ \lambda  $ replaced by $ \lambda_1 $):
\[ \tilde{\Theta}^{2; L}_{\lambda, T} = \Theta^{L_2}_{\lambda_2, T_2} \ \tilde{\Phi}^{[\frac{k}{2}| \gamma_1|^2]} (\tau, - \beta_1, \beta_1 + \gamma^\vee_1) \ \tilde{\Phi}^{[\frac{k}{2}| \gamma_2|^2]} (\tau, - \beta_2, \beta_2 + \gamma^\vee_2 + (\gamma^\vee_2 | \gamma_1) \beta_1).  \]

Repeating the same procedure $ n $ times, we arrive at the 
\textit
{ modified mock theta function}
\begin{equation}
\label{eq3.7}
\tilde{\Theta}^L_{\lambda, T} : = \tilde{\Theta}^{n; L}_{\lambda, T} = \Theta^{L_n}_{\lambda_n} \ \prod_{p = 1}^{n} \ \tilde{\Phi}^{[\frac{k}{2}| \gamma_{p}|^2]} (\tau, - \beta_{p}, \beta_{p} + \gamma^\vee_{p} + \sum_{i = 1}^{p -1} (\gamma^\vee_{p} | \gamma_i) \beta_i),
\end{equation}
where $ \Theta^{L_n}_{\lambda_n} $ is an ordinary theta function.

A similar modification process can be developed for a signed mock theta 
function 
$ \Theta^{\pm;L}_{\lambda, T}= \Theta^{L}_{\lambda, T, \epsilon^\pm} $ 
for the following choice of $ \epsilon^\pm $:
\begin{equation}
\label{eq3.8}
\epsilon_k^\pm (\gamma) = 
(\pm 1)^{  \sum_{i =  1}^{n }  (\gamma | \beta_i) + 
k |\gamma + \sum_{i=1}^{n}  (\beta_i | \gamma) \gamma_i|^2}.
\end{equation}
Here, instead of the
 first formula in (\ref{eq3.3}) we shall assume
\begin{equation}
\label{eq3.9}
k|\gamma_i|^2 \in \ZZ_{>0},\, i=1,...,n.
\end{equation}
Due to this assumption, the exponent in (\ref{eq3.8}) is an integer. 

We have the following  ``signed'' analogue of (\ref{eq3.5}):
\[ \Theta^{\pm;L}_{\lambda, T} 
= \sum_{\gamma \in L_1} \ \epsilon_k^\pm (\gamma) \, t_{\gamma} \frac{e^{\lambda}}{\prod_{\beta \in T_1} (1-e^{- \beta})} \ \Phi^{\pm[\frac{k}{2}|\gamma|^2; (\lambda| \gamma_1)]} \ (\tau, -\beta_1, \beta_1 + \gamma^\vee_1 + \tau (\gamma | \gamma^\vee_1 + 2 \beta_1)). \]
The ``signed'' analogue of (\ref{eq3.6}) is 
\begin{equation}
\label{eq3.10}
\tilde{\Theta}^
{\pm;1, L}_{\lambda, T} = \Theta^{L_1}_{\lambda, T, \epsilon_k^\pm} \ \tilde{\Phi}^{\pm[\frac{k}{2}|\gamma_1|^2; (\lambda| \gamma_1)]} \ (\tau, -\beta_1, \beta_1 + \gamma^\vee_1), 
\end{equation}
\noindent and that of (\ref{eq3.7}), called the \textit{signed  modified
mock theta function}, is
\begin{equation}
\label{eq3.11}
\tilde{\Theta}^
{\pm; L}_{\lambda, T} := \tilde{\Theta}^{n, L}_{\lambda, T, \epsilon_k^\pm} = \Theta^{\epsilon^\pm;L_n}_{\lambda_n} \ \prod_{p = 1}^{n} \tilde{\Phi}^{\pm[\frac{k}{2}| \gamma_{p}|^2; (\lambda| \gamma_{p})]} \ (\tau, -\beta_{p} + \gamma^\vee_{p} + \sum_{i=1}^{p -1} (\gamma^\vee_{p} | \gamma_i) \beta_i),
\end{equation}
where 
$ \Theta^{\epsilon^\pm;L_n}_{\lambda_n} $ is a signed theta function with the sign 
$ \epsilon^\pm(\gamma) = (\pm 1)^{k | \gamma|^2},\,
\gamma\in L_n$.

\begin{example}
\label{ex3.1}
Let $ \fg  $ be a simple finite-dimensional Lie algebra with an invariant bilinear form $ \bl $, normalized by the condition that $ (\alpha  | \alpha) = 2 $ for a long root $ \alpha $. Let $ \bar{\fh} $ be a Cartan subalgebra of $ \fg, $ and let $ L \subset \bar{\fh} $ be the coroot lattice. It is positive definite even. Let $ \{  \gamma_1, \ldots, \gamma_m\} \subset L $ be a set of simple coroots; then condition (\ref{eq3.2}) is satisfied. Let $ \fh = \bar{\fh} \oplus \CC{\beta}, $ and extend $ \bl $ from $ \bar{\fh} $ to a symmetric non-degenerate bilinear form on $ \fh, $ also denoted by $ \bl, $ letting 
\[  (\beta  |  \beta) = 0, \  (\beta  |  \gamma_i)  = -\delta_{1i}, \ i=1,...,m.  \]
Given $ \lambda \in \hat{\fh}^{\ast} $ of positive integer level $ k (=\lambda (K)), $ and such that $ ( \lambda  |  \beta ) = 0, \ ( \lambda  |  L ) \subset \ZZ,$ we may consider the mock theta function $ \Theta_{\lambda, \{ \beta\}}, $ defined by (\ref{eq1.2}), of degree $ k $ and atypicality 1. Since conditions
(\ref{eq3.3}) hold, we have its (1-step) modification (we use here Corollary
\ref{cor1.2}):
\[ \tilde{\Theta}_{\lambda, \{ \beta\} } = \Theta^{L_1}_{\lambda} \tilde{\Phi}^{[\frac{k}{2} |\gamma_1|^2]} (\tau, - \beta, \beta + \gamma^\vee_1). \]
\end{example}

Next we derive modular transformation formulae for the modified mock theta functions (\ref{eq3.7}) and their signed analogues (\ref{eq3.11}).
We keep the above notation and the assumptions of this Section. In addition, we shall assume that 
\begin{equation}
\label{eq3.12}
  \dim \fh =  \rank L + |T| \ (= m+n). 
\end{equation}
Let $M = L_n, \  M^{\perp} = \left\{ h \in \fh |\,\, (h  | M) = 0 \right\}$,
and denote by $ h^{(1)} $ the orthogonal projection of $ h \in \fh  $ to 
$ \CC  M := \CC \otimes_{\ZZ} M $. Fix a positive number $ k, $  such that 
\begin{equation}
\label{eq3.13}
k (L  |  L) \subset \ZZ, 
\end{equation}
\noindent and let
\begin{equation}
\label{eq3.14}
L_{\mbox{even}} (\mbox{resp.}  L_{\mbox{ odd}}) = \left\{ \gamma \in L |
\ k(\gamma |
\gamma) \in 2 \ZZ \ (\mbox{resp.} \in 1 + 2 \ZZ) \right\}.
\end{equation}

\noindent Introduce the set of weights of level $ k: $
\[ P_k = \left\{ \la \in \hat{\fh}^{\ast} \ | \ \la (K) = k, \ \bar{\la} \in L^{\ast}  \right\}, \]
\noindent and let 
\[ P_{k,T} = \left\{ \la \in P_k \ | \ (\la  |  T) = 0   \right\}. \]
\begin{proposition}
\label{prop3.2}
Suppose that $ L = L_{\mathrm{even}}. $ Then for $ \la \in P_{k,T} $ we have:
\begin{enumerate}
\item[(a)] $ $
 \[ \tilde{\Theta}^L_{\la, T} (\tau, z, t) = 
\Theta^{M}_{\la_n} (\tau, z^{(1)}, t ) \ \prod_{p = 1}^{n} \tilde{\Phi}^{[\frac{k}{2} |\gamma_p|^2]} (\tau, \ -\beta_{p}, \ \beta_{\rho} + \gamma^\vee_{p} + 
\sum_{i = 1}^{p -1} (\gamma^\vee_{p}  |  \gamma_i) \beta_i ). \]
\item[(b)]
\[ \tilde{\Theta}^L_{\la, T } 
\left( - \frac{1}{\tau}, \frac{z}{\tau}, t - \frac{|z|^2}{2 \tau} \right) 
= i^n (-i \tau)^{\frac{\dim \fh}{2}} | M^{\ast} / k M |^{-\half} \sum_{\mu \in P_{k,T} \atop \mod{(kM + M^{\perp} + \CC \delta)}} e^{-\frac{2 \pi i}{k} (\bar{\la}  |  \bar{\mu})} \tilde{\Theta}^L_{\mu, T} (\tau, z,t);\]  
\[\tilde{\Theta}^L_{\la, T }(\tau +1, z, t) = e^{\frac{\pi i}{k} |\bar{\la}|^2} \ \tilde{\Theta}^L_{\la, T} (\tau, z, t).   \]
\end{enumerate}
\end{proposition}
\begin{proposition}
\label{prop3.3}
Suppose that $  \gamma_1, \ldots, \gamma_n \in L_{\mathrm{odd}}. $ 
Let $ \xi_0 \in \fh^{\ast} $ be an element,
 satisfying the following conditions:
\[ (\xi_0 | \gamma_i ) \in \ZZ\,\, (\mbox{resp. } \in \half +\ZZ) \mbox{ if } \gamma_i \in L_{\even} (\mbox{ resp.}\in L_ {\odd});\,\, (\xi_0 | T) = 0. \]
\noindent Then for $ \la \in P_{k,T} $ we have:
\begin{enumerate}
\item[(a)] $ $ 
\[ \tilde{\Theta}^{\pm; L}_{\la, T} (\tau, z, t) = \Theta^{\pm; L_n}_{\la_n} (\tau, z^{(1)}, t) \prod_{p =1}^{n} \tilde{\Phi}^{\pm [\frac{k}{2} |\gamma_{p}|^2; (\la | \gamma_{p})]} (  \tau, \ -\beta_{p}, \beta_{p} +\gamma_p^\vee + \sum_{i = 1}^{p -1} (\gamma^\vee_p | \gamma_i) \beta_i ),  \]
\noindent and the same formula with $ \la $ (resp. $ \la_n $) replaced by $ \la + \xi_0 $ (resp. $ \la_n + \xi_0 $).
\item[(b)] $  $
\[ \tilde{\Theta}^{+; L}_{\la, T}  \left( - \frac{1}{\tau}, \frac{z}{\tau}, t - \frac{|z|^2}{2 \tau}  \right) = i^n (-i \tau)^{\frac{\dim \fh}{2}} | M^{\ast} / k M|^{-\half} \!\!\!\!\!\! \sum_{\mu \in P_{k,T} \atop \mod{(kM+M^{\perp} + \CC \delta)}} \!\!\!\!\!\! e^{-\frac{2 \pi i}{k} (\bar{\la} | \bar{\mu})} \tilde{\Theta}^{+;L}_{\mu, T } (\tau, z, t); \]
\[ \tilde{\Theta}^{-; L}_{\la, T} \left( - \frac{1}{\tau}, \frac{z}{\tau}, t - \frac{|z|^2}{2 \tau}\right) = i^n (-i \tau)^{\frac{\dim \fh}{2}} | M^{\ast} / kM|^{-\half}  \!\!\!\!\!\! \!\!\!\!\!\! \sum_{\mu \in P_{k,T} \atop \mod{(kM+M^{\perp} + \CC \delta)}} \!\!\!\!\!\!\!\!\!\!\!\! e^{- \frac{2 \pi i}{k} (\bar{\la}| \bar{\mu}+ \xi_0)}  \tilde{\Theta}^{+; L}_{\mu + \xi_0, T} (\tau, z, t); \]
\[ \tilde{\Theta}^{+; L}_{\la+ \xi_0, T} \left( - \frac{1}{\tau}, \frac{z}{\tau}, t - \frac{|z|^2}{2 \tau}\right) = i^n (-i \tau)^{\frac{\dim \fh}{2}} | M^{\ast} / kM|^{-\half}  \!\!\!\!\!\! \!\!\!\!\!\! \sum_{\mu \in P_{k,T} \atop \mod{(kM+M^{\perp} + \CC \delta)}} \!\!\!\!\!\!\!\!\!\!\!\! e^{- \frac{2 \pi i}{k} (\bar{\la} + \xi_0| \bar{\mu})}  \tilde{\Theta}^{-; L}_{\mu , T} (\tau, z, t); \]
\[ \tilde{\Theta}^{-; L}_{\la + \xi_0, T} \left( - \frac{1}{\tau}, \frac{z}{\tau}, t - \frac{|z|^2}{2 \tau}\right) = i^n (-i \tau)^{\frac{\dim \fh}{2}} 
| M^{\ast} / kM|^{-\half} \!\!\!\!\!\! \!\!\!\!\!\!\!\! \sum_{\mu \in P_{k,T} \atop \mod{(kM+M^{\perp} + \CC \delta)}} \!\!\!\!\!\!\!\!\!\!\!\!\!\! e^{- \frac{2 \pi i}{k} (\bar{\la}+ \xi_0| \bar{\mu}+ \xi_0)}  \tilde{\Theta}^{-; L}_
{\mu +\xi_0, T} (\tau, z, t). \]
\item[(c)] $  $
\[ \tilde{\Theta}^{\pm; L}_{\la, T} (\tau + 1, z, t) =  e^{\frac{\pi i }{k} |\bar{\la}|^2} \tilde{\Theta}^{\mp; L}_{\la , T} (\tau, z, t) ; \]
\[ \tilde{\Theta}^{\pm; L}_{\la + \xi_0, T} (\tau + 1, z, t) = e^{\frac{\pi i }{k} |\bar{\la} + \xi_0|^2} \tilde{\Theta}^{\pm; L}_{\la + \xi_0, T} (\tau, z, t) .\]
\end{enumerate}
\end{proposition}

In order to prove these propositions, let 
\begin{equation}
\label{eq3.15}
\tilde{\gamma}_p = \gamma_p + \sum_{j =1}^{\min(n, p -1)} (\gamma_p  |  \gamma_j) \beta_j,\,\, p = 1, \ldots, m.
\end{equation}
\noindent Note that $ M = \ZZ \left\{ \tilde{\gamma}_{n+1}, \dots, \tilde{\gamma}_m \right\}, $ and that $\bar{\lambda}_n$ is the orthogonal projection of
$\bar{\lambda}$ on $M$ and $\lambda_n(K)=k$.
We have the following decomposition of $ \fh  $ in an orthogonal direct sum of subspaces, due to assumption (\ref{eq3.12}):
\begin{equation}
\label{eq3.16}
\fh = \CC M \oplus V_1 \oplus \ldots \oplus V_n,
\end{equation}
\noindent where 
\[ V_p = \CC\beta_p \oplus \CC \tilde{\gamma}_{p}, \, p=1,...,n. \] 
Note that the matrix of the restriction of the bilinear form $ \bl $ to $ V_p $ is $ \left( \begin{matrix}
0 & -1 \\
-1 & |\gamma_p|^2
\end{matrix} \right). $
With respect to (\ref{eq3.16}), we have the decomposition of $ h \in \fh: $ 
\begin{equation} 
\label{eq3.17}
 h = h^{(1)} + h^{(2)}_1 + \ldots + h^{(2)}_n. 
\end{equation}
Since the modified mock theta function $  \tilde{\Theta}^L_{\la, T} $ 
(given by \ref{eq3.7}), can be rewritten as 
\[  \tilde{\Theta}^L_{\la, T} = \Theta^M_{\la_n}  \prod_{p =1}^{n} \tilde{\Phi}^{[\frac{k}{2} |\gamma_p|^2]} (\tau, -\beta_p, \beta_p + \frac{2}{|\gamma_p|^2} \tilde{\gamma_p} ), \]
\noindent and similarly for the modified signed mock theta function $  \tilde{\Theta}^{\pm; L}_{\la, T} ,$ and since $\bar{\lambda}_n\in  \CC M$,    we can write
\begin{equation}
\label{eq3.18}
 \tilde{\Theta}^L_{\la, T} (\tau, z, t) = \Theta^M_{\la_n} (\tau, z^{(1)}, t) \prod_{p =1}^{n} \tilde{\Phi}^{[\frac{k}{2}|\gamma_p|^2]} (\tau, z^{(2)}_p),
\end{equation}
\noindent where we put for $ h \in V_p: $
\[ \tilde{\Phi}^{[\frac{k}{2} |\gamma_p|^2]} (\tau, h) = \tilde{\Phi}^{[\frac{k}{2} |\gamma_p|^2]} ( \tau, - \beta_p (h), \beta_p (h) + \frac{2}{|\gamma_p|^2} \tilde{\gamma}_p (h) ), \]
\noindent and similarly for the modified signed mock theta functions 
(\ref{eq3.11}). 

Now Proposition \ref{prop3.2} follows from Theorem \ref{th1.1}(a) and the modular transformation formulae for theta functions, given, e.g., by Proposition A2 from the Appendix to \cite{KW7}.

Similarly, the proof of Proposition \ref{prop3.3} follows from Theorem \ref{th1.3} (a), (b), the modular transformation properties for signed theta functions, given by Proposition A3 from \cite{KW7}, and the following lemma.
\begin{lemma}
\label{lem3.4}
If $ \gamma_1, \ldots, \gamma_n \in L_{\mathrm{odd}}, $ then we have the following simple formula for $ \epsilon^{-}_k $ defined by (\ref{eq3.8}):
\[ \epsilon^-_k (\gamma) = (-1)^{k |\gamma|^2},\,\, \gamma  \in L. \]
\end{lemma}
\begin{proof}
It is straightforward, using assumption (\ref{eq3.13}).
\end{proof}

In conclusion of this section we derive elliptic transformation formulae for the modified mock theta functions and their signed analogues. 
For each $ p = 1, \ldots, n $ consider the following sublattices of the spaces $ V_p: $
\[ M_p = \ZZ \frac{|\gamma_p|^2}{2} \beta_p + \ZZ \tilde{\gamma}_p, \qquad M^{\prime}_p = \ZZ |\gamma_p|^2 \beta_p + \ZZ \tilde{\gamma}_p. \]
\begin{lemma}
\label{lem3.5}
Let $ \frac{k}{2} |\gamma_p|^2 \in \half \zp$ and $ (\la | \gamma_p) \in \half \ZZ. $
\begin{enumerate}
\item[(a)] For $ v \in M^{\prime}_p, $ the following formulas hold:
\begin{enumerate}
\item[(i)] $ \tilde{\Phi}^{\pm [\frac{k}{2} |\gamma_p|^2; (\la|\gamma_p)]} (\tau, h+v) = e^{2 \pi i (\la|v)} \tilde{\Phi}^{\pm[ \frac{k}{2} |\gamma_p|^2; (\la | \gamma_p)]} (\tau, h);  $
\item[(ii)] $ \tilde{\Phi}^{\pm [\frac{k}{2} |\gamma_p|^2; (\la|\gamma_p)]} (\tau, h+ \tau v) = (\pm 1)^{(\beta_p | v)} e^{- 2 \pi i k (h|v)} q^{- \frac{k}{2} |v|^2} \tilde{\Phi}^{\pm [\frac{k}{2} |\gamma_p|^2; (\la | \gamma_p)]} (\tau, h). $
\end{enumerate}
\item[(b)] In the case $ \frac{k}{2} |\gamma_p|^2 \in \zp, $ the following formulas hold for $ v \in M_p: $
\begin{enumerate}
\item[(i)] $ \tilde{\Phi}^{\pm [\frac{k}{2} |\gamma_p|^2; (\la|\gamma_p)]} (\tau, h+v)  = e^{2 \pi i (\la | v)} \tilde{\Phi}^{\pm [\frac{k}{2} |\gamma_p|^2; (\la|\gamma_p)] } (\tau, h);$
\item[(ii)] $ \tilde{\Phi}^{\pm [\frac{k}{2} |\gamma_p|^2; (\la|\gamma_p)]} (\tau, h+ \tau v) = (\pm 1)^{(\beta_p | v)} e^{- 2 \pi i k (h|v)} q^{- \frac{k}{2} |v|^2} \tilde{\Phi}^{\pm [\frac{k}{2} |\gamma_p|^2; (\la | \gamma_p)]} (\tau, h). $
\end{enumerate}
\end{enumerate}
\end{lemma}
\begin{proof}
Note that for $ v \in V_p,\, p = 1, \dots, n $ we have:
\[ v \in M_p \mbox{ iff } (\beta_p  |  v),\, (\beta_p + \frac{2}{|\gamma_p|^2} \tilde{\gamma}_p | v) \in \ZZ; \]
\[ v \in M_p^{\prime} \mbox{ iff } (\beta_p  |  v), \,(\beta_p + \frac{2}{|\gamma_p|^2} \tilde{\gamma}_p | v ) \in \ZZ, \,(\frac{2}{|\gamma_p|^2} \tilde{\gamma}_p 
|v) \in 2 \ZZ .\]
Now the lemma follows from Theorem \ref{th1.3} (c), (d).
\end{proof}

\begin{lemma}
\label{lem3.6}
For $ v \in M, z \in M_{\CC},  $ and $ \la \in P_k $ one has: 
\[ \Theta^{\pm}_{\la_n + \xi_0} (\tau, z + v, t) = e^{2 \pi i (\la + \xi_0|v)} \Theta^{\pm}_{\la_n + \xi_0} (\tau, z, t);  \]
\[ \Theta^{\pm}_{\la_n + \xi_0} (\tau, z+\tau v, t) = e^{-2 \pi i k (z|v)} q^{-\frac{k}{2} |v|^2} \Theta^{\pm}_{\la_n + \xi_0} (\tau, z, t). \]
The same formulas hold if we replace $ \xi_0 $ by 0 (note that in this case
the factor $e^{2\pi i (\la |v)}$ is $1$).
\end{lemma}
\begin{proof}
The lemma follows from Proposition A4 from \cite{KW7} on elliptic transformations of signed theta functions, using that 
$(-1)^{k |v|^2} = e^{2 \pi i (\la + \xi_0  |  v) }$ for $v \in M.   $ 
\end{proof}

Due to Proposition \ref{prop3.3}(a), Lemmas \ref{lem3.5} and \ref{lem3.6} give elliptic transformation properties of modified signed mock theta functions (\ref{eq3.11}).
\begin{proposition}
\label{prop3.7}
Let $ \la \in P_{k,T} $ and assume that $ \gamma_1, \ldots, \gamma_n \in L_{\mathrm{odd}}. $ Then for 
$ v \in \tilde{M}^{\prime} := M \oplus M^{\prime}_1 \oplus \cdots \oplus M^{\prime}_n, $
the following formulas hold:
\begin{enumerate}
\item[(i)] $ \tilde{\Theta}^{\pm; L}_{\la + \xi_0, T} (\tau, z + v, t) = e^{2 \pi i (\la + \xi_0|v)} \tilde{\Theta}^{\pm; L}_{\la + \xi_0, T} (\tau, z, t);  $
\item[(ii)] $ \tilde{\Theta}^{\pm; L}_{\la + \xi_0, T} (\tau, z + \tau v, t) = (\pm 1)^{\sum_{p =1}^{n} (\beta_p|v)} e^{-2 \pi i k (z|v)} q^{-\frac{k}{2}|v|^2} \tilde{\Theta}^{\pm; L}_{\la + \xi_0, T} (\tau, z, t). $
\end{enumerate}
The same formulas hold if we replace $ \xi_0 $ by 0 (and remove the factor $ e^{2 \pi i (\la | v)} $ in (ii)). 
\qed
\end{proposition}
Similar arguments give elliptic transformation formulae of modified mock theta function (\ref{eq3.7}).
\begin{proposition}
\label{prop3.8}
In the case $ L = L_{\mathrm{even}}, $ the following formulas hold for $ \la \in P_{k, T} $ and $ v \in \tilde{M}
:= M \oplus M_1 \oplus \cdots \oplus M_n$:
\begin{enumerate}
\item[(i)] $ \tilde{\Theta}^L_{\la, T} (\tau, z + v, t) = \tilde{\Theta}^L_{\la, T} (\tau, z, t),  $
\item[(ii)] $ \tilde{\Theta}^L_{\la, T} (\tau, z + \tau v, t) = e^{-2 \pi i k (z|v)} q^{- \frac{k}{2} |v|^2} \tilde{\Theta}^L_{\la, T} (\tau, z, t). $
\end{enumerate}
\qed
\end{proposition}

\section{Normalized supercharacters of integrable highest weight $ \widehat{\fg}$-modules}

Let $ \fg $ be a basic simple finite-dimensional Lie superalgebra, which is not a Lie algebra. Recall that $ \fg $ is isomorphic to one of the following Lie superalgebras: $ s \ell (m+1 | n) $ with $ m \geq n \geq 1,\,  p s \ell (n|n)$ with $ n > 1,\, osp (m | n) $ with $ m \geq 1, n \geq 2 $ even, $ D(2,1;a) $, $F(4)$ and $G(3)$ \cite{K1}. ``Basic'' means that the even part $ \fg_{\bar{0}} $ is a reductive subalgebra of $ \fg $ and that $ \fg $ carries a (unique, up to a non-zero factor) non-degenerate invariant
bilinear form $ \bl, $ which is supersymmetric, i.e. its restriction to $ \fg_{\bar{0}} $ (resp. odd part $ \fg_{\bar{1}} $) is symmetric (resp. skewsymmetric) and $ (\fg_{\bar{0}}  |  \fg_{\bar{1}} ) = 0. $

The subalgebra $ \fg_{\bar{0}} $, and the dual Coxeter number $h^\vee$ of 
$\fg$, introduced below, are given in the following table:
\begin{center}
\textbf{Table 1}

\medskip 
\begin{tabular}{l|l|l}
\hline
$ \fg $ & $ \fg_{\bar{0}} $ & $ h^\vee $ \\
\hline
$ s \ell (m |n),\, m> n $ & $s \ell (m) \oplus s \ell (n)\oplus \CC $ & 
$ m-n $ \\
$ ps \ell (n|n) $ & $ s \ell (n) \oplus s \ell (n) $ & 0\\
$ osp (m|n),\, m\geq n+ 2 $ & $ so (m) \oplus sp(n) $ & $ m -2 -n $\\
$ osp (m|n),\, m \leq n+2$ & $ so(m) \oplus sp(n) $ & $ \thalf (n+2-m) $ \\
$ D(2, 1; a) $ & $ s \ell (2) \oplus s \ell (2) \oplus s \ell (2) $ & 0 \\
$ F(4) $ & $ so(7) \oplus s \ell (2) $ & 3\\
$ G(3) $ & $ G_2 \oplus s \ell (2) $ & 2\\
\hline
\end{tabular}
\end{center}

For $ \fg$ different from $D(2,1;a)$ and
$osp (n|n) $, we denote by $ \fg_{\bar{0}}^{\#} $ the simple
component
of $ \fg_{\bar{0}} $ of largest rank; in the case of $  \fg = osp (n|n)$ we let $  \fg_{\bar{0}}^{\#} = sp(n) $,    and
in the case $ \fg = D(2,1;a) $ we let $  \fg_{\bar{0}}^{\#} = s \ell (2) \oplus s \ell (2), $ suitably chosen. 

Let $ \fh^{\#} $ be a Cartan subalgebra of $  \fg_{\bar{0}}^{\#}, $ and let $ \fh $ be a Cartan subalgebra of $  \fg_{\bar{0}}, $ containing $ \fh^{\#}. $ The restriction of the bilinear form $ \bl $ to $ \fh^{\#} $ and $ \fh $ is non-degenerate, hence we may identify $ \fh^{\#} $ with $ \fh^{\# \ast} $ and $ \fh $ with $ \fh^{\ast}, $ using this form. 

We have the root space decomposition of $ \fg $ with respect to $ \fh: $
\[ \fg= \fh \oplus (\underset{\alpha \in \Delta}{\oplus}   \fg_{\alpha}),  \]
\noindent where $ \fg_{\alpha} = \{ a \in \fg \ | \ [h,a] = \alpha (h) a  \mbox{ for all } h \in \fh \}, $ and $ \Delta = \{ \alpha \in \fh^{\ast} \ | \ \alpha \neq 0, \fg_{\alpha} \neq 0    \} $ is the set of roots. Unless $ \fg = ps\ell (2|2), $ which we exclude from consideration, all root spaces $\fg_{\alpha} $ are 1-dimensional, and we let $ p(\alpha) = \bar{0} $ or $ \bar{1} \in \ZZ / 2 \ZZ, $ depending on whether $ \fg_{\alpha}$ lies in $\fg_{\bar{0}} $ or in
$ \fg_{\bar{1}}. $ Denote by $ \Delta_{\bar{0}} $ (resp. 
$ \Delta_{\bar{1}} $
) the set of even (resp. odd) roots,
i.e. $ p(\alpha) = \bar{0} $ (resp. $ = \bar{1} $).
The Weyl group $ W \subset \End \fh^{\ast} $ of $ \fg $ is the group, generated by reflections in $ \alpha \in \Delta_{\bar{0}}. $ 

Let $ \Delta^{\#}_{\bar{0}} = \{ \alpha \in \Delta_{\bar{0}} \ | \ \fg_{\alpha} \in \fg^{\#}_{\bar{0}}   \} $ be the set of roots of the Lie algebra $ \fg_{\bar{0}}^{\#}. $ Let $ L \subset \fh $ be the coroot lattice for the root system $ \Delta^{\#}_{\bar{0}} $, i.e. the $ \ZZ $-span of the set $ \{ \alpha^{\vee}
| \ \alpha \in \Delta^{\#}_{\bar{0}}\}$,  
and let $ W^{\#} \subset \End \fh^{\ast}$  be the group, generated by reflections in the $ \alpha \in \Delta^{\#}_{\bar{0}}$, so that $ W^{\#} $ is a subgroup of $W$ .  We normalize the bilinear form $ \bl $ by the condition (unless otherwise stated):
\begin{equation}
\label{eq4.1}
(\alpha  |  \alpha) = 2 \mbox{ for the longest root } \alpha \,\, 
\mbox{in} \, \Delta^{\#}_{\bar{0}}.
\end{equation}
Let $ h^\vee$
be the half of the eigenvalue of the Casimir element of $ \fg, $ associated to $ \bl, $ in the adjoint representation. The values of this number, called the \textit{dual Coxeter number}, are given in Table 1. 

\begin{remark}
\label{rem4.1}
\cite{KW2} Let $ \kappa (.  ,  . )$ be the Killing form of $ \fg. $ It is non-degenerate iff $ h^{\vee} \neq 0,$  and in this case $  \Delta^{\#}_{\bar{0}} = \{ \alpha \in \Delta_{\bar{0}} \ | \ \kappa (\alpha, \alpha) >0 \}. $
\end{remark}

Let $T\subset \Delta_{\bar{1}} $ be a maximal subset of linearly independent
isotropic pairwise 
orthogonal roots (all such subsets are $W$-conjugate \cite{KW2}). 
The number $|T|$  is called the {\it defect} of $\fg$.
Denote by $\fg^{!}$ the derived subalgebra of the centraliser of $T$ in $\fg$.
Then $\fh^{!}=\fh \cap \fg^{!}$ is a Cartan subalgebra of 
$\fg^{!}$.
Let $\Delta^{!}  \subset \fh^{!*}$ be the set of 
roots of $\fg^{!}$, let
$W^{!}$ be the Weyl group of $\fg^{!}$, 
and $L^{!}\subset \fh^{!}$ be the coroot lattice of the even part of $\fg^{!}$.
The list of all Lie superalgebras $\fg^{!}$ and the defects  
for all basic Lie superalgebras $\fg$
is given in the following table. 
\begin{center}
\textbf{Table 2}
\end{center}
\[ 
\begin{array}{l|l|l}
\hline
\rule[-0.1in]{0in}{0.3in} 
 \fg & \fg^! & \mbox{defect } \fg \\
\hline
\sl(m|n),\, m > n & \sl (m-n)& m \\
\mathrm{osp}  (m|n),\, m  \mbox{ even }  \leq n & \mathrm{sp} (n-m)& m /2 \\
\mathrm{osp}  (m|n),\, m  \mbox{ odd }  \leq n + 1 & \mathrm{osp} \, (1|n-m+1)& (m-1)/2\\
\mathrm{osp} \, (m|n),\,\, m > n+2 & \mathrm{so}(m-n)& n/2\\
\mathrm{osp} \, (n+2|n) & 0& n/2\\
 D(2, 1;a) & 0& 1\\
F(4) & \sl(3) & 1\\
G(3) & \sl (2)& 1\\
\hline
\end{array}
 \]
The Lie superalgebra $ \fg^! $ has the following properties, which are easily checked by direct verification. (For example, for claim (a) we use that the dual Coxeter number for $ \sl (n), \mathrm{so} (n), \mathrm{sp} (n) $ is equal to $ n, n-2, \frac{n}{2}+1 $ respectively \cite{K2}, Chapter 6.) 
\begin{proposition}
\label{prop4.2}
Let $ \fg $ be a basic finite-dimensional Lie superalgebra, such that 
$\fg^! \neq 0$  
(i.e. $\fg$ is 
differnt from $ \mathrm{osp} (n+2|n) $ and $ D(2,1;a) $). Then
\begin{enumerate}
\item[(a)] The dual Coxeter numbers of $ \fg $ and $ \fg^! $ are equal.
\item[(b)] $ W^{!} = \{ w \in W^{\#} |\, w (\fh^{!}) \subset \fh^{!}  \} $.
\item[(c)] The restriction of $ \bl $ from $ \fg  $ to $ \fg^! $ is the normalized  invariant bilinear form on $ \fg^!. $
\item[(d)] The coroot lattice $L^{!}$ of $\fg^{!}$ coincides with 
$L \cap \fh^!$  and is isomprphic to the lattice $M$ . 
\item[(e)] $ \thalf \sdim \fg / \fh + |T| = \thalf \sdim \fg^! / \fh^!.  $
\end{enumerate}
\qed
\end{proposition}

Choose a subset of positive roots $ \Delta_+ $ in $ \Delta, $ let $ \Delta_{\epsilon+} = \Delta_{\epsilon} \cap \Delta_+, \,\epsilon = \bar{0} $ or 
$ \bar{1}, $ and let $ \fn_+ = \underset{\alpha \in \Delta_+}{\oplus} \fg_{\alpha}.  $ 
Let $ \Pi = \{ \alpha_1, \ldots, \alpha_{\ell}    \}  \subset \Delta_+$ be the set of simple roots, and let $ \theta \in \Delta_+ $ be the highest root. Let $ \rho_{\epsilon} $ be the half of the sum of the elements from $ \Delta_{\epsilon+}, $ and let $ \rho = \rho_{\bar{0}} - \rho_{\bar{1}}. $ Then one has \cite{KW2}: 
\begin{equation}
\label{eq4.2}
2 (\rho | \alpha_i) = (\alpha_i | \alpha_i),\, i = 1, \ldots, \ell;\, h^{\vee} = (\rho + \half \theta | \theta); \,(\rho | \rho) = \frac{h^{\vee}}{12} \sdim \fg.
\end{equation}
\noindent Recall that $ \sdim \fg := \dim \fg_{\bar{0}} - \dim \fg_{\bar{1}}. $

Let $ \Delta^!_+ \subset \Delta^! $ be the subset of positive roots, corresponding to the linear function on $ \fh^{\ast}, $ defining $ \Delta_+, $ and let $ \rho \in \fh^{\ast !} $ be the element, as defined above, for $ \Delta^!_+. $

Let $ \widehat{\fg} = \fg [t, t^{-1}] \oplus \CC K \oplus \CC d $ be the affine Lie superalgebra, associated to $ \fg. $ The invariant bilinear form 
$ \bl $ 
on $\fg$  is extended to a non-degenerate invariant supersymmetric bilinear form $ \bl $ on $ \widehat{\fg} $, letting
\[ (at^m | bt^n) = \delta_{m, -n} (a | b),\, (\fg [t, t^{-1} ]| \CC K + \CC d) = 0, \,(K|K) = (d|d) = 0, 
(K|d) =1. \]

Let $ \widehat{\fh} = \fh + \CC K + \CC d $ be the Cartan subalgebra of $ \widehat{\fg}. $ The restriction of the bilinear form $ \bl $ to $ \widehat{\fh} $ is non-degenerate and symmetric. Note that this notation exactly corresponds to the one, introduced in the beginning of Section 1. In particular, we have $ \delta  $ and $ \Lambda_0$ in
$\widehat{\fh}^{\ast} $, corresponding to
$K$ and  $d$, and
we have translations $ t_{\gamma} \in \End \widehat{\fh}^{\ast} $ for $ \gamma \in \fh,  $ defined by (\ref{eq1.1}).
Note also that the action of the group $ W $ on $ \fh^{\ast} $ extends to $ \widehat{\fh}^{\ast}, $ keeping $ \delta $ and $ \Lambda_0 $ fixed. 

We have the root space decomposition of $ \widehat{\fg} $ with respect 
to $ \widehat{\fh}: $
\[ \widehat{\fg} = \widehat{\fh} 
\oplus (\underset{\alpha \in \widehat{\Delta}}{\oplus} \fg_{\alpha}),\, \fg_{\alpha} = \{ a \in \widehat{\fg} |\, [h,a] = \alpha (h) a \mbox{ for all } h \in \widehat{\fh}    \}, \]
where $\widehat{\Delta }=\{\alpha \in \widehat{\fh}^*|\, \alpha \neq 0,\,
\fg_\alpha \neq 0\}$ 
is the set of all affine roots, i.e. the roots of 
$\widehat{\fg}$. 
Note that $  \widehat{\Delta} = \widehat{\Delta}^{\re} 
\cup \widehat{\Delta}^{\im}, $ where 
\[ \widehat{\Delta}^{\re} = \{ \alpha + n \delta |\, \alpha \in \Delta, n \in \ZZ \},\, \widehat{\Delta}^{\im} = \{ n \delta |\, n \in \ZZ, n\neq 0 \}, \]
\noindent the multiplicity of $ \alpha \in \widehat{\Delta}^{\re} $ being 1, and of $ n \delta  $ being $ \ell. $  
Note that $ p (\alpha + n \delta) = p (\al). $

The affine Weyl group $ \widehat{W} \subset \End \widehat{\fh}^{\ast} $ is the group, generated by reflections in the $ \al \in \widehat{\Delta}^{\re}_{\bar{0}}. $ This group has two important homomorphisms $ \epsilon^{\pm}: \widehat{W} \rightarrow \{ \pm 1 \} .$ Namely, given a decomposition of $ w $ as a product of $ s_+ $ reflections in $ \widehat{W}, s_-$  of them being with respect roots whose half is not a root, we let 
$ \epsilon^{\pm} (w) = (-1)^{s_{\pm}}. $
Recall that $ \widehat{W} = W \ltimes \{ t_{\gamma} |\, \gamma \in \ZZ 
\mbox{-span of the set of coroots } $  $ \al^{\vee}, \mbox{ where } \al \in \Delta_{\bar{0}} \}$ \cite{K2}. The group $ \widehat{W} $ contains the important 
subgroup 
\begin{equation}
\label{eq4.3}
\widehat{W}^{\#}= W^{\#} \ltimes \{ t_{\gamma} |\, \gamma \in L \}, 
\end{equation}
which is generated by reflections with respect to roots of the form $ \al + n \delta, $ where $ \al \in \Delta^{\#}_{\bar{0}},\, n \in \ZZ. $
\begin{proposition}
\label{prop4.3}
\begin{enumerate}
\item[(a)] One has: 
$ \epsilon^+ (t_\gamma) = 1,\, \gamma\in L. $
\item[(b)] If $ \fg = s \ell (m|n), \ psl (n|n), \ osp(m|n) $ with $ m $ even, 
$ D(2,1;a), $ or $ F(4),  $ then $ \epsilon^- = \epsilon^+ $ on $\widehat{W}$,
in particular, 
$ \epsilon^-(t_\gamma)=\epsilon^+ (t_\gamma) = 1,\,\gamma\in L. $  
If  $\fg=G(3)$, then $\epsilon^-=\epsilon^+$ on $\widehat{W}^{\#}$,
in particular, 
$\epsilon^-(t_\gamma)=\epsilon^+(t_\gamma) = 1,\,\gamma \in L.$  
\item[(c)] If $\fg =osp(m|n)$ with $m$ odd, 
then  
$ \epsilon^- (t_\gamma) = (-1)^{(\gamma|\gamma)/2}$, $\gamma \in L$.
\end{enumerate}
\end{proposition}
\begin{proof}
Claim (a) is (6.5.10) in \cite{K2}. Claim (b) holds 
since for all these $\fg$ the half of an  even  root is never a root,
hence $s_+=s_-$. 
Claim (c) can be deduced by the discussion in
$\S6.5$ of \cite {K2}.
\end{proof}

The set of postive affine roots, corresponding to the choice of 
$ \Delta_+, $ is 
\[ \widehat{\Delta}_+ = \Delta_+ \cup  \{ \al + n \delta |\, \al \in \Delta \cup \{ 0\},\, n > 0   \}, \]
\noindent its subset of simple roots being
\[ \widehat{\Pi} = \{ \al_0 = \delta - \theta   \} \cup \Pi. \]
Let $ \widehat{\rho} = h^{\vee} \Lambda_0 + \rho \in \widehat{\fh}^{\ast}. $
One has: $\hat{\rho}(K)=h^\vee$. Define the fundamental weights
$\Lambda_i \in \hat{\fh}^*$ by
$(\Lambda_i|\alpha_j^\vee)=\delta_{ij}, \,i,j=0,...,l$, where
$\alpha^\vee =2\alpha /(\alpha|\alpha)$ (resp. $\alpha^\vee =\alpha$)
if $\alpha$ is a non-isotropic (resp. isotropic) root. (The fundamental weights
are defined up to adding a multiple of $\delta$; for $\Lambda_0$ we take the 
one we had before.)

Define the positive subalgebra of $ \widehat{\fg}: $
\[ \widehat{\fn}_+ = \underset{\al \in 
\widehat{\Delta}_+}
{\oplus} \fg_{\al} 
(= \fn_+ + \sum_{n > 0} \fg t^n  ). \]

Given $ \Lambda \in \widehat{\fh}^{\ast},  $ the highest weight $ \widehat{\fg}- $module $ L(\Lambda) $ is defined as the irreducible $ \widehat{\fg} $-module, admitting an even non-zero vector $ v_{\Lambda} $ (called the highest weight vector), such that 
\[ \widehat{\fn}_+ v_{\Lambda} = 0,\, h v_{\Lambda} = \Lambda(h) v_{\Lambda} \mbox{ for } h \in \widehat{\fh}. \]
Note that for this choice of 
$\hat{\Pi}$ 
the definition of $L(\Lambda)$ (as well as of 
$\rho$ and $\hat{\rho}$) depends only on the choice of the set of simple roots
$\Pi$ of $\fg$. 

The central element $ K $ acts as the scalar $ k = \Lambda (K)  $ on $ L(\Lambda),  $ called the level of $ \La $ (and of $ L(\La)$); the level $k=-h^\vee$ is called
\textit{critical}. Given 
$\gamma \in \widehat{\Delta}_+$, the 
$\hat{\fg}$-module $L(\Lambda)$ is called 
\textit{$\gamma$-integrable} if $\fg_{-\gamma}$ acts locally nilpotently on this module. 
 The $\hat{\fg}$-module $ L(\La) $ is called \textit{integrable} (cf. \cite{KW3}, \cite{GK}) if it is integrable with respect to the subalgebra 
$ \widehat{\fg^{\#}_{\bar{0}}} $
, that is, if $L(\Lambda)$ is $\gamma$-integrable for all 
$ \gamma \in \Delta^{\#}_{\bar{0},+}$ and all $ \gamma= \al + n \delta$ with 
$ \al \in \Delta^{\#}_{\bar{0}}$ and $n \in \ZZ_{> 0}$.

Since any integrable highest weight module of level $k=0$ over an affine Lie algebra is $1$-dimensional \cite{K2}, and since the ideal, generated by the
affine Lie subalgebra $ \widehat{\fg^{\#}_{\bar{0}}} $ in the affine Lie superalgebra 
$ \widehat{\fg} $ 
is the whole $ \widehat{\fg} $, it follows that any integrable level $k=0$ 
highest weight $ \widehat{\fg} $-module is $1$-dimensional as well.

Define the character and supercharacter of the $ \widehat{\fg} $-module $ L(\La) $ by the series, corresponding to the weight space decomposition of $ L(\La) $ with respect to $ \fh: $
\[ \ch^+_{L(\La)} (h) = \tr_{L(\La)} e^h, \quad \ch^-_{L(\La)} (h) = \str_{L(\La)} e^h, \, h \in \widehat{\fh}. \]

The series $ \ch^{\pm}_{L(\La)} $ converge to holomorphic functions in the domain $ \{ h \in \widehat{\fh} | \,\al_i (h) > 0 , i = 0, \ldots, \ell     \}  \subset X, $ and one has:
\begin{equation}
\label{eq4.4}
e^{- \La} \ch^+_{L(\La)} (\tau, z, t) = (e^{- \La} \ch^-_{L(\La)}) (\tau, z + \xi, t), 
\end{equation}
\noindent where $ \xi \in \fh $ is such that 
$\al( \xi)\in \ZZ $ 
(resp. $ \in \half + \ZZ$ ) if $ \al \in \Delta_{\bar{0}} $ 
(resp. $ \in \Delta_{\bar{1}} $).

For integrable $ L(\La) $ the explicit formulas for $ \ch^{\pm}_{L(\La)}, $ given below, show that these holomorphic functions extend to meromophic functions in the domain $ X $ (defined by (\ref{eq1.4})). We conjecture that this holds for arbitrary $ L (\La). $

In order to write down the (super)character formulae, define the normalized affine denominator and superdenominator $ \widehat{R}^+ $ and $ \widehat{R}^- $ by the formulae:

\[  \widehat{R}^{\pm} = q^{\frac{\sdim \fg}{24}} \ e^{\widehat{\rho}} \ \frac{  \prod_{n =1}^{\infty}(1-e^{-n \delta})^{\ell} \prod_{\al \in \Delta^{\re}_{\bz+}} (1-e^{-\al})      }{\prod_{\al \in \Delta_{\bo+}}(1 \pm e^{-\al})}
 \]
In coordinates (\ref{eq1.5}) they look as follows \cite{KW6}:
\begin{equation}
\label{eq4.5}
\widehat{R}^- (\tau, z,t) = e^{2 \pi i h^{\vee} t} i^{d_0 - d_1} \eta (\tau)^{\ell - d_{\bz}+ d_{\bo}} \prod_{\al \in \Delta_+} \vartheta_{11} (\tau, \al (z))
^{(-1)^{ p (\al)}},
\end{equation}
\begin{equation}
\label{eq4.6}
\widehat{R}^+ (\tau, z,t) = e^{2 \pi i h^{\vee} t} i^{d_{\bz}} \eta (\tau)^{\ell - d_{\bz} + d_{\bo}} \ \frac{\prod_{\al \in \Delta_{\bz +}} \vartheta_{11} 
(\tau, \al (z))}{\prod_{\al \in \Delta_{\bo +}} \vartheta_{10} (\tau, \al(z))},
 \end{equation}
\noindent where $ d_{\epsilon} = |\Delta_{\epsilon+}|, $ and $ \vartheta_{11},
\, \vartheta_{10} $ are the standard Jacobi theta functions of level 2 (see Appendix of \cite{KW6}).

Recall that the normalized (super)character of $ L(\La) $ is defined as
\[ \ch^{\pm}_{\La} 
= q^{\frac{| \La + \widehat{\rho}|^2}{2 (k + h^{\vee})} -  \frac{\sdim \fg}{24}} \ch^{\pm}_{L(\La)}.\]
\noindent Here and further we assume that the level $ k $ is non-critical, i.e. $ k + h^{\vee} \neq 0. $

Note that $ \ch^{\pm}_{\La}$ remains unchanged if we add to $\Lambda$ a scalar multiple of $\delta$. In particular, from the above discussion it follows 
that  $\ch^{\pm}_{\La}=1$ if $\Lambda$ has level $0$. Hence we shall exclude from consideration integrable $\hat{\fg}$-modules of level $k=0$.

Recall that for an integrable highest weight $ \widehat{\fg} $-module
$L(\Lambda)$, satisfying certain conditions, the following (super)character formula holds \cite{KW3}, \cite{GK}:
\begin{equation}
\label{eq4.7}
\widehat{R}^{\pm} \ch_{\La}^\pm =
 q^{\frac{| \La + \widehat{\rho}|^2}{2 (k + h^{\vee})}} \ \sum_{w \in \widehat{W}^{\#}} \, \epsilon^{\pm} (w) w \frac{e^{\La + \widehat{\rho}}}{\prod_{\beta \in T} (1 \pm e^{-\beta})},
\end{equation}
\noindent where $ T $ is a maximal linearly independent subset in 
$\hat{ \Delta}_{\bo+}, $ such that 
\begin{equation}
\label{eq4.8}
(\Lambda +\hat{\rho}|T)=0 \,\,\hbox{and}\,\, (T|T)=0.
\end{equation}
Formula (\ref{eq4.7}) is proved in \cite{GK}, provided that $h^\vee \neq 0$,
under the following two conditions on $T$:
\begin{equation}
\label{eq4.9}
 |T| = \mbox{defect}\, \fg,
\end{equation}
\begin{equation} 
\label{eq4.10}
T \subset \widehat{\Pi},
\end{equation}
for a ``good'' choice of $T \subset \hat{\Pi}$.

We conjecture that this is valid (with minor modifications) also for $\fg=osp(n+2|n)$ and $D(2,1;a)$
(this is proved in \cite{GK} only for vacuum modules). For these $\fg$
(i.e. different from $ps\ell(n|n)$), condition (\ref{eq4.9}) is equivalent to (\ref{eq3.12}), which is essentially used in Section 3. 
By (\ref{eq4.2}), second formula in (\ref{eq4.8}) and (\ref{eq4.10}), 
$(\hat{\rho}|T)=0$,
hence, provided that (\ref{eq4.10}) holds,   
(\ref{eq4.8}) 
is equivalent to 
\begin{equation} 
\label{eq4.11}
(\Lambda |T)=0 \,\,\hbox{and}\,\, (T|T)=0.
\end{equation}
We shall exclude from consideration 
$\fg=ps\ell (n|n)$ 
since condition 
(\ref{eq3.12}) doesn't hold in this case. The case 
$\fg=ps\ell (2|2)$ was studied in \cite{KW6}.  
 
Note that, due to (\ref{eq4.4}), it suffices to discuss the supercharacter 
formula. 
We shall express the normalized supercharacter $ch^-_\Lambda$, given by
(\ref{eq4.7}), in terms of mock theta functions 
$\Theta_{\lambda, T}^L$ 
(resp. signed mock theta functions $\Theta_{\lambda, T, \epsilon^-}^L$)
in the cases 
$\fg \neq osp(m|n)$ with $m<n, \,m$ odd, cf. Proposition \ref{prop4.3}(b) 
(resp. $\fg =osp(m|n)$ with $m<n,\,m$ odd, cf. Proposition \ref{prop4.3}(c)).  
Due to (\ref{eq4.3}), we can rewrite (\ref{eq4.7}) in the following 
equivalent form, respectively:
\begin{equation}
\label{eq4.12}
\widehat{R}^- \ch_{\La}^- =  \sum_{w \in W^{\#}} \epsilon^- (w) 
w(\Theta_{\La + \hat{\rho}, T}^L) , 
\end{equation}
and
\begin{equation}
\label{eq4.13}
\widehat{R}^- \ch_{\La}^- =  \sum_{w \in W^{\#}} \epsilon^- (w) w(\Theta_{\La + \hat{\rho,} T, \epsilon^-}^L), 
\end{equation}
\noindent where $\epsilon^-(\gamma)= (-1)^{|\gamma|^2/2},\,\gamma\in L$.

\section{Modular invariance of modified normalized supercharacters}

The \textit{modified normalized supercharacter} $ \tilde{\ch}^-_{\La} (\tau, z, t) $ is defined by replacing in 
(\ref{eq4.12}) (resp. (\ref{eq4.13}) 
the mock theta function 
$ \Theta_{\La + \widehat{\rho}, T}^L $ (resp. $ \Theta_{\La + \widehat{\rho}, T, \epsilon^-}^L $) 
by its modification 
$ \tilde{\Theta}_{\La + \widehat{\rho}, T}^L$ (resp. $ \tilde{\Theta}_{\La + \widehat{\rho}, T, \epsilon^-}^L $), 
if it exists:
\begin{equation}
\label{eq5.1}
\widehat{R}^- \ \tilde{\ch}^-_{\La} = \sum_{w \in W^{\#}} \epsilon^- (w) w (\tilde{\Theta}_{\La + \widehat{\rho}, T}^L),
\end{equation}
and
\begin{equation}
\label{eq5.2}
\widehat{R}^- \ \tilde{\ch}^-_{\La} = \sum_{w \in W^{\#}} \epsilon^- (w) w (\tilde{\Theta}_{\La + \widehat{\rho}, T}^{-;L}).
\end{equation}

From the conditions of integrability of $L(\Lambda)$ we shall see that in the 
cases when 
$\fg\neq osp(m|n)$ with $m$ odd,
the conditions for existence of the modification of the mock theta function
$\Theta_{\Lambda +\hat{\rho}, T}^L$ 
hold. In the case when $\fg=osp(m|n)$ with $m$ odd,
we must modify the signed mock theta functions $\Theta_{\Lambda +\hat{\rho},T, \epsilon}^L$ , which is 
possible only for the sign function 
$\epsilon =\epsilon_k^-$, 
given by
(\ref{eq3.8}). We shall check, using Lemma \ref{lem3.4}, that in this case one 
has: $\epsilon_k^-(t_\gamma)=\epsilon^-(\gamma)
=(-1)^{|\gamma|^2/2},\,\gamma \in L$, hence the modification
of  
$\Theta_{\Lambda +\hat{\rho},T,\epsilon^-}^L$ 
exists, and it is equal to  $\tilde{\Theta}_{\Lambda +\hat{\rho},T}^{-;L}$
(see (\ref{eq3.11})). 

Using (\ref{eq3.18}), we can rewrite (\ref{eq5.1}) as follows:

\begin{equation}
\label{eq5.3}
(\widehat{R}^- \tilde{\ch}^-_{\La} ) (\tau, z, t) = 
\sum_{w \in W^{\#}} \ \epsilon^- (w) w (\Theta^M_{(\La + \widehat{\rho})_n} (\tau, z^{(1)}, t)\tilde{\Phi}_k (\tau, z)) ,
\end{equation}
\noindent where
\[ \tilde{\Phi}_k (\tau, z) = \prod_{p = 1}^{|T|} \tilde{\Phi}^{[\frac{k + h^{\vee}}{2} |\gamma_p|^2]} (\tau, - \beta_p (z^{(2)}_p), (\beta_p  + \frac{2}{|\gamma_p|^2} \tilde{\gamma}_p) (z^{(2)}_p)), \]
\noindent and we use the decomposition (\ref{eq3.17}) of 
$ z \in \fh. $ Similarly, using the signed analogue of (\ref{eq3.18}) for the modified signed mock theta function $ \tilde{\Theta}^{-;L}_{\la, T} $ (given by (\ref{eq3.11})), we can rewrite (\ref{eq5.2}) as follows:
\begin{equation}
\label{eq5.4}
(\widehat{R}^- \tilde{\ch}_{\La} ) (\tau, z, t) = 
\sum_{w \in W^{\#}} \ \epsilon^- (w) w (\Theta^{-;M}_{(\La + \widehat{\rho})_n} (\tau, z^{(1)}, t)\tilde{\Phi}^-_k (\tau, z)) ,
\end{equation}
\noindent where
\[ \tilde{\Phi}^-_k (\tau, z) = \prod_{p = 1}^{|T|} \Phi^{-[\frac{k + h^{\vee}}{2} |\gamma_p|^2, (\La + \rho| \gamma_p)]} (\tau, - \beta_p (z_p^{(2)}), (\beta_p + \frac{2}{|\gamma_p|^2} \tilde{\gamma}_p) (z^{(2)}_p)). \]

\noindent Since $ (\La + \hat{\rho})_n \in (k+h^\vee)\Lambda_0+\CC M +\CC\delta$, and $ \CC T$ lies in the orthocomplement to $ \CC M $ in $ \fh^{\ast}, $ from (\ref{eq5.3}) and (\ref{eq5.4}) we obtain the following corollary. 

\begin{corollary}
\label{cor5.1}
The modified normalized supercharacter $ \tilde{\ch}^-_{\La} $ of an integrable $ \widehat{\fg} $-module $ L(\La) $ depends only on $ \La \!\!\mod (\CC T + \CC \delta). $
\qed
\end{corollary}
\begin{remark}
\label{rem5.2}
Since $\gamma_1,...,\gamma_m,\,\beta_1,...,\beta_n$ form a basis of $\fh$,
we see that $P_{k,T}$ (defined before Proposition \ref{prop3.2}) lies in 
$k\Lambda_0 + \CC T + \CC \delta + \CC 
\{\gamma_{n+1},...,\gamma_m \}$.
Since $M^\perp = \CC T+ \CC\{\gamma_{1},...,\gamma_n \}$,
we conclude that $P_{k,T} \!\! \mod (kM+M^\perp +\CC \delta) =
P_{k,T}\!\! \mod (kM+\CC T +\CC \delta)$.
\end{remark} 

Next, we study modular invariance of modified normalized supercharacters of integrable $ \widehat{\fg} $-modules $ L(\La) $ with the highest weight $ \La,  $ satisfying conditions (\ref{eq4.9})--(\ref{eq4.11}). For this we shall use formula 
(\ref{eq5.1}) (resp. (\ref{eq5.2})), obtained from 
(\ref{eq4.12}) (resp. (\ref{eq4.13}) by modification. 

Given $ k > 0, $ denote by $ \Omega_k \subset \widehat{\fh}^{\ast} $ the set of level $ k $ highest weights of integrable highest weight $ \widehat{\fg} $-modules, such that $ (\La | T) = 0, $ and by 
$ \Omega^!_k \subset \widehat{\fh}^{! \ast}$ 
the set of level $ k $ highest weights of integrable highest weight 
$ \widehat{\fg}^! $-modules. Note that $ \Omega_k \subset P_{k, T} $,
and $\Omega_k^! \subset \Omega_k$ by extending 
$\lambda \in \widehat{\fh}^{! \ast}$ to $\widehat{\fh}^{\ast}$ 
by $0$ on $\widehat{\fh}^{! \perp}$ . 
\begin{theorem}
\label{Th5.3}
Let $ \fg $ be one of the basic simple Lie superalgebras $ \sl (m+1 | n) $ with $ m \geq n \geq 1,\,\, \osp (2n|2m) $ or $ \osp (2n+1 | 2m) $   with $ m \geq n \geq 1,\,\,\osp(2m+1|2n)$ with $m>n\geq 1$,\,\, $F(4)$ or $G(3), $ and let $ \Delta_+ \subset \Delta $ be a choice of positive roots, such that the following conditions hold:
\begin{enumerate}
\item[(i)] $ \rho \ \big{|}_{\fh^{!}} = \rho^! $;
\item[(ii)] the set of simple roots of $ \Delta_+ $ contains a maximal subset $ T = \{ \beta_1, \ldots, \beta_n    \} $ of pairwise orthogonal isotropic roots, and $ (\beta_i | \gamma_j) = - \delta_{ij}$ for some basis $ \gamma_1, \dots, \gamma_m $ of L;
\item[(iii)] for each $ \mu \in P_{k, T}, $ such that $ (W^!)_{\mu} = \{ 1\}, $ there exists a (unique) $ w \in W^!, $ such that 
$ w (\mu) \in \Omega^!_k \! \!\mod (\CC T +\CC \delta). $
\end{enumerate}
Then the $ \CC $-span of the set of modified normalized supercharacters of highest weight $ \widehat{\fg} $-modules $ L(\La) $ with 
$ \La \in \Omega_k \! \mod (\CC T +\CC \delta) $ 
(see Corollary \ref{cor5.1}) is $ SL_2(\ZZ) $-invariant with the same transformation matrix as for $ \widehat{\fg}^! $ (which can be found in \cite{K2}, Theorem 3.8(a)).
\end{theorem}

\begin{proof}
Let $ \{ g_i \}_{i \in I}  $ be a set of left coset representative of $ W^! $ in $ W^{\#}. $ First, we consider the case $ \fg \neq \osp (m|n) $ with $ m\leq n+1, $ odd. 

By (\ref{eq5.1}) and Proposition \ref{prop4.3}(b) we have for $ \La \in \Omega_k: $
\begin{equation}
\label{eq5.5}
\tilde{\ch}^-_{\La}=  \sum_{i \in I} \epsilon^- (g_i) g_i \sum_{w \in W^!} \epsilon^+ (w) \tilde{\Theta}^L_{w (\La + \widehat{\rho}), T} / \widehat{R}^-.
\end{equation}

\noindent By Theorem 4.1 from \cite{KW6} for $ \epsilon = \epsilon^{\prime}= 0, $ we have:
\begin{equation}
\label{eq5.6}
\widehat{R}^- (- \frac{1}{\tau}, \frac{z}{\tau}, t - \frac{|z|^2}{2 \tau}) 
= i^{d_{\bo} - d_{\bz}} (-i \tau)^{\frac{\dim \fh}{2}} 
\widehat{R}^- (\tau, z, t),\,\, \,\,
\widehat{R}^- (\tau+1, z, t)=
 e^{\frac{\pi i}{12}\sdim \fg}  \widehat{R}^- (\tau, z, t).
\end{equation}
\noindent By Proposition \ref{prop3.2}(b) and Remark \ref{rem5.2}, we have:
\begin{equation}
\label{eq5.7}
\begin{split}
&\tilde{\Theta}^L_{w (\La + \widehat{\rho}), T} \left(- \frac{1}{\tau}, \frac{z}{\tau}, t - \frac{|z|^2}{2 \tau}\right) =  \\
& i^{|T|} (-i \tau)^{\frac{ \dim \fh}{2}} |L^{! \ast} / (k+ h^{\vee}) L^!|^{-\half} \! \! \!\!\!\!\! \sum_{\mu \in P_{k + h^{\vee}, T} \atop \!\! \mod ( (k+h^{\vee})M + \CC T + k \delta)} \!\!\!\!\!\!\!\!\!\!\!\!\!\!  e^{-\frac{2 \pi i}{k + h^{\vee}}(w(\bar{\La} + \rho) | \bar{\mu}))} \tilde{\Theta}^L_{\mu, T} (\tau, z, t). \\
\end{split}
\end{equation}
\noindent Using (\ref{eq5.6}) and (\ref{eq5.7}), we obtain the following modular transformation of (\ref{eq5.5})
\[ \tilde{\ch}^-_{\La} \left(- \frac{1}{\tau}, \frac{z}{\tau}, t - \frac{|z|^2}{2 \tau} \right)= i^{|T| + d_{\bz} - d_{\bo}}  | L^{! \ast} / (k+ h^{\vee}) L^! |^{-\half} \widehat{R}^- (\tau, z, t)^{-1}\] 
\[\times\sum_{i \in I} \epsilon^- (g_i) g_i \sum_{\mu \in P_{k + h^{\vee}, T} \atop \! \mod \!((k+h^{\vee}) M + \CC T + \CC \delta)} \sum_{w \in W^!} \epsilon^+ (w) e^{-\frac{2 \pi i}{k + h^{\vee}} (w (\bar{\La}+ \rho) | \mu )}\tilde{\Theta}^L_{\mu, T} (\tau, z, t). \]
\noindent If $ (W^!)_{\mu} \neq \{ 1\}, $ the sum over $ w \in W^! $ in the RHS is zero, hence we may assume that $ \mu \in P^{k + h^{\vee}}_T $ are such that $ (W^!)_{\mu} = \{ 1 \} $. Therefore, applying condition (iii), we may assume that there exists a (unique) $ w^{\prime} \in W^!, $ such that $ \mu 
= w^{\prime} (\Lambda^{\prime} + \widehat{\rho}),$ where $ \Lambda^{\prime} \in \Omega^!_k. $

Thus, using also Proposition \ref{prop4.2}(d), the RHS of the above expression for $ \tilde{\ch}_{\Lambda}^- (- \frac{1}{\tau}, \frac{z}{\tau}, t - \frac{|z|^2}{2 \tau}) $ can be rewritten as follows:
\[ i^{\half (\sdim \fg^!/\fh^!)} \big{|} L^{! \ast} / (k + h^{\vee}) L^! \big{|}^{-\half} \widehat{R}^- (\tau, z, t)^{-1} B(\tau, z, t), \]
\noindent where
\[ B = \sum_{i \in I} \epsilon^-(g_i) \ g_i
 \sum_{\Lambda^{\prime} \in \Omega^!_k \!\!\!\!\! \mod \CC \delta } \sum_{w, w^{\prime} \in W^!} \epsilon^- (w) e^{- \frac{2 \pi i}{k + h^{\vee}} (\bar{\La} + \rho | w^{-1} w^{\prime} (\bar{\La}^{\prime} + \rho )   ) } \ \tilde{\Theta}^L_{w^{\prime} (\La^{\prime} + \hat{\rho} ), T   }\]
\noindent Letting $ w^{\prime \prime} = w^{-1}w^{\prime}, $ we obtain: 
\[ B = \sum_{\La^{\prime} \in \Omega^!_k \!\!\!\!\! \mod \CC \delta } \sum_{w^{\prime \prime} \in W^!} 
\epsilon^- (w^{\prime \prime}) e^{-\frac{2 \pi i}{k + h^{\vee}} (\bar{\La} + \rho | w^{\prime \prime} (\bar{\La}^{\prime} + \rho  )  )  } \sum_{i \in I} \epsilon^- (g_i) \ g_i \sum_{w^{\prime} \in W^!} \epsilon^+ (w^{\prime}) \ \tilde{\Theta}^L_{w^{\prime}  (\La^{\prime} + \widehat{\rho}), T  }. \]
\noindent Hence, substituting (\ref{eq5.5}) in the RHS of B, we obtain for $\Lambda \in \Omega_k$: 
\begin{equation}
\label{eq5.8}
\begin{split}
& \tilde{\ch}^-_{\La} (- \frac{1}{\tau}, \frac{z}{\tau}, t - \frac{|z|^2}{2 \tau}) = i^{\half (\sdim \fg^!/\fh^!)} | L^{! \ast} / (k + h^{\vee}) L^! |^{- \half} \\
& \times \sum_{\La^{\prime} \in \Omega^{!}_k\!\!\!\!\!\! \mod \CC \delta } \sum_{w \in W^!} \epsilon^+ (w) e^{-\frac{2 \pi i (\bar{\La} + \rho | w (\bar{\La^{\prime}} + \rho))}{k+h^\vee}} \tilde{\ch}_{\La^{\prime}} (\tau, z, t). \\
\end{split}
\end{equation}
\noindent By assumption (i) we may replace $ \rho  $ by $ \rho^! $ in this expression, and, by Proposition \ref{prop4.2} (a), we also may replace $ h^{\vee} $ by $ h^{\vee !}. $ Thus
we obtain that (\ref{eq5.8}) is the $ S $-transformation formula of Theorem 3.8(a) from \cite{K2} for $\fg (A)=\hat{\fg^!}$.
For the $ T $-transformation the claim is obvious. 

In the remaining case of $ \fg = \osp (m|n) $ with $ m\leq n+1 $, odd, the proof is similar. The difference is in that $ h^{\vee} \in \half + \ZZ,\, \epsilon^- (w) \neq \epsilon^+ (w)  $ for $ w \in W^! $ and $ (\rho | \gamma_i) \in \half + \ZZ $ (not $ \in \ZZ $ as in all other cases). But then we can take $ \xi_0 = \rho, $ so that
\[ \widehat{R}^- \tilde{\ch}^-_{\La} = \sum_{i \in I} \epsilon^- (g_i) \ g_i \sum_{w \in W^!} \epsilon^- (w) \tilde{\Theta}^{-;L}_{w (\La + \hat{\rho}), T} \ . \]
So, the modified mock theta functions involved are of the form $ \tilde{\Theta}^{-;L}_{\la + \xi_0, T} $ with $ \la \in P_{k, T}, $ and we can use the last $ S $-transformation formula from Proposition \ref{prop3.3}, claim (b), and the last $ T $-transformation formula from claim (c). Note also that in this case $ W^!_{\La + \xi_0} =\{1\}. $ (In this case $\fg^!=\osp(1|n+1-m)$. The Lie superalgebra $\hat{\fg}$, where $\fg= \osp(1|n), $ is not considered in \cite{K2}, but the proof 
of Theorem 3.8(a) in this case is the same.) 
\end{proof}

In the next few subsections we will show that the conditions of Theorem \ref{Th5.3} hold for some choice of $ \Delta_+ $ in $ \fg = \sl (m+1|n) $ with
$ m \geq n$, 
$osp (2n|2m) $ with $ m \geq n$, $osp(2m+1|2n)$, 
$F(4)$, and $ G(3). $

\subsection{Case $ \mathbf{\fg = s \ell (m+1|n), \ m \geq n \geq 1.} $} $  $ \\
As usual, we choose the Cartan subalgebra $ \fh $ of $ \fg $ to be the subspace of supertraceless matrices of the space $\tilde{\fh}$ of all diagonal
 matrices. We choose for $(.|.)$ the supertrace form $(a|b)=str\, ab$. 
Let $ \epsilon_1, \dots, \epsilon_{m+1},$ $ \delta_1, \ldots, \delta_n $ be 
the basis of $ \tilde{\fh}, $ dual to the standard basis of 
$ \tilde{\fh}, $ 
so that 
\[ (\epsilon_i | \epsilon_j) = \delta_{ij}, \quad (\delta_i | \delta_j) = - \delta_{ij}, \quad (\epsilon_i | \delta_j) = 0.\]
These elements define linear functions on $\tilde{\fh}$ via the supertrace form, hence on $\fh$ by restriction. We have:
\[ \Delta_{\bar{0}}=\{\epsilon_i-\epsilon_j|\,1\leq i,j\leq m+1,\,i\neq j\}  
\cup \{\delta_i-\delta_j|\, 1\leq i,j\leq n,\, i\neq j \},\,\]
\[\Delta_{\bar{1}}=\{\pm(\epsilon_i-\delta_j)|\,1\leq i\leq m+1,\, 1\leq j\leq n\}.\]

Let  $ \gamma_i = \epsilon_{m+1} - \epsilon_i, \ i=1, \dots, m,  $ and $ \beta_i = \epsilon_i - \delta_i, \ i= 1, \ldots, n $. Let $ L $ be the $ \ZZ $-span of vectors $ \gamma_1, \ldots, \gamma_m, $ and let $ T = \{  \beta_1, \dots, \beta_n \}. $ Then $ L $ is the coroot lattice (of rank $m$) of the subalgebra
$ \fg^{\#}_{\bar{0}}$, isomorphic to $ s \ell (m+1) $ of 
the Lie algebra $\fg_{\bar{0}}$, corresponding to
$\epsilon_1,...,\epsilon_{m+1}$, and the bilinear form $(.|.)$
is normalized as in (\ref{eq4.1}). 
Notice that $ L = L_{\mathrm{even}} $ and conditions (\ref{eq3.1}) and (\ref{eq3.2}) hold. Also, condition (\ref{eq3.3}) on $ k $ holds iff $ k $ is a positive integer, which we shall assume. Furthermore, $ (\beta_i | \beta_j) = 0 $ for all $ i, j = 1, \ldots, n. $ 
Recall the set $P_{k,T}$, defined before Proposition \ref{prop3.2}.
We clearly have: 
\begin{equation}
\label{eq5.9}
P_{k,T}= \{ k\La_0 + \sum_{i = 1}^{n} \, k_i \beta_i + \sum_{i = n+1}^{m} 
k_i \epsilon_i|\,\, \,  k_1, \ldots, k_m\in \ZZ \}+\CC \delta. 
\end{equation}
Thus, we may consider the mock theta function $ \Theta^L_{\la, T} $ and construct its modification $ \tilde{\Theta}^L_{\la, T}, $ for all $\lambda \in P_{k,T}$, where $k\in \ZZ_{>0}$.
 Since condition (\ref{eq3.12}) holds, we have modular invariance of the functions $ \tilde{\Theta}^L_{\la, T}, $ given by Proposition \ref{prop3.2},
for $\lambda\in P_{k,T}$.

Choose the set of positive roots of $ \fg $ with the subset of simple roots
$\Pi= \{\alpha_i|\,i=1,...,m+n\}$, where: 
\[ \alpha_{2i-1} =  \epsilon_i - \delta_i,\, \alpha_{2i}=\delta_i - 
\epsilon_{i+1} \,\, \hbox {for}\,\,1\leq i\leq n,\, \alpha_i= \epsilon_{i-n} - 
\epsilon_{i-n+1} \,\,\hbox{for} \,\, 2n+1\leq i\leq m+n, \]
\noindent so that the first $ 2n $ simple roots are odd and the remaining $ m-n $ simple roots are even, and the highest root is $\theta=\epsilon_1-\epsilon_{m+1}$. 
We shall consider the highest weight $ \widehat{\fg} $-modules $ L (\La), $ corresponding to this choice of $\Pi$.
Recall that $ h^{\vee} = m+1-n $ and note that 
$$ 
\rho =(m - n) \sum_{i =1}^{n} \beta_i + \sum_{i = n+1}^{m} (m+1-i) \epsilon_i. 
$$
By the same method as in \cite{KW3}, Theorem 2.1, one proves the following result.
\begin{theorem}
\label{Th5.4}
The $ \widehat{\fg}$-module $ L(\La) $ with the level $k$ ($\neq 0$) highest weight $\Lambda$, such that $(\Lambda|T)=0$, is integrable if and only if $\Lambda\in P_{k,T}$  (see (\ref{eq5.9})) satisfies  
\begin{equation}
\label{eq5.10} 
k\in \ZZ_{>0}, \, k_i \in \ZZ_{\geq 0}, \ i = 1, \ldots, m, \mbox{ and } k \geq k_1 \geq k_2 \geq \ldots \geq k_m. 
\end{equation}
\end{theorem}
\begin{proof}
Recall that a $ \widehat{\fg} $-module $ L(\La) $ is integrable if and only if it is $ \al $-integrable for all simple roots of $ \widehat{\fg}^{\#}_0$, cf. \cite{K2}. For the  
choice of positive roots of $\hat{\fg}$, associated with the above 
$\hat{\Pi}$, they
are: $ \tilde{\al}_i := \al_{2i-1}+\al_{2i} = \epsilon_i - \epsilon_{i-1} $ for $ 1 \leq i \leq n  $, $\al_i = \epsilon_{i-n} -\epsilon_{i - n +1}  $ for $ 2n+1 \leq i \leq m+n $, and $ \al_0 = \delta - \epsilon_1 + \epsilon_{m+1} $. Recall also that $ L(\La) $ is $ \al $-integrable for an even simple root $ \al $ of $ \widehat{\fg} $ iff $ (\La | \al^{\vee}) \in \ZZ_{\geq 0}. $ Hence $ L(\La) $ is $ \al_i $-integrable for $ 2n+1 \leq i \leq m+n $ iff $ (\La | \epsilon_i - \epsilon_{i+1}) \in \ZZ_{\geq 0}, $ i.e. iff $ k_i - k_{i+1} \in \ZZ_{\geq 0} $ for $ 2m+1 \leq i \leq m+n-1 $ and $ k_{m+n} \in \ZZ_{\geq 0}. $

Next, a root $ \tilde{\al}_i $ is simple in the set of simple roots $ \widehat{\Pi}^{(i)} = r_{\beta_i} (\Pi), $ 
where $r_{\beta_i}$ is an odd reflection with respect to $\beta_i$ \cite{S},
\cite{KW3}.
The highest weight of $ L(\La) $ with respect to $ \widehat{\Pi}^{(i)}  $ is $ \La,  $ since $ (\La | \beta_i) = 0. $ Hence again $ L(\La) $ is $ \tilde{\al}_i $-integrable for $ 1 \leq i \leq 
n $ iff $ k_i - k_{i+1} \in \ZZ_+ $ for these $ i. $

Finally $ L(\La) $ is  $ \al_0 $-integrable iff $ (\La | \delta - \epsilon_1 + \epsilon_m) = k - k_1 \in \ZZ_{\geq 0} $ . This completes the proof. 
\end{proof}

Note that for $ \La $ of the form (\ref{eq5.9}) we have:
\[ \La + \widehat{\rho} = (k + h^{\vee})\Lambda_0 + \sum_{i =1}^{n} k^{\prime}_i \, \beta_i \ +  \sum_{i = n+1}^{m} k^{\prime}_i \, \epsilon_i,\]
\noindent where $ k^{\prime}_i = k_i + m - n $ for $ i = 1, \ldots, n, $ and $ k^{\prime}_i = k_i + m + 1 - i $ for $ i = n+1, \ldots, m. $ Hence for integrable $ L(\La) $ we have by Theorem \ref{Th5.4}:
$$
k+ h^{\vee} > k^{\prime}_1 \geq \ldots \geq k^{\prime}_n \geq k^{\prime}_{n+1} > \ldots  > k^{\prime}_m > 0, \quad k, k_i \in \ZZ.
$$
Note that (cf. (\ref{eq3.15})) $\tilde{\gamma}_p = \gamma_p + \sum_{i = 1}^{\min (p-1,n)} \, \beta_i.$ 
Hence formula (\ref{eq5.3}) for $ \La $ of the form (\ref{eq5.9}), such that 
(\ref{eq5.10}) holds, becomes
\begin{equation}
\label{eq5.11}
\begin{split}
 (\widehat{R}^- \tilde{\ch}^-_{\La}) (\tau, z, t) =  &\sum_{w \in W^{\#}} \epsilon^- (w) w (\Theta^M_{(k + h^{\vee}) \La_0 + \sum_{i = n+1}^{m}  k^{\prime}_i \epsilon_i}    (\tau, z^{(1)}, t) \\ 
& \prod_{p =1}^{n} \tilde{\Phi}^{[k + h^{\vee}]} (\tau, - \beta_p (z^{(2)}_p),  (\gamma_p + \sum_{i =1}^{p} \beta_i )(z^{(2)}_p))). \\
\end{split}
\end{equation}

By Theorem \ref{Th5.4},
the set $\Omega_k$  is non-empty iff
$k\in \ZZ_{\geq 0}$, and in this case it  
consists of elements of the form (\ref{eq5.9}) satisfying conditions 
(\ref{eq5.10}) , up to adding multiples of $\delta$ .

Note that $\fg^{!}$  is the subalgebra $s \ell(m+1-n)$  of $\fg_{\bar{0}}^\#$, corresponding to $\epsilon_{n+1},...,\epsilon_{m+1}$. Conditions (i) and (ii) of Theorem \ref{Th5.3} obviously hold. In order to check condition (iii), note that the RHS of (\ref{eq5.11}) is independent of the labels $ k_1, \ldots, k_n $ of $ \La $; if  we put all these labels equal $ k $, then, by Theorem \ref{Th5.4}, we obtain that  $ \Omega_k \mod \CC T $ consists of the set $ \Omega_k^{!} $ of highest weights of all integrable highest weight $\hat{\fg^{!}}$-modules of level $k$. Thus,  modular invariance, claimed by Theorem \ref{Th5.3}, holds in this case.  

\subsection{Case $ \mathbf{\fg = osp (2n|2m), \ m \geq n\geq 1. }$} $  $ \\
The set of roots of $ \fg $ is described in terms of an orthogonal basis $ \epsilon_1, \ldots, \epsilon_m, \ \delta_1, \dots, \delta_n$ of $\fh^*$,
such that $ (\epsilon_i | \epsilon_i) = \half, (\delta_i | \delta_i) = - \half. $
We have:
\[ \Delta_{\bz}  =  \{ \epsilon_i - \epsilon_j,\, \pm (\epsilon_i + \epsilon_j),\,\pm 2 \epsilon_i |\, 1 \leq i,j \leq m,\, i \neq j\}\cup
\{\delta_i - \delta_j,\, \pm (\delta_i + \delta_j) 
| \, 1 \leq i, j \leq n,\, i \neq j         \},\]
\[ \Delta_{\bo} = \{ \pm (\epsilon_i - \delta_j),\, \pm (\epsilon_i + \delta_j ) |\, 1 \leq i \leq m, 1 \leq j \leq n           \}.\]

Let $ \gamma_i = -2 \epsilon_i,\, i = 1, \ldots, m, $ and $ \beta_i = \epsilon_i - \delta_i, \,i = 1, \dots, n. $ Let $ L = \ZZ \gamma_1 + \ldots + \ZZ \gamma_m. $ and let $ T = \{ \beta_1, \ldots, \beta_n \}. $ Then $ L $ is the coroot lattice of $ \fg^{\#}_{\bz} = \mathrm{sp} \, (2m), $ and the bilinear form is normalized as in (\ref{eq4.1}). Again $ L = L_{\even}, (T|T) =0 $, condition (\ref{eq3.3}) on $ k $ holds iff $ k $ is a positive integer, and $P_{k,T}$
is described by (\ref{eq5.9}).
Hence we may consider the mock theta function $ \Theta^L_{\la, T} $ and construct its modification $ \tilde{\Theta}^L_{\la, T} $ for $ \la \in P_{k,T}  $.
Again, we have modular invariance of these modified mock theta functions, by Proposition \ref{prop3.2}.

Choose a subset of positive roots with the subset of simple roots $ \Pi = \{  \al_i | \ i = 1, \ldots, m+n \}, $ where $ \al_i = \epsilon_{m-i+1} - \epsilon_{m-i} $ for $ 1 \leq i \leq m-n,\, \al_{m-n+1} = \epsilon_n - \delta_n,\, \al_{m-n+2} = \delta_n - \epsilon_{n-1}, \dots, \al_{m+n-3} = \epsilon_2 - \delta_2,
\, \al_{m+n-2} = \delta_2 - \epsilon_1,\, \al_{m+n-1} = \epsilon_1 - \delta_1,
\, \al_{m+n} = \epsilon_1 + \delta_1, $ so that $\theta=2\epsilon_m$.
Recall that $ h^{\vee} = m + 1 -n, $ and note that 
\[ \rho =  \sum_{i = n+1}^{m} \, (i-n) \epsilon_i. \]
\noindent The following theorem is proved by the same method as Theorem \ref{Th5.4}.
\begin{theorem}
\label{Th5.5}
The $ \widehat{\fg}$-module $ L(\La) $ with the level $k$ ($\neq 0$) highest weight $\Lambda$, such that $(\Lambda|T)=0$, is integrable if and only if $\Lambda\in P_{k,T}$, described by (\ref{eq5.9}), satisfies 
\begin{equation}
\label{eq5.12}
k\in \ZZ_{>0}, \,k_i \in \ZZ_{\geq 0}, \quad i = 1, \ldots, m, \mbox{ and } k \geq k_m \geq k_{m -1} \geq \ldots \geq k_1. 
\end{equation}
\qed
\end{theorem}
Note that $ \tilde{\gamma}_p = \gamma_p $ for all $ p, $ in particular, $ M = \ZZ \gamma_{n+1} + \dots + \ZZ \gamma_m $ is the coroot lattice of 
$\fg^!= \mathrm{sp} \, (2m-2n)$, the subalgebra of $\fg_{\bar{0}}^\#$, corresponding to $\epsilon_{n+1},...,\epsilon_{m}$.  
Note also that for $ \La $ of the form (\ref{eq5.9}) we have
\[ \La + \widehat{\rho} = (k + h^{\vee}) \La_0 + \sum_{i =1}^{n} \, k_i \beta_i + \sum_{i = n+1}^{m} k^{\prime}_i \epsilon_i, \]
\noindent where $ k^{\prime}_i = k_i + i -n, i = n+1, \ldots, m. $ Hence for integrable $ L (\La) $ we have, by Theorem \ref{Th5.4}:
\[ k + h^{\vee} > k^{\prime}_m > \ldots > k^{\prime}_{n +1} > k_n \geq k_{n-1} \geq \ldots \geq k_1 \geq 0,  \quad k,\, k_i,\, k^{\prime}_i \in \ZZ. \]
Formula (\ref{eq5.3}) for this $ \La $ becomes:
\begin{equation}
\label{eq5.13}
\begin{split}
(\widehat{R}^- \, \ch^-_{\La}) (\tau, z, t) =  & \sum_{w \in W^{\#}} \epsilon^- (w) w (\Theta^M_{(k + h^{\vee}) \La_0 + \sum_{i = n+1}^{m} k^{\prime}_i \epsilon_i  } (\tau, z^{(1)}, t) \\ & \prod_{p =1}^{n} \tilde{\Phi}^{[k+h^{\vee}]} (\tau, - \beta_p (z^{(2)}_p), (\beta_p + \gamma_p)(z^{(2)}_p) )). \\
\end{split}
\end{equation}

Conditions (i) and (ii) of Theorem \ref{Th5.3} obviously hold. In order to check condition (iii), note that the RHS of (\ref{eq5.13}) is independent of the coefficients $ k_1, \ldots, k_n $ of $ \La $; if  we put all of them equal $ 0 $, then, by Theorem \ref{Th5.4}, we obtain that  $ \Omega_k \mod \CC T $ consists of the set $ \Omega_k^{!} $ of highest weights of all integrable highest weight $\hat{\fg^{!}}$-modules of level $k$. Thus,  modular invariance, claimed by Theorem \ref{Th5.3}, holds in this case as well.  

\subsection{Case $ \mathbf{\fg = osp (2n+1 | 2m), \ m \geq n.} $} $  $ \\
The set of roots of $ \fg $ is described in terms of the same basis as in Section 5.2, but the set of roots is larger: one should add roots 
$\{  \delta_i|\, 1 \leq i \leq n \}  $ to $ \Delta_{\bz} $  (resp. $\{ \epsilon_i|\, 1 \leq i \leq m \}$ to $ \Delta_{\bo} $) from Section 5.2. The roots $ \gamma_1, \ldots, \gamma_m $ and the set $ T = \{  \beta_1, \ldots,\beta_n\} $ are the same as in Section 5.2, and $ \fg^{\#}_0 $, $ L $ and $M$ are the same.
We have:
\begin{equation}
\label{eq5.14}
P_{k,T} = \{k \La_0 + \sum_{i =1 }^{n} k_i \beta_i + 
\sum_{i = n +1}^{m} k_i \epsilon_i|\, k_1,...,k_m\in \ZZ \}+\CC \delta.
\end{equation}

In the case $ m > n $ (resp. $m=n$) we choose the following subset of simple roots $ \Pi $ of $ \Delta: $
$\al_i  =  \epsilon_{m-i+1} - \epsilon_{m-i}$  for  
$1 \leq i \leq m - n -1$, $\al_{m-n}  =  \epsilon_{n+1} - \delta_n, \,
\al_{m-n+1}  =  \delta_n - \epsilon_n, \ldots, \al_{m+n-2}  =  
\epsilon_2 - \delta_1,\,\al_{m+n-1}  =  \delta_1 -\epsilon_1,\,
\al_{m+n}  =  \epsilon_1$ (resp.
$\al_1=\delta_n-\epsilon_n, \al_2=\epsilon_n-\delta_{n-1}, \al_3=\delta_{n-1}-
\epsilon_{n-1},...,\al_{2n-2}=\epsilon_2 - \delta_1,\,
\al_{2n-1}  =  \delta_1 -\epsilon_1,\,\al_{2n}  = \epsilon_1)$,
so that the highest root is $\theta=2 \epsilon_n$ (resp. 
$=\epsilon_n+\delta_n)$.

The conditions of integrability of a highest weight $ \widehat{\fg} $-module $ L(\La) $ with $ \La \in P_{k,T} $, given by (\ref{eq5.14}), 
are (\ref{eq5.10}). 
However, the formula for the modified supercharacter is different since $ h^{\vee} $ and the coefficients of $ \rho $ 
lie in $\frac{1}{2} +\ZZ$:
\[ h^{\vee} = m - n + \half, \quad \rho =  \half \sum_{i = 1}^{n} \beta_i + \sum_{i = 1}^{m - n} (i - \half ) \epsilon_{n + i}. \]

For $ \La  $ of the form (\ref{eq5.14}), satisfying conditions (\ref{eq5.10}) 
we have
\[ \La + \widehat{\rho} = (k + h^{\vee}) \La_0 + \sum_{i = 1}^{n} k^{\prime}_i \beta_i + \sum_{i = n+1}^{m} k^{\prime}_i \epsilon_i, \]
\noindent where $ k^{\prime}_i = k_i + \half $ for $ 1 \leq i \leq n, $ and $ k^{\prime}_i = k_i + i - n - \half  $ for $  n + 1 \leq i \leq m. $

Furthermore, 
since $2(k+h^\vee)$ is an odd integer, we obtain
\begin{equation}
\label{eq5.15}
\epsilon^-_{k+h^\vee} (t_{\gamma}) = (-1)^{|\gamma|^2/2}, \quad \gamma \in L.
\end{equation}
\noindent Therefore, the modified normalized supercharacters are expressed in this case in terms of  signed theta functions and modified signed theta functions. Namely, formula (\ref{eq5.4}) for integrable $ L(\La), $ where $ \La $ is of the form (\ref{eq5.14}), satisfying conditions (\ref{eq5.10})
, becomes (due to Corollary 1.5(a)):
\begin{equation}
\label{eq5.16}
\begin{split}
(\widehat{R}^- \tilde{\ch}^-_{\La}) (\tau, z, t) =  &  \sum_{w \in W^{\#}} \epsilon^- (w)w( \tilde{\Theta}^{-;M}_{(k + h^{\vee}) \La_0 + \sum_{i=n+1}^m k_i'\epsilon_i} (\tau, z^{(1)}, t) \\ & \prod_{p = 1}^{n} \tilde{\Phi}^{-[k + h^{\vee}; \half ]} (\tau, -\beta_p (z^{(2)}_p), (\beta_p + \gamma_p ) (z^{(2)}_p) )). \\
\end{split}
\end{equation}
Here we were able to to replace the second superscript in $\tilde{\Phi}^-$ by
$\frac{1}{2}$ due to Corollary \ref{cor1.4}.
In this case $\fg^!$ is the subalgebra $osp(1|2m-2n)$ ($=0$ if $m=n$)
of $\fg$, corresponding to $\epsilon_{n+1},...,\epsilon_m$.
Conditions of Theorem \ref{Th5.3} are again easily checked, hence we again have the modular invariance, claimed by this theorem.

\subsection{Case $ \fg = osp(2m+1|2n),\,m>n \geq 1. $} $  $ \\  
The set of roots of $ \fg $ is described in terms of an orthogonal basis
$ \epsilon_1, \ldots, \epsilon_m, \ \delta_1, \ldots, \delta_n $ of $ \fh^{\ast}, $ such that $ (\epsilon_i  |  \epsilon_i) =1,\, (\delta_i  |  \delta_i) = -1. $ 
We have: 
\begin{eqnarray*}
\Delta_{\bz} & = &  \{ \epsilon_i - \epsilon_j, \pm (\epsilon_i + \epsilon_j),  \pm \epsilon_i  | \ 1 \leq i, j \leq m \} \cup \{ \delta_i - \delta_j, \pm (\delta_i + \delta_j)   ,\pm 2 \delta_i  | \ 1 \leq i, j \leq n \}, i \neq j , \\
\Delta_{\bo}& =& \{\pm (\epsilon_i - \delta_j), \pm (\epsilon_i + \delta_j), \pm \delta_j  | \ 1 \leq i \leq m, 1 \leq j \leq n   \}.\\
\end{eqnarray*}
Let $ \gamma_i = \epsilon_{n+1} - \epsilon_i $ for $ 1 \leq i \leq n, \gamma_{n+1} = 2 \epsilon_{n+1}, \gamma_i = \epsilon_i - \epsilon_{i-1} $ for $ n+2 \leq i \leq m, $ and let $ \beta_i = \epsilon_i - \delta_i, i = 1, \dots, n. $ Then our basic condition (\ref{eq3.1}) holds. Let $ L = \ZZ \gamma_1 + \ldots + \ZZ \gamma_m $ and $ T = \{ \beta_1, \ldots, \beta_n \}. $ Then $ L $ is the coroot lattice of $ \fg^{\#}_{\bar{0}} = so (2m+1), $ the bilinear form is normalized as in (\ref{eq4.1}), $ L=L_{\even}, (T |  T) =0, $ and condition (\ref{eq3.3}) on $ k $ holds iff $ k $ is a positive integer. The set $ P_{k,T} $ is as follows: 
\begin{equation}
\label{eq5.17}
P_{k,T} = \{ k \La_0 + \sum_{i =1}^{n} k_i \beta_i + \sum_{i = n+1}^{m} k_i \epsilon_i | \,  k_1 \in \half \ZZ, \ k_i - k_1 \in \ZZ,\, i=2,...,m \} +\CC \delta.
\end{equation}
\noindent Hence we may consider the mock theta function $ \Theta^L_{\la, T} $ and construct its modification $ \tilde{\Theta}_{\la, T} $ for all $ \la \in P_{k, T}, $ where $ k \in \ZZ_{>0}, $ and we have modular invariance of the functions $ \tilde{\Theta}_{\la, T}, \la \in P_{k,T}, $ given by Proposition \ref{prop3.2}.

Choose a subset of positive roots in $ \Delta $ with the subset of simple roots $ \Pi = \{ \al_1 = \epsilon_m - \epsilon_{m-1}, \ldots, \al_{m-n} = \epsilon_{n+1} - \epsilon_n, \al_{m-n+1} = \epsilon_n - \delta_n, \al_{m-n+2} = \delta_n - \epsilon_{n-1}, \ldots, \al_{m+n-2} = \delta_2 - \epsilon_1, \al_{m+n-1} = \epsilon_1 - \delta_1, \al_{m+n} = \delta_1 \}, $ so that $ \theta = \epsilon_m + \epsilon_{m-1}. $ Recall that $ h^{\vee} = 2(m-n) -1 $ and 
\[ \rho = \sum_{i = n+1}^{m} (i - n- \half ) \epsilon_i - \half \sum_{i = 1}^{n} \beta_i. \]
\noindent The following theorem is proved by the same method as Theorem \ref{Th5.4}.
\begin{theorem}
\label{th5.6}
The $ \widehat{\fg}$-module $ L(\La) $ with the level $k$ ($\neq 0$) highest weight $\Lambda$, such that $(\Lambda|T)=0$, is integrable if and only if $\Lambda\in P_{k,T}$, described by (\ref{eq5.17}), satisfies
\begin{equation}
\label{eq5.18}
k \in \ZZ_{> 0}, \ k_1 \in \half \ZZ_{\geq 0}, \ k_i - k_1 \in  \ZZ_{\geq 0}, \ k_m \geq k_{m-1} \geq \ldots \geq k_1, \ k \geq k_m + k_{m-1}.
\end{equation}
\end{theorem}   
Note that $ \fg^{!} $ is the subalgebra $ so(2m+1 -2n) $ of 
$ \fg^{\#}_{\bar{0}},$ corresponding to $\epsilon_{n+1}, \ldots, \epsilon_m.$ 
It is easy to check the conditions of Theorem \ref{Th5.3}. Hence, modular invariance, claimed by Theorem \ref{Th5.3}, holds in this case.

\subsection{Case $ \fg = F(4) $.} $  $ \\  
We consider the set of simple roots $ \Pi = \{ \al_1, \al_2, \al_3, \al_4 \}, $ where $ \al_1 $ is even and $ \al_2, \al_3, \al_4 $ are odd isotropic, with the following non-zero scalar products:
\[ (\al_1 | \al_1) = 2, \quad (\al_1|\al_2) = -1, \quad (\al_2|\al_3) =1, \quad (\al_2|\al_4) = \half, \quad (\al_3|\al_4) = - \frac{3}{2}.  \]
\noindent The highest root is $ \theta = 2 \al_1 + 3 \al_2 + \al_3 + 2 \al_4.$ The subalgebra $ \fg^{\#}_{\bar{0}}$ is the simple Lie algebra, isomorphic to $ \so (7), $ with simple roots $ \{ \al_2 + \al_3, \al_1, \al_2 + \al_4  \}. $ We choose the following basis of the coroot lattice $ L $ of $ \fg^{\#}_{\bar{0}}: $
\[ \gamma_1 = \al_1, \quad \gamma_2 = \al_1 + \al_2 + \al_3, \quad \gamma_3 = \al_1 + 2 \al_2 + 2 \al_4. \]
\noindent We let $ \beta_1 = \al_2, T= \{ \beta_1\}. $ Then one has:
\[ \tilde{\gamma}_1 = \gamma_1, \quad \tilde{\gamma}_2 = \gamma_2 + \beta_1, \quad \tilde{\gamma}_3 = \gamma_3,\,\, \rho = \gamma_2 + \gamma_3 .\]

Furthermore, $ \La \in P_{k,T} $ iff, up to adding a multiple of $ \delta, $ we have:
\begin{equation}
\label{eq5.19}
\La = k \La_0 + k_1 \beta_1 + k_2 \gamma_2 + k_3 \gamma_3,
\end{equation}
\noindent where $ k_2 - k_1, k_1 +k_2 - k_3, k_1 - 2k_2 - 2k_3 \in \ZZ; $ (equivalently: $ k_1 = \frac{1}{3} (2a_2 + a_3), \quad k_2 = \frac{1}{3} (3 a_1 + 2a_2 + a_3),  \quad k_3 = \frac{1}{3} (3a_1 + a_2 + 2a_3) $ for some $ a_1, a_2, a_3 \in \ZZ $). 

We consider the irreducible highest weight $ \widehat{\fg} $-modules $ L (\La) $ for the corresponding to $ \Pi $ choice of 
positive affine roots. The following theorem is proved along the same lines as Theorem \ref{Th5.4}.  
\begin{theorem}
\label{Th5.7}
The $ \widehat{\fg} $-module $ L(\La) $ with the highest weight $\Lambda$ such 
that $(\Lambda|T)=0$, is integrable if and only if $\Lambda\in P_{k,T}$  (see (\ref{eq5.19})) satisfies  
\[ k - k_2 - k_3, \ k_2 - k_1, \ k_1 + k_2 - k_3, \ k_1 - 2k_2 -2k_3 \in \ZZ_{\geq 0}; \]
\noindent equivalently, iff for some $ a_0, a_1, a_2, a_3 \in \ZZ_{\geq 0} $ one has:\[ k = a_0 + 2a_1 + a_2 + a_3, \quad k_1 = \frac{1}{3} (2 a_2 + a_3), \quad k_2 = \frac{1}{3} (3a_1 + 2a_2 + a_3), \quad k_3 = \frac{1}{3} (3a_1 + a_2 + 2a_3). \]
\end{theorem}

Next, $ \fg^! $ is the simple subalgebra of $ \fg^{\#}_{\bar{0}}, $ isomorphic to $ \sl (3), $ with simple roots $ \gamma_2, \gamma_3. $ Hence $ \rho^! = \rho, $ and conditions (i) and (ii) of Theorem \ref{Th5.3} hold. It is also not difficult to check condition (iii). Hence Theorem \ref{Th5.3} holds in this case as well.

\subsection{Case $ \fg = G(3) $.} $  $ \\ 
We consider the set of simple roots $ \Pi = \{ \al_1, \al_2, \al_3\}, $ where $ \al_1 $ is even and $ \al_2, \al_3$ are odd, with the following non-zero scalar products:
\[ (\al_1 | \al_1) = 2, \quad (\al_1|\al_2) = -1, \quad (\al_2|\al_3)=
\frac{2}{3}, \quad (\al_3|\al_3) = - \frac{2}{3} . \]
\noindent The highest root is $ \theta = 2 \al_1 + 3 \al_2 + 3\al_3 .$ The subalgebra $ \fg^{\#}_{\bar{0}} $ is the simple Lie algebra, isomorphic to $G_2, $ with simple roots 
$ \{ \al_1\,\, \al_2 + \al_3  \}. $ We choose the following basis of the coroot lattice $ L $ of $ \fg^{\#}_{\bar{0}}: $
\[ \gamma_1 = \al_1, \quad \gamma_2 = \theta. \]
\noindent We let $ \beta_1 = \al_2, \,T= \{ \beta_1\}. $ One has:
\[ \tilde{\gamma}_1 = \gamma_1, \quad \tilde{\gamma}_2 = \gamma_2 + \beta_1, 
\,\, \rho = -\frac{1}{2}\beta_1 + \frac{1}{2}\gamma_2 .\]
\noindent Furthermore, $ \La \in P_{k,T} $ iff, up to adding a multiple of $ \delta, $ we have:
\begin{equation}
\label{eq5.20}
\La = k \La_0 + k_1 \beta_1 + k_2 \gamma_2, \hbox{ where}\, k_2 - k_1,\, 2k_2 \in \ZZ.  
\end{equation}
 
We consider the irreducible highest weight $ \widehat{\fg} $-modules $ L (\La) $ for the corresponding to $ \Pi $ choice of 
positive affine roots. 
The following theorem is proved along the same lines as Theorem \ref{Th5.4}.  
\begin{theorem}
\label{Th5.8}
The $ \widehat{\fg} $-module $ L(\La) $ with the highest weight $\Lambda$ 
such that $(\Lambda|T)=0$, is integrable if and only if $\Lambda\in P_{k,T}$  (see (\ref{eq5.20})) satisfies  
\[ k - 2k_2, \,2k_1,\,2k_2,\, k_2 - k_1 \in \ZZ_{\geq 0}. \]
\end{theorem}

Next, $ \fg^! $ is the simple subalgebra of $ \fg^{\#}_{\bar{0}}, $ isomorphic to $ \sl (2), $ with the simple root $ \gamma_2. $ Hence 
$ \rho^! = \frac{1}{2}\gamma_2, $ and conditions (i) and (ii) of Theorem \ref{Th5.3} hold. It is also not difficult to check condition (iii). Hence Theorem \ref{Th5.3} holds in this case as well.

\section{Remaining cases of modular invariance of modified normalized supercharacters}
In Section 5 we found all cases of basic simple Lie superalgebras $ \fg $ (which are not Lie algebras), for which there exists a set of simple roots $ \Pi $ and a maximal set of pairwise orthogonal (isotopic) roots $ T $ in $ \Pi, $ such that the span of modified normalized supercharacters of integrable $ \widehat{\fg} $-modules $ L(\La) $ with $ (\La  |  T) = 0 $ is $ SL_2 (\ZZ) $-invariant. It turned out that in the remaining cases, 
$ \fg=osp(2m|2n),\, m\geq n+1 $ , and $\fg=D(2,1;a)$, there exist a set of affine simple roots $\hat{ \Pi} $ and two maximal subsets $ T $ and $ T^{\prime} $ in $\hat{ \Pi} $ of pairwise orthogonal roots, such that the span of the union of modified normalized supercharacters of integrable $ \widehat{\fg} $-modules $ L(\La) $ with $ (\La  |  T) = 0 $ or $ (\La  |  T^{\prime}) = 0 $ is $ SL_2 (\ZZ) $-invariant.
A similar situation occurs in the case
$\fg=osp(3|2)$, subprincipal integrable $\hat{\fg}$-modules.

\subsection{Case $ \fg = \osp (2m  |  2n),\, m > n +2. $} $ $\\
The set of roots is described in terms of the same basis as is \S 5.4, and it is a subset of the set of roots of $ \osp (2m+1  |  2n), $ obtained by removing the roots $ \pm \epsilon_i $  and $ \pm \delta_j. $
Then $ L = \ZZ \{ \epsilon_i \pm \epsilon_j \ | \ 1 \leq i < j \leq m     \} $ is the coroot lattice of 
$ \fg^{\#}_{\bz} = so (2m) $. 
Let $ \sigma_0 \in W^{\#} $ be the element, defined by 
\[ \sigma_0 (\epsilon_1) = - \epsilon_1, \ \sigma_0 (\epsilon_{n+1}) = - \epsilon_{n+1}, \,\mbox{and}\, \sigma_0 (\epsilon_i) =  \epsilon_i\ \mbox{ otherwise }. \]
Let $ T = \{ \beta_i = \delta_i - \epsilon_i \ | \ i = 1, \ldots, n \}$, 
$ T^{\prime} = \sigma_0 (T)=\{ \beta_i^{\prime} = \sigma_0 (\beta_i) \ | \ i = 1, \ldots, n \} $, 
and choose the following bases of $ L : \gamma_i = \epsilon_i - \epsilon_m $ for $ 1 \leq i \leq n,\, \gamma_{n+1} = \epsilon_{n+2} + \epsilon_{n+1},\, \gamma_i = \epsilon_i - \epsilon_{i-1} $ for $ n+2 \leq i \leq m; $ and $ \gamma^{\prime}_i = \sigma_0 (\gamma_i)$ for  $1 \leq i \leq m. $ Then our basic condition $ (\ref{eq3.1}) $ holds in both cases, the bilinear form $ \bl $ is normalized as in (\ref{eq4.1}), $ L = L_{\even}, (T  |  T) = (T^{\prime}  |  T^{\prime}) = 0, $ and condition (\ref{eq3.3}) on $ k $ holds iff $ k $ is a positive integer. 

We have:
\begin{equation}
\label{eq6.1}
P_{k,T} = \{ k \La_0 - \sum_{i = 1}^{n}k_i \beta_i + \sum_{i = n+1}^{m} k_i \epsilon_i \ | \ k_1 \in \half \ZZ,\, k_i - k_1 \in \ZZ, i = 2, \ldots, m \} + \CC \delta, 
\end{equation}
\noindent and $ P_{k, T^{\prime}} $ is obtained from $ P_{k,T} $ replacing $ \beta_1 $ by $ \beta^{\prime}_1. $ Hence we may consider the modification 
$ \tilde{\Theta}_{\la, T} $ (resp. $ \tilde{\Theta}_{\la, T^{\prime}} $) 
of a mock theta function 
$ \Theta_{\la, T} $ (resp $ \Theta_{\la, T^{\prime}} $) 
for all $ \la \in P_{k,T}$ 
(resp. $P_{k, T^{\prime}} $), where $ k \in \zp, $ and we have $ SL_2 (\ZZ) $-invariance of the span of the set $ \{  \tilde{\Theta}_{\la, T} \ | \ \la \in P_{k, T} \} $ and the set
 $ \{ \tilde{\Theta}_{\la, T^{\prime}} 
\ | \ \la \in P_{k, T^{\prime}} \}. $
Recall that $ \tilde{\Theta}_{\la, T} $ (resp. $  \tilde{\Theta}_{\la, T^{\prime}} $) depends only on $\la \!\mod (\CC T $ (resp. $\CC T'$) $+\CC \delta)$. 
Choose the following subset of simple roots: $ \Pi = \{ \al_i =\epsilon_{m-i+1} - \epsilon_{m-i} \mbox{ for } 1 \leq i \leq m-n-1, \ \al_{m-n} = \epsilon_{n+1} - \delta_n, \ \al_{m-n+1} = 
\delta_n - \epsilon_n, ..., \al_{m+n-3} = \delta_2 - \epsilon_2, \ \al_{m+n-2} = \epsilon_2 - \delta_1, \ \al_{m+n-1} = \delta_1 - \epsilon_1, \ \al_{m+n} = \delta_1 + \epsilon_1 \}, $ so that $ \theta = \epsilon_m + \epsilon_{m-1}. $ Recall that $ h^{\vee} = 2(m-n-1) $ and $ \rho = \sum_{i = n+2}^{m} (i-n-1) \epsilon_i = \sigma_0 (\rho). $

The following theorem is proved by the same method as Theorem \ref{Th5.4}.

\begin{theorem}
\label{th6.1}
The $ \widehat{\fg} $-module $ L(\La) $ with the level $ k $ ($\neq 0$) highest-weight $ \La, $ such that $ (\La | T) = 0 $ (resp. $ (\La | T^{\prime} ) =0 $), is integrable if and only if $ \La \in P_{k,T} $ (resp. $ \in P_{k, T^{\prime}}$),  
described by (\ref{eq6.1}), satisfies
\[ k \in \ZZ_{> 0},\, k_1 \in \half \ZZ,\, k_i - k_1 \in \ZZ, \]
\[ k_m \geq \ldots \geq k_2 \geq |k_1|,\, k \geq k_m + k_{m-1},\, k_2=-k_1 
\Rightarrow k_1 = k_2 = 0. \]
\qed
\end{theorem}
Denote by 
$\Omega_k$ (resp. $\Omega_k^{\prime}$) the set of all $\Lambda \in P_{k,T}$
(resp. $\in P_{k,T^{\prime}}$), such that the $\hat{\fg}$-module $L(\Lambda)$
is integrable and 
$(\Lambda|T)=0$  (resp. $(\Lambda|T^{\prime})=0$). Note that 
$\Omega_k \cap\Omega_k^{\prime}$ consists of weights, for which $k_1=0$.

Note that $ \fg^! $ is the subalgebra $ \so (2m-2n), $ corresponding to $ \epsilon_{n+1}, \ldots, \epsilon_m, $ and conditions (i) and (ii) of Theorem \ref{Th5.3} hold. However, condition (iii) fails. For this condition to hold it suffices to consider the union of the sets $ P_{k,T} $ and $ P_{k,T^{\prime}} $,
and modify the proof of Theorem \ref{Th5.3} as follows.

Since the modification of of 
$\ch_{\La}$ 
depends on the choice of $T$, we shall denote the modified supercharacters
by $\tilde{\ch}_{\La, T}$ (resp. $\tilde{\ch}_{\La, T^{\prime}}$) if
$\La \in \Omega_k$ (resp. $\in \Omega_k^{\prime}$). The functions
$\tilde{\ch}_{\La, T}$ (resp. $\tilde{\ch}_{\La, T^{\prime}}$) depend only
on $\La \in \Omega_k \!\mod (\CC T +\CC\delta) $ (resp. $\Lambda \in \Omega_k^{\prime} \! \mod (\CC T^{\prime}+\CC \delta)$). 
 
We have, in the notation of the proof of Theorem \ref{Th5.3}, formula (\ref{eq5.5}) for $ \La \in \Omega_k, $ and the analogous formula
for $ \La \in \Omega^{\prime}_k$:
\begin{equation}
\label{eq6.2}
\begin{split}
\tilde{\ch}^-_{\La, T} = \sum_{i \in I} \epsilon^+ (g_i)\, g_i \sum_{w \in W^!} \epsilon^+ (w) \tilde{\Theta}^L_{w (\La + \widehat{\rho}), T} / \widehat{R}^-,\\
\tilde{\ch}^-_{\La, T^{\prime}} = \sum_{i \in I} \epsilon^+ (g_i)\, g_i\, \sigma_0 \sum_{w \in W^!} \epsilon^+ (w) \tilde{\Theta}^L_{w (\La + \widehat{\rho}), T^{\prime}} / \widehat{R}^-.
\end{split}
\end{equation}

Let $ \overset{\circ}{\Omega}^{\prime}_k = \{ \la \in \Omega^{\prime}_k  \,
\hbox {of the form}\, (\ref{eq6.1})\, \hbox{with}\,  k_{n+1} > 0\} $ .  
Then, using (\ref{eq5.5}) and (\ref{eq6.2}), by the same argument, as in the proof of Theorem \ref{Th5.3}, we obtain for $ \La \in \Omega_k $
\[\tilde{\ch}^-_{\La,T} (\tau, z, t) \big{|}_S = i^{n + d_{\bz} - d_{\bo}} \big{|} L^{! \ast} / (k + h^{\vee}) L^! \big{|}^{-\frac{1}{2}} \]
\[ \times ( \mkern-6mu \sum_{\substack{\mu \in \Omega_k \\ \mod \\ \CC T + \CC \delta}}  \sum_{w \in W^!} \!\! \epsilon^+ \! (w) e^{-\frac{2 \pi i}{k + h^{\vee}} (\bar{\La} + \rho | w (\bar{\mu}+\rho))} \tilde{\ch}^-_{\mu,T} + \mkern-18mu \sum_{\substack{\mu\in \overset{\circ}{\Omega}^{\prime}_k \\ \mod \\ \CC T^{\prime} +  \CC \delta}}  \sum_{w \in W^!}  \epsilon^+ (w) e^{-\frac{2 \pi i}{k + h^{\vee}} (\bar{\La} + \rho | \sigma_0 w (\bar{\mu}+\rho))} 
\tilde{\ch}^-_{\mu, T^{\prime}}  ) (\tau,z,t),\]
\noindent and a similar formula for $ \La \in \overset{\circ}{\Omega}^{\prime}_k, $ replacing in $ T $ by $ T^{\prime} $ and $ \Omega_k $ by $ \overset{\circ}{\Omega}^{\prime}_k. $ Letting 
$ \tilde{\ch}^-_{\la} =  \tilde{\ch}^-_{\la, T}$
  (resp. $ \tilde{\ch}^-_{\la} =  \tilde{\ch}^-_{\la, T^{\prime}}$) 
if $ \la \in \Omega_k $ (resp. $ \overset{\circ}{\Omega}^{\prime}_k $), we obtain from these formulas the following unified formula for $ \la \in \Omega_k \cup \overset{\circ}{\Omega}_k: $
\begin{equation}
\label{eq6.3}
\tilde{\ch}^-_{\la} (\tau, z, t) \big{|}_S  = i^{n + d_{\bz} - d_{\bo}} \big{|} L^{! \ast} / (k+h^{\vee}) L^{!}\big{|}^{-\half}  \sum_{\substack{ \mu \in \Omega_k \!\!\! \mod (\CC T + \CC \delta) \\ \cup \overset{\circ}\Omega^{\prime}_k 
\!\!\! \mod (\CC T^{\prime} + \CC \delta)}} S_{\la, \mu}\, \tilde{\ch}^-_{\mu},
\end{equation}
 \noindent where
 \[ S_{(\la, T), (\mu, T)} = \sum_{w \in W^!} \epsilon^+ (w) e^{- \frac{2 \pi i}{k + h^{\vee}}(\bar{\la}+\rho | w (\bar{\mu} + \rho))},  \]
 \noindent and similarly for $ S_{(\la, T^{\prime}), (\mu, T)}, $ etc., replacing $ \bar{\la} + \rho $ by $ \sigma_0 (\bar{\la} + \rho), $ etc. 
 
 Denote by $ \Omega^!_k \subset \widehat{\fh}^{! \ast } $ the set of all dominant integral weights for $ \widehat{\fg}^!. $ Note that for $ \La \in \Omega^{\prime}_k, $ such that $ k_{n+1} = 0 $ in its form (\ref{eq6.1}), we have $ \La \in \Omega_k $ and $ \tilde{\ch}^-_{\La, T} = \tilde{\ch}^-_{\La, T^{\prime}}. $ Hence we may consider only $ \La \in \Omega_k \cup \overset{\circ}{\Omega}^{\prime}_k $. For $ \La \in \Omega_k $ (resp. $\overset{\circ}{\Omega}^{\prime}_k  $) of the form (\ref{eq6.1}) we let
 \[ \La^! = k \La_0 + \sum_{i = n+1}^{m} k_i \epsilon_i \, ( \mbox{resp. } \sigma_0 (\La^!) = k \La_0 - k_{n+1} \epsilon_{n+1} + \sum_{i = n+2}^{m} k_i \epsilon_i ). \]
\noindent Then we have a bijective map $ \nu_{\la}: \Omega_k \cup \overset{\circ}{\Omega}^{\prime}_k \rightarrow \Omega^!_k, \ \la \longmapsto \nu_{\la}, $ defined by $ \nu_{\la} = \la^! $ if $ \la \in \Omega_k $ (resp. $ = \sigma_0 (\la^!) $ if $ \la \in \overset{\circ}{\Omega}^{\prime}_k $). Furthermore, for $ \la, \mu \in \Omega_k \cup \overset{\circ}{\Omega}^{\prime}_k $ we have 
\[ S_{\la, \mu} = S^!_{\nu_{\la}, \nu_{\mu}}, \, \hbox{ where}\,\,
 S^!_{\nu_{1}, \nu_{2}} = \sum_{w \in W^!} \epsilon^+ (w) e^{- \frac{2 \pi i}{k + h^{\vee}}(\bar{\nu}_1+\rho^! | w (\bar{\nu}_2 + \rho^!))}, \quad \nu_1, \nu_2 \in \Omega^!_k. \]
\noindent It follows from (\ref{eq6.3}) that Theorem \ref{Th5.3} holds in this case as well if we replace $ \Omega_k \!\! \mod (\CC T + \CC \delta) $ in this theorem by the union of $ \Omega_k \!\! \mod (\CC T + \CC \delta) $ and $ \overset{\circ}{\Omega}^{\prime}_k \!\! \mod (\CC T^{\prime} + \CC \delta). $

\subsection{Case $ \fg = \osp (2n+2 | 2n). $} $  $\\
The set of roots, $ L, \fg^{\#}_{\bz} $ and $ T=\{\beta_1,...,\beta_n \} $ are the same as in \S 6.1 for $  m = n+1. $ We choose the following basis of $L: \gamma_i = \epsilon_{n+1} - \epsilon_i  $ for $ i = 1, \ldots, n$, $\gamma_{n+1} = 2 \epsilon_{n+1}.$

Choose the following set of simple roots $ \widehat{\Pi} = \{ \al_0 = \delta - (\epsilon_1 + \delta_1),\, \al_1 = \epsilon_1 - \delta_1,\, \al_2 = \delta_1 - \epsilon_2,\, \al_3 = \epsilon_2 - \delta_2, \ldots, \al_{2n-2} = \delta_{n-1} - \epsilon_n,\, \al_{2n-1} = \epsilon_n - \delta_n,\, \al_{2n} = \delta_n - \epsilon_{n+1}, \al_{2n+1} = \delta_n + \epsilon_{n+1}     \}. $ Note that $ \widehat{\rho} = \rho = 0. $ Let $ \sigma_0 $ be the involution of the Lie superalgebra $ \widehat{\fg}, $ such that $ \sigma_0 $ permutes $ \al_0 $ with $ \al_1, \al_{2n} $ with $ \al_{2n+1}, $ and fixes all other roots from $ \widehat{\Pi}. $ Note that $ \sigma_0 (\epsilon_1) = \delta - \epsilon_1, \sigma_0 (\epsilon_{n+1}) = - \epsilon_{n+1}, \sigma_0 (\epsilon_i) = \epsilon_i  $ for $ i \neq 1 , n+1,\, \sigma_0 (\delta_j) = \delta_j. $ We normalize $ \La_1 $ by the condition $ |\La_1|^2 = 0, $ so that $ \La_1 = \La_0 + \epsilon_1 - \half \delta ,$ and $ \sigma_0 $ permutes $ \La_0 $ and $ \La_1. $
We let 
$T^{\prime} = \sigma_0 (T),\, \beta^{\prime}_i = 
\sigma_0 (\beta_i),\, \gamma^{\prime}_i = \sigma_0 (\gamma_i). $
The following theorem is proved by the same method as Theorem \ref{Th5.4}.
\begin{theorem}
\label{Th6.2}
The $ \widehat{\fg} $-module $ L(\La) $ with the level $ k$ ($\neq 0 $) highest weight $ \La, $ such that $ (\La | T) = 0 $ (resp. ($ \La | T^{\prime}) = 0 $), is integrable if and only if, up to adding a multiple of $ \delta$, 
\[
\La = k \La_0 + k_{n+1} \epsilon_{n+1} + \sum_{i =1}^{n} k_i \beta_i\, 
(\hbox{resp.} = k \La_1 - k_{n+1} \epsilon_{n+1} + \sum_{i =1}^{n} k_i \beta^{\prime}_i ) 
\]
satisfies 
\[
k \in \ZZ_{>0}, k\,_{n+1} \in \half \ZZ,\, k_i - k_{n+1} \in \ZZ,
\]
\[  k_1 \geq k_2 \geq \ldots \geq k_n \geq |k_{n+1}|,\, k \geq k_1 + k_2, \,
\hbox{and}\,\,k = k_1 + k_2 \Rightarrow k_1 = k_2. \]
\qed
\end{theorem}
Denote by $ \Omega_k $ (resp. $ \Omega_k^{\prime} $) the set of level $ k $ weights $ \La, $ orthogonal to $ T $ (resp. $ T^{\prime} $), for which the $ \widehat{\fg} $-module $ L(\La) $ is integrable. 
In order to write down the formulas for the modified normalized supercharacters for $ \La \in \Omega_k $ (resp. $ \La^{\prime} \in \Omega^{\prime}_k $), introduce the following notation:
\[ e_p = \sum_{j = 1}^{p} \epsilon_j, \ d_p = \sum_{j=1}^{p} \delta_j, \ \tilde{\Phi}_k = \prod_{p =1 }^{n} \tilde{\Phi}^{[k]} (\tau, \delta_p - \epsilon_p, \epsilon_{n+1} + e_{p-1} - d_p). \]
Then, by formula (\ref{eq5.3}), we have, in notation of Theorem \ref{Th6.2}:
\begin{equation}
\label{eq6.4}
\widehat{R}^- 
\tilde{\ch}^-_{\La} = e^{2 \pi i k t} \sum_{w \in W^{\#}} \epsilon^- (w) w \left( \Theta_{2k_{n+1},\, 2k} (\tau, e_{n+1} - d_n) \tilde{\Phi}_k \right),
\end{equation}
\begin{equation}
\label{eq6.5}
\widehat{R}^- 
\tilde{\ch}^-_{\La^{\prime}} =  e^{2 \pi i k t} \sum_{w \in W^{\#}} \epsilon^- (w) w \left( \Theta_{2k_{n+1} + 2k, 2k} (\tau, e_{n+1} - d_n) \tilde{\Phi}_k \right).
\end{equation}

Since the span of the theta functions
\[ \{  \Theta_{2k_{n+1},\, 2k}, \ \Theta_{2k_{n+1}+2k,\, 2k} \ | \ k_{n+1} \in \half \ZZ, -k \leq 2k_{n+1} \leq k   \} \]
\noindent is $ SL_2 (\ZZ) $-invariant and $ \tilde{\Phi}_k $ is $ SL_2 (\ZZ) $-invariant, we conclude that the span of all modified normalized supercharacters $ \tilde{\ch}^-_{\La} $ for $ \La \in \Omega_k $ and $\tilde{ \ch}^-_{\La^{\prime}} $ for $ \La^{\prime} \in \Omega^{\prime}_k $ is $ SL_2 (\ZZ) $-invariant.

In order to write down the explicit transformation formula, note that 
\[ \Omega_k \cap \Omega^{\prime}_k = \{ k \La_0 \pm \frac{k}{2} \epsilon_{n+1} + \frac{k}{2} \sum_{i =1}^{n} \beta_i = k \La_1 \mp \frac{k}{2} \epsilon_{n+1} + \frac{k}{2} \sum_{i = 1}^{n} \beta_i^{\prime}   \}, \]
and let $\overset{\circ}{\Omega}_k^{\prime} = \Omega^{\prime}_k / (\Omega_k \cap \Omega^{\prime}_k). $
For $ \la = k \La_0 + k_{n+1} \epsilon_{n+1} + \sum_{i =1}^{n} k_i \beta_i \in \Omega_k$ (resp. $ \la = k \La_1 - k_{n+1} \epsilon_{n+1} + \sum_{i =1}^{n} k_i \beta_i^{\prime} \in \overset{\circ}{\Omega}_k^{\prime} $) let
\[ \nu_{\la} = k_{n+1} \epsilon_{n+1}\,\, (\mbox{resp. } = (k+k_{n+1}) \epsilon_{n+1}). \]
\noindent Then for 
$ \la \in \Omega_k \cup \overset{\circ}{\Omega}_k^{\prime}$ (disjoint union) 
we deduce, from the transformation formula for theta functions (see e.g. (A5) 
from \cite{KW6}), equations (\ref{eq5.6}) and Theorem \ref{th1.1} (a) the following transformation formulae:
\begin{equation}
\label{eq6.6}
\begin{split}
& \tilde{\ch}^-_{\la} \left( - \frac{1}{\tau}, \frac{z}{\tau}, t - \frac{|z|^2}{2 \tau} \right) 
 = \frac{1}{2 \sqrt{k}} \sum_{\substack{\mu \in \Omega_k \mkern-15mu  \mod{ \CC T + \CC \delta \\ \,\cup\, \overset{\circ}{\Omega}_k^{\prime} \mkern-15mu \mod \CC T^{\prime}+ \CC \delta}}} e^{ \frac{-2 \pi i}{k} (\nu_{\la} | \nu_{\mu} )}  \tilde{\ch}^-_{\mu}(\tau,z,t),\\
& \tilde{\ch}^-_{\la}(\tau +1,z,t)  =  e^{ \frac{\pi i}{k}|\nu_{\la}|^2-\frac{\pi i}{12}}  \tilde{\ch}^-_{\la}(\tau,z,t). \\
\end{split}
\end{equation}

\subsection{Case $ \fg = D (2,1;a),\, a \in \QQ,\, -1 < a < 0. $} $  $\\
Recall that the Lie superalgebra $ D(2,1;a) $ is a family of simple 17-dimensional Lie superalgebras with the even part isomorphic to $ \sl (2) \oplus \sl (2) \oplus \sl (2), $ depending on a parameter $ a \neq 0, -1. $ Two members of this family are isomorphic iff the corresponding values of parameters lie on the orbit of the group, generated by transformations $ a \mapsto -1 -a, \ a \mapsto \frac{1}{a} $ \cite{K1}. Hence, provided that $ a \in \QQ, $ we may assume that $ -1 < a < 0, $ and write it uniquely in the form:
\begin{equation}
\label{eq6.7}
 a = -\frac{p}{p + q}, \mbox{ where } p, q \in \zp \mbox{ are coprime.}
\end{equation}

We consider the set of odd simple roots $ \Pi = \{  \al_1, \al_2, \al_3 \}$ of $ \fg $, with the following non-zero scalar products:
\[ (\al_1 | \al_2) = a, \quad (\al_1| \al_3) = -(a+1),  \quad (\al_2 | \al_3) = 1. \]
\noindent The highest root is $ \theta =  \al_1 + \al_2 + \al_3,  $ and $ (\theta | \theta ) = 0; $ we also have $ \rho = 0, \ h^{\vee} = 0. $ We let $ \fg^{\#}_{\bz} $ be the subalgebra of $ \fg_{\bz}, $ isomorphic to $ \sl (2) \oplus \sl(2)  $ with the simple roots $ \al_1 + \al_2, \al_1 + \al_3. $ The corresponding coroot lattice is $ L = \ZZ \frac{\al_1 + \al_2}{a} + \ZZ \frac{\al_1+ \al_3}{a+1}. $ Note that, due to (\ref{eq6.7}), this lattice is negative definite.

In this section we consider non-critical (i.e. of non-zero level) irreducible highest weight $ \widehat{\fg} $-modules $ L(\La) $ for the corresponding to $ \Pi $ choice of positive affine roots. Such a module is called integrable if it is integrable with respect to $ \widehat{\fg^{\#}_{\bz}}. $ The following theorem is proved along the same lines as Theorem \ref{Th5.4}.

\begin{theorem}
\label{Th6.3}
Let $ \La = \sum_{i=0}^{3} m_i \La_i, $ so that the level $ k = \sum_{i = 0}^{3} m_i. $ Then
\begin{enumerate}
\item[(a)] The $ \widehat{\fg} $-module $ L(\La) $ is integrable if and only if the following conditions hold:
\begin{enumerate}
\item[(i)] $ \frac{m_1 + m_2}{a}, \ \frac{m_0 + m_3}{a}, \ -\frac{m_1 + m_3}{a+1}, $ and $ -\frac{m_0 + m_2}{a+1} $ are non-negative integers;
\item[(ii)] $ m_i + m_j = 0 \Rightarrow m_i = m_j = 0 $ for $ i=0 $ or 1, $ j = 2 $ or 3. 
\end{enumerate}
\item[(b)] The $ \widehat{\fg} $-module $ L(\La) $ is $ \al_0 +\al_1  $ (resp, $ \al_2 + \al_3 $)-integrable iff
\begin{enumerate}
\item[(i)] $ m_0 + m_1 $ (resp. $ m_2 + m_3 $) $ \in \ZZ_{\geq 0} \, ; $
\item[(ii)] $ m_0 + m_1 $ (resp. $ m_2 + m_3 $) $ = 0 \Rightarrow m_0 = m_1=0 $ (resp. $ m_2 = m_3=0 $).
\end{enumerate}
\end{enumerate}
\end{theorem}

\begin{corollary}
\label{Cor6.4}
If there exists an integrable $ D(2,1;a) $-module $ L(\La), $ such that its level 
is non-zero, then $ a \in \QQ, $ and we may assume that $ -1 < a <0. $ so that $ a $ is of the form (\ref{eq6.7}). In this case
\begin{equation}
\label{eq6.8}
k = - \frac{pqn}{p + q} \mbox{ for some positive integer } n.
\end{equation}
\qed
\end{corollary}

From now on we shall assume that $ a $ is of the form (\ref{eq6.7}) and $ k $ is of the form (\ref{eq6.8}). We shall study the integrable level $ k $  
$ \widehat{\fg}$-modules $ L(\La), $  such that $ (\La | \beta) = 0 $ (resp. $ (\La | \beta^{\prime}) = 0$), where $ \beta = \al_1, \beta^{\prime} = \al_0. $ (These conditions mean that $ m_1 = 0  $ (resp. $ m_0 = 0 $)).

Let $ T = \{ \beta \} $ (resp. $ T^{\prime} = \{ \beta^{\prime} \} $), and choose the following $ \ZZ $-bases $ \gamma_1, \gamma_2 $ (resp. $ \gamma_1^{\prime}, \gamma_2^{\prime} $) of the lattice L, such that (\ref{eq3.1}) and (\ref{eq3.2}) hold:
\[ \gamma_1 = \frac{\al_1 + \al_3}{a + 1} = - \gamma^{\prime}_1, \ \gamma_2 = \frac{\al_1 + \al_2}{a} + \frac{\al_1 + \al_3}{a+1} = \gamma^{\prime}_2. \]
\noindent Then we have $ \tilde{\gamma}_1 = \gamma_1 = - \tilde{\gamma}^{\prime}_1 $ and $ \tilde{\gamma}_2 =- \frac{p+q}{pq} ((p+q)\al_1 + q\al_2 - p\al_3), \  \tilde{\gamma}^{\prime}_2= - \frac{p+q}{pq} ((p+q) \al_1 +(2p+q) \al_1 + p\al_3 ). $

The Lie superalgebra $\hat{\fg}$ has an involution $\sigma_0$, which permutes
$\al_0$ with $\al_1$ and $\al_2$ with $\al_3$. We choose $\Lambda_1 =\sigma_0(\Lambda_0)$. Note that $\sigma_0(\beta)=\beta^{\prime}$.
 
Denote by $ \Omega_k $ (resp. $ \Omega^{\prime}_k $) the set of all level $ k $ weights $ \La, $ such that the $ \widehat{\fg} $-module $ L (\La) $ is integrable and $ (\La | \beta ) = 0$ (resp. $ (\La | \beta^{\prime}) = 0) $).
We deduce from Theorem \ref{Th6.3} the following corollaries.
\begin{corollary}
\label{Cor6.5}
\begin{enumerate}
\item[(a)] Let $ \La \in \Omega_k. $ Then, up to adding a multiple of $ \delta, $ we have
\begin{equation}
\label{eq6.9}
\La = k \La_0 - \frac{pq}{2 (p+q)^2} k_2 \tilde{\gamma}_2 + k_1 \al_1,
\end{equation}
where $ k_1, k_2 \in \ZZ  $ satisfy one of the following conditions:
\begin{enumerate}
\item[(i)] $ -pn < k_2 <qn, \ \max \{ 0, -k_2\} \leq k_1 < \min \{ qn - k_2, pn\}; $
\item[(ii)] $ k_1 = 0,\, k_2 = qn; $
\item[(iii)] $ k_1\, = pn, k_2 = -pn. $
\end{enumerate} 
\item[(b)] Let $ \La \in \Omega^{\prime}_k. $ Then up to adding a multiple of $ \delta, $ we have
\begin{equation}
\label{eq6.10}
\La = k(\La_0 + \gamma^{\prime}_1) - \frac{pq}{2 (p+q)^2} k_2 \tilde{\gamma}^{\prime}_2 - k_1 \theta, 
\end{equation}
where $ k_1, k_2 \in \ZZ  $ satisfy one of the following conditions:
\begin{enumerate}
\item[(i)] $ pn < k_2 < pn + (p+q)n, \ \max \{ qn, k_2 -pn\} \leq k_1 < \min \{ (p+q)n, k_2 -pn + qn\}; $
\item[(ii)]$  k_1 = (p+q)n,\, k_2 = (2p+q)n; $
\item[(iii)] $ k_1 = qn,\, k_2 = pn. $
\end{enumerate}
\end{enumerate}
\qed
\end{corollary}
\begin{corollary}
\label{Cor6.6}
Let $ \La \in \Omega_k $ be of the form (\ref{eq6.9}). Then
\begin{enumerate}
\item[(a)] $ L(\La) $ is $ \al_0 + \al_1 $-integrable iff one of the following three conditions holds: 
\begin{enumerate}
\item[(i)] $ k_2 = qn - (p+q)s, $ where $  s \in \ZZ,\, 0 < s < n, $ and $ ps \leq k_1 < \min \{ (p+q)s, pn\}; $
\item[(ii)] or (iii) from Corollary \ref{Cor6.5}(a).
\end{enumerate}
\item[(b)]  If $ -pn < k_2 < qn, $ then there exists $ k_1, $ such that $ L(\La) $ is not $ \al_0 + \al_1 $-integrable.
\item[(c)] $ L(\La) $ is $ \al_2 + \al_3 $-integrable iff $ \La = k \La_0 $ (i.e. $k_1 = k_2 = 0  $).
\end{enumerate}
\qed
\end{corollary}
\begin{corollary}
\label{Cor6.7}
Let $ \La \in \Omega^{\prime}_k $ be of the form (\ref{eq6.10}). Then
\begin{enumerate}
\item[(a)] $ L (\La) $ is $ \al_0 + \al_1 $-integrable iff one of the following three conditions holds:
\begin{enumerate}
\item[(i)] $ k_2 = s(p+q) - qn, $ where $ s \in \ZZ,\, n<s<2n, $ and $ ps + (q-p)n \leq k_1 \\ < \min \{ (p+q)n, s(p+q) -pn   \}; $
\item[(ii)] or (iii) from Corollary \ref{Cor6.5} (b). 
\end{enumerate} 
\item[(b)] If $ pn < k_2 < (2p+ q)n, $ then there exists $ k_1, $ such that $ \La  $ is not $ \al_0 + \al_1 $-integrable. 
\item[(c)] $ L (\La) $ is $ \al_2 + \al_3 $-integrable iff $ \La = k \La_1 $ (i.e. $ k_1 = qn, k_2 = (p+q)n $).
\end{enumerate}
\qed
\end{corollary}
\begin{remark}
\label{Rem6.8}
It follows from Corollaries \ref{Cor6.6} and \ref{Cor6.7} that $ L(\La) $ with $ \la \in \Omega_k \cup \Omega^{\prime}_k $ cannot be simultaneously $ \al_0 + \al_1 $-integrable and $ \al_2 + \al_3 $-integrable. 
\end{remark}

\begin{conjecture}
\label{Conj6.9}
Let $ L (\La) $ be an integrable level $ k$ ($\neq 0$) \ $\widehat{\fg}$-module, such that $ \La \in \Omega_k $ (resp. $ \La \in \Omega^{\prime}_k $). 
Let $ \overline{W} = W^{\#} (= \langle r_{\al_1 + \al_2}, r_{\al_1 + \al_3} \rangle ) $ if $ L (\La) $ is neither $ \al_0 + \al_1 $-integrable, nor $ \al_2 + \al_3 $-integrable, and let $ \overline{W} = \langle W^{\#}, r_{\al} \rangle $ if $ L(\La) $ is either $ \al = \al_0 + \al_1 $-integrable, or $ \al = \al_2 + \al_3  $-integrable. Let $ c_{\La} = 2 $ if $ \La \in \Omega_k $ (resp. $ \La \in \Omega^{\prime}_k $) and in cases (ii) or (iii) of 
Corollary \ref{Cor6.5}(a) 
(resp.\ref{Cor6.5}(b)), and let $c_{\La}=1$ otherwise. Then the normalized supercharacter of $ L(\La) $ is given by the following formula, similar to (\ref{eq4.12}):
\begin{equation}
\label{eq6.11}
\widehat{R}^- \ch^-_{\La} = \frac{1}{c_{\La}} \sum_{w \in \overline{W}} \epsilon^- (w) w ( \Theta^L_{\La, T (\mbox{resp. } T^{\prime})}  ).
\end{equation}
\end{conjecture}
As before, the formula for the modified normalized supercharacter $ \tilde{\ch}^-_{\La} $ is obtained if we replace mock theta function $ \Theta^L $ by its modification $ \tilde{\Theta}^L $ in the RHS of (\ref{eq6.11}). By (\ref{eq3.7}), we obtain the following explicit formulas for the modified mock theta function:
\begin{equation}
\label{eq6.12}
\tilde{\Theta}^L_{\La, T} = e^{2 \pi i kt } \Theta_{k_2,\, (p+q)n} (\tau, \al_1 + (a+1)\al_2 + a \al_3) \ \tilde{\Phi}^{[pn]} (\tau, -\al_1, -\al_3),
\end{equation}
\begin{equation}
\label{eq6.13}
\tilde{\Theta}^{L}_{\La, T^{\prime}} = e^{2 \pi i kt } \Theta_{-k_2,\, (p+q)n} (\tau, -\al_1 + (a-1)\al_2 + a \al_3) \ \tilde{\Phi}^{[pn]} (\tau, \theta, -\al_2).
\end{equation}
Using these formula, we can write down explicit formulas for the modified normalized supercharacters. 
\begin{proposition}
\label{Prop6.10}
Let $ \La \in \Omega_k $ be of the form (\ref{eq6.9}), and let $N=(p+q)n$.
\begin{enumerate}
\item[(a)] If $ L(\La) $ is neither $ \al_0 + \al_1 $-integrable nor $ \al_2+\al_3 $-integrable, then
\[ \widehat{R}^- \tilde{\ch}^-_{\La} =e^{2 \pi i kt} \left( \Theta_{k_2,\,N} (\tau, \al_1 + (a+1)\al_2 + a \al_3) \ \tilde{\Phi}^{[pn]} (\tau, -\al_1, -\al_3) \right. \]  
 \[ \left. + \Theta_{k_2,\, N} (\tau, -\al_1 + (a-1)\al_2 + a \al_3) \ \tilde{\Phi}^{[pn]} (\tau, \theta, -\al_2)  \right). \]
\item[(b)] If $ L(\La) $ is $ \al_0 + \al_1 $-integrable, then
\[ \widehat{R}^- \tilde{\ch}^-_{\La} = 
 \frac{1}{c_{\La}} e^{2 \pi i kt} \left(( \Theta_{k_2,\, N} + \Theta_{-k_2 + 2qn,\, N})  (\tau, \al_1 +(a+1) \al_2 + a \al_3)  \tilde{\Phi}^{[pn]} (\tau, -\al_1, -\al_3)\right. \]
 \[+ \left. (\Theta_{k_2,\, N} + \Theta_{-k_2 + 2qn,\, N}) \!(\tau, -\al_1  +  (a-1)\al_2 +a \al_3)  \tilde{\Phi}^{[pn]} (\tau, \theta, -\al_2 ) \right) \!. \]
\item[(c)] In cases (ii) and (iii) of Corollary \ref{Cor6.5} (a) we have:
\[ \widehat{R}^- \tilde{\ch}^-_{\La} = e^{2 \pi i kt} \left( \Theta_{k_2,\,N} (\tau, \al_1+ (a+1) \al_2 + a \al_3) \ \tilde{\Phi}^{[pn]} (\tau, -\al_1, 
-\al_3)\right. \]
 \[ \left. + \Theta_{k_2,\,N } (\tau, - \al_1 + (a-1) \al_2 + a \al_3) \ \tilde{\Phi}^{[pn]} (\tau, \theta, -\al_2) \right). \]
\item[(d)] If $ L(\La) $ is $ \al_2 + \al_3 $-integrable, i.e. $ \La = k \La_0 $ (by Corollary \ref{Cor6.6}(c)), then
\[ \widehat{R}^- \tilde{\ch}^-_{\La} = e^{2 \pi i kt}  \left( \Theta_{0,\, N} (\tau, \al_1+ (a+1) \al_2 + a \al_3) \ \tilde{\Phi}^{[pn]} (\tau, -\al_1, -
\al_3 )\right. \]
\[ \left. + \Theta_{0,\, N} (\tau, - \al_1 + (a-1) \al_2 + a \al_3) \ \tilde{\Phi}^{[pn]} (\tau, \theta, -\al_2) \right). \]
\end{enumerate}
\qed
\end{proposition}

Let $ \overset{\circ}{\Omega}_k = \{ \La \in \Omega_k | \, \mbox{either} \, 
\La \, \mbox{is not} \, (\al_0 + \al_1)\mbox{-integrable, or} \, (k_1, k_2) = (0, qn) \mbox{ or }=(pn, -pn)\}. $
\begin{corollary}
\label{Cor6.11}
If $ \La \in \overset{\circ}{\Omega}_k $ 
is of the form (\ref{eq6.9}),
then the range of $ k_2 $ in $ \La $ is $ \{ j \in \ZZ | -pn \leq j \leq qn  \}, $ and we have:
\[  \widehat{R}^- \tilde{\ch}^-_{\La} 
= e^{2 \pi i kt} \left(  \Theta_{k_2,\,(p+q)n} (\tau, \al_1 + (a+1) \al_2 + a \al_3) \ \tilde{\Phi}^{[pn]} (\tau, -\al_1, - \al_3)\right. \] 
\[ \left. +\Theta_{k_2,\, (p+q)n} (\tau, - \al_1 + (a-1) \al_2+ a \al_3) \ \tilde{\Phi}^{[pn]} (\tau, \theta,- \al_2) \right).  \]
\qed
\end{corollary}
\begin{proposition}
\label{Prop6.12}
Let $ \La \in \Omega^{\prime}_k $ be of the form (\ref{eq6.10}), and let $N=(p+q)n$.
\begin{enumerate}
\item[(a)] If $ L(\La) $ is neither $ \al_0 + \al_1 $-integrable nor $ \al_2 + \al_3 $-integrable, then
\[ \widehat{R}^- \tilde{\ch}^-_{\La} = e^{2 \pi ik t} \left( \Theta_{-k_2,\,N}  (\tau, \al_1 + (a+1) \al_2 + a \al_3) \ \tilde{\Phi}^{[pn]} (\tau, -\al_1, - \al_3)\right. \] 
\[ \left. + \Theta_{-k_2,\, N} (\tau, - \al_1 + (a-1) \al_2 + a \al_3) \ \tilde{\Phi}^{[pn]} (\tau, \theta, -\al_2) \right). \]
\item[(b)] If $ L(\La) $ is $ \al_0 + \al_1 $-integrable, then
\[ \widehat{R}^- \tilde{\ch}^-_{\La}= 
 \frac{1}{c_{\La}} e^{2 \pi ik t}  \left( (\Theta_{-k_2,\, N} + \Theta_{k_2 + 2qn,\, N}) (\tau, \al_1 + (a+1) \al_2 + a \al_3) \ \tilde{\Phi}^{[pn]} (\tau, -\al_1, -\al_3)\right. \]
\[ \left. + ( \Theta_{-k_2,\, N} + \Theta_{k_2 + 2qn,\, N} ) (\tau, -\al_1 + (a-1) \al_2 + \al_3) \ \tilde{\Phi}^{[pn]} (\tau, \theta, -\al_2)\right). \]
\item[(c)] In cases (ii) and (iii) of Corollary \ref{Cor6.5}(b) we have:
\[ \widehat{R}^- \tilde{\ch}^-_{\La} = e^{2 \pi ik t} \left( \Theta_{-k_2,\, N} (\tau, \al_1 + (a+1) \al_2 + a \al_3) \ \tilde{\Phi}^{[pn]} (\tau, - \al_1, -\al_3) \right. \]  
\[ \left. +\Theta_{-k_2,\, N} (\tau, - \al_1 + (a-1) \al_2 + a \al_3) \ \tilde{\Phi}^{[pn]} (\tau, \theta, -\al_2) \right).\]
\item[(d)] If $ L(\La) $ is $ \al_2 + \al_3 $-integrable, i.e. $ \La = k \La_1 $ (by Corollary \ref{Cor6.7}(c)), then
\[ \widehat{R}^- \tilde{\ch}^-_{\La} = e^{2 \pi ik t} \left(  \Theta_{N,\, N} (\tau, \al_1 + (a+1) \al_2 + a \al_3) \ \tilde{\Phi}^{[pn]} (\tau, -\al_1, - \al_3)\right. \] 
\[ \left. +  \Theta_{N,\, N} (\tau, -\al_1 + (a-1) \al_2 + a \al_3) \ \tilde{\Phi}^{[pn]} (\tau, \theta, -\al_2) \right). \]
\end{enumerate}
\qed
\end{proposition}

Let $ \overset{\circ}{\Omega}_k^{\prime} = \{  \La \in \Omega^{\prime}_k |\, \La \mbox{ is not } (\al_0 + \al_1)\mbox{-integrable}\}. $
\begin{corollary}
\label{Cor6.13}
If $ \La \in \Omega^{\prime}_k $ is of the form (\ref{eq6.10}), then the range of $ -k_2 $ in $ \La $ is $ \{j \in \ZZ\, |\,\,\, -(2p+q)n < j < -pn \}, $ and we have:
\[ \widehat{R}^- \tilde{\ch}^-_{\La} = e^{2 \pi i k t} \left( \Theta_{-k_2,\, (p+q)n} (\tau, \al_1 + (a+1) \al_2 + a \al_3 ) \ \tilde{\Phi}^{[pn]} (\tau, -\al_1, -\al_3)\right. \]
\[ \left. +  \Theta_{-k_2,\, (p+q)n} (\tau, -\al_1 + (a-1) \al_2 + \al_3) \ \tilde{\Phi}^{[pn]} (\tau, \theta, -\al_2) \right). \]
\qed
\end{corollary}

Note that $ \overset{\circ}{\Omega}_k \cap \overset{\circ}{\Omega}_k^{\prime} = \emptyset,   $ and define the map $ \nu: \overset{\circ}{\Omega}_k \cup \overset{\circ}{\Omega}_k^{\prime} \rightarrow \ZZ, $ letting $ \nu (\La) = k_2  $ (resp. $ -k_2 $) if $ \La \in \overset{\circ}{\Omega}_k $ (resp. $ \overset{\circ}{\Omega}_k^{\prime} $). Then
\[ \{ \nu(\La) |\, \La \in \overset{\circ}{\Omega}_k \cup \overset{\circ}{\Omega}_k^{\prime} \} = \left\{ j \in \ZZ \ | \ qn-2(p+q)n < j \leq qn \right\}. \]
\noindent Hence we obtain the following theorem.
\begin{theorem}
\label{Th6.14}
The span of the modified normalized characters $ \tilde{\ch}^-_{\la} $ with $ \la \in  \overset{\circ}{\Omega}_k \cup \overset{\circ}{\Omega}_k^{\prime} $ is $ SL_2 (\ZZ) $-invariant, and the transformation formulae are as follows:
\[  \tilde{\ch}^-_{\la} ( - \frac{1}{\tau}, \frac{z}{\tau}, t - \frac{|z|^2}{2 \tau} ) 
= \frac{1}{\sqrt{2n(p+q)}}  \sum_{\mu \in \overset{\circ}{\Omega}_k \cup  \overset{\circ}{\Omega}_k^{\prime} } e^{-\frac{\pi i}{(p+q)n} \nu (\la) \nu(\mu)} \tilde{\ch}^-_{\mu} (\tau,z,t), \]
\[ \tilde{\ch}^-_{\la}(\tau +1,z,t) = e^{\frac{\pi i }{2n(p+q)} \nu (\la)^2-\frac{\pi i}{12}} \tilde{\ch}^-_{\la}(\tau,z,t). \]
\qed
\end{theorem}

\subsection{Case $ \fg = \osp (3|2), $ subprincipal integrable $ \widehat{\fg} $-modules} $  $ \\
Here $ \fg $ is the same as in Section 5.3 with $ m = n = 1. $ As there, the set of roots is described in terms of an orthogonal basis $ \epsilon_1,\, \delta_1 $ of $ \fh, $ such that $ (\epsilon_1 | \epsilon_1) = - (\delta_1 | \delta_1 ) = \half. $ We have $ \Delta_{\bz} = \{  \pm 2 \epsilon_1, \,\pm \delta_1   \}, \quad \Delta_{\bo} = \{  \pm \epsilon_1 \pm \delta_1,\, \pm \epsilon_1  \}. $ We choose the following set of odd simple roots of $ \fg  $ and of $ \widehat{\fg}: $
\[ \Pi = \{ \al_1 = \delta_1 - \epsilon_1, \,\al_2 = \epsilon_1 \}, \quad \widehat{\Pi} = \{  \al_0 = \delta -  \epsilon_1-\delta_1, \,\al_1,\, \al_2   \}. \]
\noindent The highest root of $ \fg $ is $ \theta =  \epsilon_1 + \delta, $ and $ \rho = - \half \al_1,\, \widehat{\rho} = \half (\La_0 - \al_1). $ We let $ \fg^{\#}_{\bz} $ be the subalgebra of $ \fg_{\bz}, $ isomorphic to $ \sl (2) $ with the simple root $ \al_1+\al_2 $ (which is different from that in Section 5.3 for $ m = n = 1 $). The corresponding coroot lattice $ L = \ZZ \gamma, $ where $ \gamma =4 (\al_1 + \al_2), $ is negative definite. We let 
$\beta = \al_1, \ T = \{ \beta \}. $ 

In this section we consider non-critical (i.e. of level $ k \neq - \half $) irreducible highest weight $ \widehat{\fg} $-modules $ L(\La) $ for the corresponding to $ \widehat{\Pi} $ choice of positive affine roots. Such a $ \widehat{\fg} $-module is called subprincipal integrable if it is integrable with respect to $ \widehat{\fg}^{\#}_{\bz}. $ This means that $ L(\La) $ is $ \al_0 + \al_2 $-integrable and $ \al_1 + \al_2 $-integrable. 

Let $ \sigma_0 $ be the involution of the Lie superalgebra $ \widehat{\fg}, $ such that $ \sigma(\al_0) = \al_1 $ and $ \sigma (\al_2) = \al_2. $ We normalize $ \La_1 $ by the condition $ |\La_1 |^2 = 0,$ so that $ \sigma_0 (\La_0) = \La_1 = \La_0 + \al_0 - \al_1 $ and $ \sigma_0 (\widehat{\rho}) = \widehat{\rho}. $ Let $ \gamma^{\prime} = \sigma_0 (\gamma) = \al_0 + \al_2, \ \beta^{\prime} = \sigma_0 (\beta) = \al_0, \ T^{\prime} = \{ \beta^{\prime} \} $. Note that $ \sigma_0 $ maps $ \widehat{\fg}^{\#}_{\bz} $ to itself (though doesn't keep $ \fg^{\#}_{\bz} $ invariant).

In this section we consider irreducible highest weight $ \widehat{\fg} $-modules $ L(\La) $ for the corresponding to $ \widehat{\Pi} $ set of positive affine roots. We shall also assume, as before, that the level $ k $ of $ \La $ is not 0. 
Denote by $ \Omega_k $ (resp. $ \Omega^{\prime}_k $) the set of highest weight $ \La $ of subprincipal integrable $ \widehat{\fg} $-modules of level $ k $ such that $ (\La | T) = 0 $ (resp. $ (\La | T^{\prime}) = 0 $).
The following theorem is proved along the same lines as Theorem \ref{Th5.4}.
\begin{theorem}
\label{Th6.15}
Let $ \La \in \Omega_k $ (resp. $ \Omega^{\prime}_k $). Then
\begin{enumerate}
\item[(a)] 
Up to adding a scalar multiple of $ \delta$, $\La  $ is of the form
\begin{equation}
\label{eq6.14}
\La = k \La_0 + \frac{m}{2} \al_1 (\mbox{resp. } = k \La_1 + \frac{m}{2} 
\al_0),
\end{equation}
where the following conditions hold:
\begin{enumerate}
\item[(i)] $ k \in \tfrac{1}{4} \ZZ, \ k \leq - \half; $
\item[(ii)] $ m \in \ZZ, \ 0 \leq m \leq - (4k+2), $ or $k=-\frac{1}{2},\,m=1$.
\end{enumerate}
\item[(b)] Subprincipal integrable $ \widehat{\fg} $-module $ L (\La) $ of the form (\ref{eq6.14}) is $ 2\al_2 $-integrable iff $ m = 0,  $ i.e. $ \La = k \La_0  $ (resp. $ k \La_1 $).
\item[(c)]  Subprincipal integrable $ \widehat{\fg} $-module $ L (\La) $ with $ \La  $ of the form (\ref{eq6.14})  is $ \delta - 2 \al_2 (= \al_0 + \al_1) $-integrable iff $ k \in \half \ZZ_{<0}, \ -2k \leq m \leq -4k-2 $ and $  m \equiv -2k \mod 2 \ZZ. $
\end{enumerate}
\qed
\end{theorem}
\begin{corollary}
\label{Cor6.16}
\begin{enumerate}
\item[(a)] If $ k \in \tfrac{1}{4} \ZZ_{<-2}, $ and $ \La = k \La_0 + \frac{m}{2} \al_1 $ (resp. $ = k \La_1 + \frac{m}{2} \al_0 $), where $ m  $ = 0 or 1, then $ \La \in \Omega_k $ (resp. $ \Omega^{\prime}_k $) and $ \La $ is not $ \delta - 2 \al_2 $- integrable. 
\item[(b)] If $ 2k \notin \ZZ,  $ then $ \Omega_k $ (resp. $ \Omega^{\prime}_k $) $ \cap \{ \delta - 2 \al_2- \mbox{ integrable } \La\} = \emptyset. $
\end{enumerate}
\qed
\end{corollary}

Note that for $ \La \in \Omega_k  $ (resp. $ \Omega^{\prime}_k $) we have:
\[ \Theta_{\La + \widehat{\rho},\, T \ (\mbox{resp. } T^{\prime})} = \sum_{\gamma \in L} t_{\gamma} \frac{e^{\La + \widehat{\rho}}}{1 - e^{-\al_1 \  (\mbox{resp. } -\al_0)}},  \]
\noindent and let
\[ A_{\La + \widehat{\rho},\, T (\mbox{resp. } T^{\prime})}  = \Theta_{\La + \widehat{\rho},\, T \ (\mbox{resp. } T^{\prime})} - r_{\al_1 + \al_2} \ \Theta_{\La + \widehat{\rho},\, T \ (\mbox{resp. } T^{\prime})}. \]
The character formulas for subprincipal integrable $ \hat{\fg} $-modules, proved in \cite{GK}, can be rewritten as follows.
\begin{proposition}
\label{Prop6.17}
Let $ \La \in \Omega_k  $ (resp. $ \Omega^{\prime}_k $).
\begin{enumerate}
\item[(a)] If $ \La  $  is neither $ 2 \al_2 $-integrable, nor $ \delta-2 \al_2 $-integrable, then
\[  \widehat{R}^- \ch^-_{\La  } = A_{\La + \widehat{\rho},\, T 
\ (\mbox{resp. } T^{\prime})}
.\]
\item[(b)] If $ \La $ is $ 2 \al_2 $-integrable, then
\[ \widehat{R}^- \ch^-_{\La  } = \thalf \left\{ A_{\La + 
\widehat{\rho},\, T \ (\mbox{resp. }\, T^{\prime})} + r_{\al_2} A_{\La + 
\widehat{\rho}, \,T \ (\mbox{resp.}\,T^{\prime})} \right\}. \]
\item[(c)] If $ \La $ is $ \delta - 2 \al_2 $-integrable, then
\[ \widehat{R}^- \ch^-_{\La} =  A_{\La + \widehat{\rho},\, T \ (\mbox{resp. }\, T^{\prime})} - r_{\al_0 + \al_1} A_{\La + \widehat{\rho},\, T \ (\mbox{resp.}\,T^{\prime})}. \]
\end{enumerate}
\qed
\end{proposition}

Next, we introduce the following coordinates on $ \widehat{\fh}: $
\[ h = 2 \pi i (- \tau \La_0 + z_1 (\al_1 + 2 \al_2) + z_2 \al_1 + t \delta) = (\tau, z_1, z_2, t). \]

\noindent Then the functions $ A_{\La + \widehat{\rho},\, T \ (\mbox{resp. }\, T^{\prime})} $ where $ \La $ is of the form (\ref{eq6.14}), can be expressed in terms of the functions
\begin{equation}
\label{eq6.15}
\Psi^{[M;s]} (\tau, z_1, z_2, t): = e^{-2 \pi i Mt} \left( \Phi^{+[M;s]} (\tau, z_1, z_2) - \Phi^{+[M;s]} (\tau, -z_2, -z_1) \right), 
\end{equation}
\noindent where $ \Phi^{+[M;s]}  $ is defined by (\ref{eq1.10}), as follows:
\begin{equation}
\label{eq6.16}
A_{\La + \widehat{\rho}, \,T } = \Psi^{[-(2k+1); \frac{1-m}{2}]} (2 \tau, z_1, z_2, \frac{t}{2}),
\end{equation}
\begin{equation}
\label{eq6.17}
A_{\La + \widehat{\rho}, \, T^{\prime}} = e^{-\pi i (2k+1)(z_1 + z_2 + \tau)} \ \Psi^{[-(2k+1); \frac{1-m}{2}]} (2 \tau, z_1 + \tau, z_2 + \tau, \frac{t}{2}).
\end{equation}

Recall that in Section 1 we defined a modification of the function $ \Phi^{+[M;s]} (\tau, z_1, z_2),  $ denoted by $ \tilde{\Phi}^{+[M;s]} (\tau, z_1, z_2),  $ for any $ M \in \half \zp, s \in \half \ZZ. $ Denote by $ \tilde{\Psi}^{[M;s]} $ the function, obtained from $ \Psi^{[M;s]}, $ replacing $ \Phi^{+[M;s]} $ by $ \tilde{\Phi}^{+[M;s]} $ in (\ref{eq6.15}).

Due to Theorem \ref{Th6.15}, if $ \La \in \Omega_k \cup \Omega^{\prime}_k $ 
and $k\neq -\half$ (i.e. $k$ is non-critical), we have $ -(2k+1) \in \half \zp$ and $ \frac{1-m}{2} \in \half \ZZ $ in (\ref{eq6.16}), (\ref{eq6.17}). Hence we can define modifications $  \tilde{A}_{\La + \widehat{\rho},\, T \ (\mbox{resp. }\, T^{\prime})} (\tau, z_1, z_2, t), $ replacing $ \Psi $ by $ \tilde{\Psi} $ in (\ref{eq6.16}) and (\ref{eq6.17}). Hence, by Proposition \ref{Prop6.17}, we can define the modified normalized supercharacters $ \tilde{\ch}^-_{\La}, $ replacing $ A $ by $ \tilde{A} $ in this proposition. By Corollary \ref{cor1.4} (a) the function $ \tilde{\Phi}^{+[M;s]}, M \in \half \zp, s \in \half \ZZ, $ depends only on $ s \mod \ZZ. $ Thus we obtain the following corollary. 
\begin{corollary}
\label{Cor6.18}
The modified normalized supercharacter $ \tilde{\ch}^-_{\La} $ for $ \La \in \Omega_k \cup \Omega^{\prime}_k,\, k \neq -\half, $ of the form (\ref{eq6.14}) depends only on $ k $, and on $ m \mod \ZZ. $ Consequently, the span of all modified normalized supercharacters $ \tilde{\ch}^-_{\La} $ for $ \La \in \Omega_k $ (resp. $ \Omega^{\prime}_k $), $ k \neq -\half, $ is spanned by $ \tilde{\ch}^-_{k \La_0} $ and $ \tilde{\ch}^-_{k \La_0 + \half \al_1} $ (resp. by $ \tilde{\ch}^-_{k \La_1} $ and $ \tilde{\ch}^-_{k \La_1 + \half \al_0} $).
\qed
\end{corollary}

\begin{lemma}
\label{Lemma6.19}
For $ M \in \half \zp, s \in \half \ZZ $ one has:
\[ 2\tilde{\Psi}^{[M;s]} (2 \tau, z_1, z_2, t) =  \tilde{\Psi}^{[2M;2s]} (\tau, \frac{z_1}{2}, \frac{z_2}{2}, \frac{t}{2} ) + e^{-2 \pi i s} \tilde{\Psi}^{[2M;2s]} ( \tau, \frac{z_1 +1}{2}, \frac{z_2 -1}{2}, \frac{t}{2}) .\]
\qed
\end{lemma}

By Proposition \ref{Prop6.17}, formulae (\ref{eq6.16}), (\ref{eq6.17}), and Lemma \ref{Lemma6.19}, we obtain for $ k \in \frac{1}{4} \ZZ, \\ k < - \half: $
\begin{equation}
\label{eq6.18}
\tilde{\ch}^-_{k \La_0} = \half (f_1 - f_2), \quad \tilde{\ch}^-_{k \La_0 + \half \al_1} = \half (f_1 + f_2),
\end{equation}
\begin{equation}
\label{eq6.19}
\tilde{\ch}^-_{k \La_1} = - \half (f_3 - f_4), \quad \tilde{\ch}^-_{k \La_1 + \half \al_1} = - \half (f_3 + f_4),
\end{equation}
\noindent where (we can remove the second superscript in $\tilde{\Psi}$ due to
Corollary \ref{cor1.2})
\begin{eqnarray*}
(\widehat{R}^- f_1) (\tau, z_1, z_2, t) &=& \tilde{\Psi}^{[-4k-2]} ( \tau, \frac{z_1}{2}, \frac{z_2}{2}, \frac{t}{4}), \\
(\widehat{R}^- f_2) (\tau, z_1, z_2, t) &=& \tilde{\Psi}^{[-4k-2]} ( \tau, \frac{z_1 + 1}{2}, \frac{z_2 +1}{2}, \frac{t}{4} ), \\
(\widehat{R}^- f_3) (\tau, z_1, z_2, t) &=& e^{-  \pi i (2k+1) (z_1 + z_2 + \tau)} \  \tilde{\Psi}^{[-4k-2]} ( \tau, \frac{z_1 + \tau}{2}, \frac{z_2 + \tau}{2}, \frac{t}{4} ), \\
(\widehat{R}^- f_4) (\tau, z_1, z_2, t) &=& e^{-  \pi i (2k+1) (z_1 + z_2 + \tau)} \  \tilde{\Psi}^{[-4k-2]} ( \tau, \frac{z_1 + \tau +1}{2}, \frac{z_2 + \tau + 1}{2}, \frac{t}{4} ). \\
\end{eqnarray*}
From Theorem \ref{th1.1}(a) and (\ref{eq5.6}) we obtain:
\begin{equation}
\label{eq6.20}
\begin{split}
f_1(-\frac{1}{\tau},\frac{z_1}{\tau},\frac{z_2}{\tau},t+\frac{z_1z_2}{\tau}) &=\tau f_1(\tau,z_1,z_2,t), \\
f_2(-\frac{1}{\tau},\frac{z_1}{\tau},\frac{z_2}{\tau},t+\frac{z_1z_2}{\tau}) &=\tau f_3(\tau,z_1,z_2,t), \\ f_3(-\frac{1}{\tau},\frac{z_1}{\tau},\frac{z_2}{\tau},t+\frac{z_1z_2}{\tau}) &=\tau f_2(\tau,z_1,z_2,t), \\ f_4(-\frac{1}{\tau},\frac{z_1}{\tau},\frac{z_2}{\tau},t+\frac{z_1z_2}{\tau}) &= (-1)^{4k}\tau f_4(\tau,z_1,z_2,t); \\  
\end{split}
\end{equation}
\begin{equation}
\label{eq6.21}
\begin{split}
f_i(\tau+1,z_1,z_2,t) & = f_i(\tau,z_1,z_2,t) \mbox{ for } i = 1,2, \\  
f_3 (\tau+1,z_1,z_2,t) & = -(- i)^{4k} f_4(\tau,z_1,z_2,t), \\
f_4 (\tau+1,z_1,z_2,t) & = -(- i)^{4k} f_3(\tau,z_1,z_2,t). \\
\end{split}
\end{equation}
Corollary \ref{Cor6.18}, along with (\ref{eq6.18})--(\ref{eq6.21}), imply the following theorem.
\begin{theorem}
\label{Th6.20}
The span of the modified normalized supercharacters $ \tilde{\ch}^-_{\La} $ for $ \La \in \Omega_k \cup \Omega^{\prime}_k, $ where 
$ k \in \frac{1}{4} \ZZ,\, k < - \half,  $ is a 4-dimensional $ SL_2 (\ZZ) $-invariant space, spanned by $ f_1, f_2, f_3, f_4. $
\qed
\end{theorem}

\begin{remark}
\label{Rm6.21}
The maximal non-critical $ k, $ for which $ \Omega_k \cup \Omega^{\prime}_k $ is non-empty, is $ k = - \frac{3}{4}. $ Note that for this value of $ k, $ we have $ -4k - 2 = 1, $ and that $ \Phi^{[1;0]}_{\add} (\tau, z_1, z_2) = \Phi^{[1;0]}_{\add} (\tau, -z_2, -z_1) $ (see (\ref{eq1.13})), hence $ \tilde{\Psi}^{[1]} = \Psi^{[1]}. $ Furthermore, by the denominator identity for $ s \ell (2|1)^{\hat{ }} $ we have:
\[ \Psi^{[1]} (\tau, z_1, z_2, 0) = -i \frac{\eta (\tau)^3 \vartheta_{11} (\tau, z_1 + z_2)}{\vartheta_{11} (\tau, z_1) \vartheta_{11} (\tau, z_2)}. \]
\noindent Hence, using the elliptic transformations of the functions $ \vartheta_{ab} $ (see e.g. Appendix of \cite{KW6}) we obtain:
\begin{eqnarray*}
(\widehat{R}^- f_1) (\tau, z_1, z_2, t)  & = & -i e^{-\frac{\pi i t}{2}} \frac{\eta (\tau)^3 \vartheta_{11} (\tau, \frac{z_1 + z_2}{2})}{\vartheta_{11} (\tau, \frac{z_1}{2}) \vartheta_{11} (\tau, \frac{z_2}{2})}, \\
(\widehat{R}^- f_2) (\tau, z_1, z_2, t)  & = & i e^{-\frac{\pi i t}{2}} \frac{\eta (\tau)^3 \vartheta_{11} (\tau, \frac{z_1 + z_2}{2})}{\vartheta_{10} (\tau, \frac{z_1}{2}) \vartheta_{10} (\tau, \frac{z_2}{2})}, \\
(\widehat{R}^- f_3) (\tau, z_1, z_2, t)  & = & -i e^{-\frac{\pi i t}{2}} \frac{\eta (\tau)^3 \vartheta_{11} (\tau{, \frac{z_1 + z_2}{2}})}{\vartheta_{01} (\tau, \frac{z_1}{2}) \vartheta_{01} (\tau, \frac{z_2}{2})}, \\
(\widehat{R}^- f_4) (\tau, z_1, z_2, t)  & = & -i e^{-\frac{\pi i t}{2}} \frac{\eta (\tau)^3 \vartheta_{00} (\tau, \frac{z_1 + z_2}{2})}{\vartheta_{00} (\tau, \frac{z_1}{2}) \vartheta_{00} (\tau, \frac{z_2}{2})}.\\
\end{eqnarray*}
Since for $ \osp (3|2)^{\hat{ }} $ we have by (\ref{eq4.5}):
\[ \widehat{R}^- (\tau, z_1, z_2, t) = -i e^{\pi i t} \frac{ \eta(\tau)^3 \vartheta_{11} (\tau, z_1 - z_2) \vartheta_{11} (\tau, \frac{z_1 + z_2}{2})   }{ \vartheta_{11} (\tau, z_1)  \vartheta_{11} (\tau, z_2) \vartheta_{11} (\tau, \frac{z_1 - z_2}{2})   }, \]
\noindent all four functions $ f_j, $ up to a factor $ \pm i e^{- \frac{3 \pi i t}{2}}, $ are ratios of products of the theta functions $ \vartheta_{ab}. $ 
Of course, these
 four functions satisfy the transformation properties (\ref{eq6.20}), (\ref{eq6.21}) with $ 4k $ replaced by $ -3 $. 
\end{remark}

\subsection{Fermionic construction of $ \mathbf{\osp (M|N)^{\hat{}}} $-modules} $  $ \\
Let $ V = V_{\bz} \oplus V_{\bo} $ be a vector subspace with $ \dim V_{\bz} = N = 2n$ even and $ \dim V_{\bo} = M, $ and let $ \langle .|. \rangle  $  be a non-degenerate super-skewsymmetric bilinear form on $ V $ (i.e. it is skewsymmetric on $ V_{\bz}, $ symmetric on 
$ V_{\bo} $ and $ \langle V_{\bz} | V_{\bo} \rangle = 0 $), 
and let $ \fg = \osp (M|N) $ be the corresponding ortho-symplectic Lie superalgebra. Let $ \Phi_V $ be the corresponding  to 
$ (V, \langle .|. \rangle ) $ 
vertex algebra of free superfermions. 

In our paper \cite{KW3}, Section 7, we constructed a representation of $ \widehat{\fg} $ in the superspace $ \Phi_V, $ which is integrable with respect to the Lie subalgebra $ \mathrm{so} \, (M)^{\hat{}}. $ We showed that (if $ (M,N) \neq (1,0) $ or (2,0)), this $ \widehat{\fg} $-module decomposes in a direct sum of two irreducible level 1 highest weight modules.
Note that these $ \widehat{\fg} $-modules are integrable if $ M \geq N+2 $ and subprincipal integrable otherwise. 
As we pointed out in \cite{KW3}, Remark 7.3, the normalized supercharacters 
of these $\hat{\fg}$-modules 
are theta functions, which, however, do not span an $ SL_2 (\ZZ) $-invariant space. In this section, using the twisted superfermions, we shall construct the missing, level $1$ $\hat{\fg}$-modules, integrable with respect to the affine Lie subalgebra $\mathrm{so}(M)^{\hat{}}$, which restore the $ SL_2 (\ZZ) $-invariance. For information on twisted fields we refer to e.g. \cite{KW5}.

First, consider the case $ M = 2m+1 $ is odd. Choose a basis $ B = B_{\bz} \cup B_{\bo} $ of $ V, $ where $ B_{\bz} = \{ \varphi^i,\, \varphi^{i \ast} | i = 1, \ldots, n  \} $ is a basis of $ V_{\bz} $ and $ B_{\bo} = \{ \psi,\, \psi^i,\, \psi^{i \ast} | i = 1, \ldots, m \} $ is a basis of $ V_{\bo}, $ such that all non-zero inner products between basis elements are:
\[ \langle \varphi^{i \ast}| \varphi^{i} \rangle = -\langle \varphi^i | \varphi^{i \ast} \rangle = 1, 
\ \langle \psi^{i\ast} | \psi^{i} \rangle =  \langle \psi^i | \psi^{i \ast} \rangle = 1, 
\ \langle \psi | \psi \rangle = 1.\]
\noindent Let $ \Phi^{\tw}_V $ denote the vector superspace with an even vector $ \vac $ and operators $ a_n $ for each $ a \in B $ and $ n \in \ZZ, $ such that 
\begin{enumerate}
\item[(i)] $ \left[ a_r, b_s \right] = \langle a | b \rangle \ \delta_{r, -s}I, \ a,b \in B, \ r, s \in \ZZ, $
\item[(ii)] $ \varphi^i_r \vac = \psi^i_r \vac = 0 \mbox{ for } r \geq 0,\, 
\varphi^{i \ast}_r\vac=\psi^{i \ast}_r \vac = \psi_r \vac = 0 \mbox{ for } r > 0, $ 
\item[(iii)] $ \Phi^{\tw}_V $ is irreducible with respect to all the operators $ a_n,\, a\in B,\, n\in \ZZ. $
\end{enumerate}
\noindent The fields $ \left\{ a(z) = \sum_{n \in \ZZ} a_n z^{-n-\half} \right\}_{a \in B} $ are called \textit{twisted superfermions}; define the annihilation part $ a(z)_- = \sum_{n : a_n \vac = 0} a_n z^{-n-\half} $ and the creation part $ a(z)_+ = a(z) - a(z)_-. $ For two such fields $ a(z) $ and $ b(z) $ their normally ordered product $: a(z) b(z): = \sum_{r \in \ZZ} :ab:_r z^{-r -1}$ is
computed by the following formula (see \cite{KW5}, (1.14)):
\[ :ab:_r = \sum_{k \in \ZZ_+} (a_{s_a-k-1}b_{r-s_a+k+1} +
(-1)^{p(a)p(b)} b_{r-s_a-k}a_{s_a+k})-s_a \langle a | b \rangle \delta_{r, 0}, \]
\noindent where $ s_a = - \half $ (resp. $ \half $) for $a= \psi^i $ and $\varphi^i $ (resp. for $a= \psi^{i \ast},\, \varphi^{i \ast}, $ and $ \psi $).

We construct a representation of the Lie superalgebra $ \wg $  in the superspace $ \Phi^{\tw}_V $ as follows. Given $ b \in \fg (=\osp (2m+1|2n)), $ denote by $ b(z) = \sum_{n \in \ZZ} (bt^n) z^{-n-1} $ the corresponding current in
$\hat{\fg}$. For the choice of simple roots of $ \fg $ as in 
\cite{KW3}, Table 6.1, let $ e_i, h_i, f_i,\, i = 1,2, \ldots, m+n, $ be the 
Chevalley generators of $ \fg $, and let $ e_0 \in \fg_{- \theta}, f_0 \in \fg_{\theta} $ be such that $ [e_0, f_0] = -h_0, \ [h_0, e_0] = 2 e_{0} $ ($ \theta $  is the highest 
root of $\fg$). It is straightforward to check that the following formulas define a representation of $ \wg $ in the superspace $ \Phi^{tw}_V $ with $ K=1 $ and $ d, $ such that $ d \vac = 0, \ [d, a_n] = -na_n $ for $ a \in B, 
n \in \ZZ$: 
\[e_0 (z) = \varphi^{1\ast} (z) \varphi^{1 \ast} (z); \ f_0 (z) = -\frac{1}{4}\varphi^1 (z) \varphi^1 (z); \ e_i (z) = \varphi^i (z) \varphi^{i + 1 \ast} (z) \,\,\hbox{for}\,\, i = 1, \ldots, n-1;\] 
\[ e_n (z) = \varphi^n (z) \psi^{1 \ast} (z); \ e_{n+i} (z) = \psi^i (z) \psi^{i + 1 \ast} (z) \,
\hbox{for }\, i=1, \ldots, m-1; \ e_{n+m} (z) = \psi^m (z) \psi (z);
\] 
\[h_0 (z) = - : \varphi^1 (z) \varphi^{1 \ast} (z): \,; \ h_i (z) = :\varphi^i (z) \varphi^{i \ast} (z): - :\varphi^{i+1}(z) \varphi^{i + 1 \ast} (z): 
\hbox{for} \, i = 1, \ldots, n-1;\] 
\[ h_n (z) = : \varphi^n (z)\varphi^{n \ast} (z): + : \psi^1 (z) \psi^{1 \ast} (z): \, ; \ h_{n+i} (z)= : \psi^i (z) \psi^{i \ast} (z): - :\psi^{i + 1} (z)\psi^{i + 1 \ast} (z):\, \]
\[ \hbox{ for} \, i = 1, \ldots, n-1; \ h_{n+m} (z) = 
2 : \psi^m (z) \psi^{m \ast} (z): \, . \]

Obviously, we have a $ \wg $-module decomposition $  \Phi^{\tw}_V = \Phi^{tw}_{\even} \oplus \Phi^{\tw}_{\odd}, $ where $ \Phi^{tw}_{\even} $ (resp. $ \Phi^{tw}_{\odd} $) is the span of monomials in $ a_n, \ a \in B, n \in \ZZ, $ of even (resp. odd) degree. Moreover, using the same method as that in the proof of Theorem 7.1 of \cite{KW3}, one can show that the $ \wg $-modules $ \Phi^{\tw}_{\even} $ and $ \Phi^{\tw}_{\odd} $ are irreducible, with highest weight vectors 
(for the choice of simple roots as in \cite{KW3}) $ \vac $ and $ \psi_0 \vac $ respectively. It is not difficult to show that
with respect to the choice of simple roots and the normalization of $ \bl $ as in Section 5.4, both modules have highest weight $ - \La_{m+n} $, and that for this choice of simple roots, 
we have the following isomorphisms of $ \wg $-modules:
 $ \Phi_{\even} \cong L(\La_{0})$,   $\Phi_{\odd} \cong L(\La_1). $ 
By Theorem \ref{th5.6}, provided that $ m > n+1, $ the three modules $ L(\La_0), L(\La_1), $ and $ L(-\La_{m+n}) $ are all $ so(2m+1)^{\hat{ }} $-integrable $ \wg $-modules of level 1, such that the highest weight vanishes on $ T. $

Recall that in the case $ M < N+2,  $ the bilinear form, considered in Section 5.4, is multiplied by $ - \half, $ see Section 5.2. In particular for $ \fg = \osp (3|2), $ considered in Section 6.4, we obtain the (critical) level $ k = - \half  $ $\hat{\fg}$-modules $ \Phi_V = L (- \half \La_0) \oplus L (-\half \La_1), $ and $ \Phi^{\tw}_V = L(\La_2) \oplus L (\La_2).$ Due to Theorem \ref{Th6.15} the three $ \wg $-modules $ L (- \half \La_0), L (- \half \La_1),  $ and $ L (\La_2) $ are all subprincipal integrable $ \wg $-modules of level $ - \half, $ such that the highest weight vanishes on an odd simple root. 

In \cite{KW3} we computed the character of the $ \wg $-module $ \Phi_V. $ In the same way we compute the supercharacter of $ \Phi^{\tw}_V $, hence of 
$L(-\Lambda_{m+n})$, in notation of Section 5.4:
\[ \ch^-_{\Phi^{\tw}_{even}} = \ch^-_{\Phi^{\tw}_{odd}} 
= e^{\La_{m+n}} \prod_{k=1}^{\infty} 
\frac{\prod_{i = 1}^{m}(1-e^{-\epsilon_i} q^{k-1})(1-e^{\epsilon_i} q^k)} 
{\prod_{i = 1}^{n}(1-e^{-\delta_i} q^{k-1})(1-e^{\delta_i} q^k)}.\]
Using this formula, we can write down 
in the coordinates
\begin{equation}
\label{eq6.22}
 h = 2 \pi i (-\tau \La_0 - \sum_{i = 1}^{m} x_i \epsilon_i + \sum_{i = 1}^{n} y_i \delta_i + t \delta )\, 
\end{equation}
the following formulas for the the normalized supercharacters: 
\begin{equation}
\label{eq6.23}
(\ch^-_{\La_0} + \ch^-_{\La_1})(\tau,x,y,t) = e^{2 \pi i t} \frac{\eta (\tau)^2 \eta (\tau)^{n-m}}{\eta (\frac{\tau}{2}) \eta (2 \tau)} \ \frac{\prod_{i =1}^{m} \vartheta_{00} (\tau, x_i)}{\prod_{i = 1}^{n} \vartheta_{00} (\tau, y_i)} , 
\end{equation}
\begin{equation}
\label{eq6.24}
(\ch^-_{\La_0} - \ch^-_{\La_1})(\tau,x,y,t) = e^{2 \pi i t} \frac{\eta (\frac{\tau}{2})}{\eta (\tau)} \eta(\tau)^{n-m} \ \frac{\prod_{i =1}^{m} \vartheta_{01} (\tau, x_i)}{\prod_{i = 1}^{n} \vartheta_{01} (\tau, y_i)} , 
\end{equation}
\begin{equation}
\label{eq6.25}
\ch^-_{-\La_{m+n}}(\tau,x,y,t) = (-i)^n e^{2 \pi i t} \frac{\eta (2 \tau)}{\eta (\tau)} \eta(\tau)^{n-m} \ \frac{\prod_{i = 1}^{m} \vartheta_{10} (\tau, x_i)}{\prod_{i = 1}^{n} \vartheta_{10}(\tau, y_i)}.
\end{equation}

Consequently, from the known modular transformation formulas for $ \eta (\tau) $ and $ \vartheta_{ab} (\tau, z), $ see e.g. Appendix of \cite{KW6}, we obtain, by (\ref{eq6.23})--(\ref{eq6.25}), modular transformation formulas for the normalized supercharacters:
\[(\ch^-_{\La_0} + \ch^-_{\La_1} ) \big{|}_S  =  \ch^-_{\La_0} + \ch^-_{\La_1}, \]
\[ (\ch^-_{\La_0} - \ch^-_{\La_1}) \big{|}_S = \sqrt{2} i^n \ \ch^-_{-\La_{m+n}}, \]
\[ \ch^-_{-\La_{m+n}} \big{|}_S = \frac{1}{\sqrt{2}} (-i)^n (\ch^-_{\La_0} - \ch^-_{\La_1}), \]
and the three normalized supercharacters are eigenvectors with respect to the transformation $ (\tau, x,y, t) \mapsto (\tau+1, x,y, t) $ with the eigenvalues $ e^{\frac{\pi i}{12}(n-m- \half)}, -e^{\frac{\pi i}{12}(n-m- \half)},  $ and $ e^{\frac{\pi i}{6}(m-n-\half)} $ respectively. 

Thus, the span of these three normalized supercharacters is $ SL_2 (\ZZ) $-invariant. Likewise, the span of the subprincipal integrable critical level $ k = - \half$ $\osp (3|2)^{\hat{ }} $-modules $ L (- \half \La_0), L (- \half \La_1),  $ and $ L (\La_2) $ is $ SL_2(\ZZ) $-invariant, with the transformation
matrices obtained from the above by letting $m=n=1$.

The case $ M= 2m $ is quite similar. The basis $ B_{\bz} $ of $ V_{\bz} $ is the same, and the basis of
$ V_{\hat{1}} $ is $ B_{\bo} = \{ \psi^i, \psi^{i \ast} |
\, i = 1, \dots, m    \},$ 
with the same inner products. The space $ \Phi^{\tw}_V $ is constructed in a similar way. The construction of the representation of $ \wg $ in $ \Phi^{\tw}_V $, where $ \fg = \osp (2m|2n), $ 
is constructed similarly. The only change is in the following formulas:
\begin{eqnarray*}
e_{n+m} (z) & = & \psi^{m-1} (z) \psi^m (z),\\
h_{n+m} (z) &=& : \psi^{m-1} (z) \psi^{m-1 \ast} (z): + :\psi^m (z) \psi^{m \ast} (z): \, .\\
\end{eqnarray*}
\noindent Then the $ \wg $-modules $ \Phi^{\tw}_{\even} $ and $ \Phi^{\tw}_{\odd} $ are irreducible, and, for the choice of simple roots as in \cite{KW3}, the highest weight vectors are $ \vac $ and $ \psi^{m \ast}_0 \vac $ with weights 
$ \La_{m+n} $ and $ \La_{m+n-1} $ respectively. Consequently, for the choice of simple roots as in Section 5.4, we obtain the following isomorphisms of $ \wg $-modules:
\[ \Phi_{\even} \cong L(\La_0), \quad \Phi_{\odd} \cong L(\La_1), \quad \Phi^{tw}_{\even} \cong L(\La_{m+n}), \quad \Phi^{\tw}_{\odd} \cong L(\La_{m+n-1}). \]
\noindent By Theorem \ref{th6.1}, for $ m>n $ we thus obtain all integrable $ \wg $-modules of level 1 with highest weights vanishing on $ T $ or $ T^{\prime}. $

The supercharacter formula for $ \Phi^{\tw}_V $ is similar to that in the case of odd $ M: $
\[ \ch^-_{\Phi^{\tw}_{\even}} \pm \ch^-_{\Phi^{\tw}_{\odd}} = e^{\La_{m+n}} \prod_{k = 1}^{\infty} 
\frac{\prod_{i =1}^{m} (1 \mp e^{- \epsilon_i} q^{k-1}) (1 \mp e^{\epsilon_i} q^k)}{\prod_{i=1}^{n} (1 \mp e^{-\delta_i} q^{k-1}) (1 \mp e^{\delta_i} q^k) }. \]
Using this formula, we can write down the normalized supercharacters in coordinates (\ref{eq6.22}) as follows:
\begin{equation}
\label{eq6.26}
(\ch^-_{\La_0} + \ch^-_{\La_1}) (\tau, x, y, t) = e^{2 \pi i t} \eta (\tau)^{n-m} \frac{\prod_{i =1}^{m} \vartheta_{00} (\tau, x_i)}{\prod_{i=1}^{n} \vartheta_{00} (\tau, y_i)}, 
\end{equation}
\begin{equation}
\label{eq6.27}
(\ch^-_{\La_0} - \ch^-_{\La_1}) (\tau, x, y, t) = e^{2 \pi i t} \eta (\tau)^{n-m} \frac{\prod_{i =1}^{m} \vartheta_{01} (\tau, x_i)}{\prod_{i=1}^{n} \vartheta_{01} (\tau, y_i)},
\end{equation}
\begin{equation}
\label{eq6.28}
(\ch^-_{\La_{n+m}} + \ch^-_{\La_{n+m-1}}) (\tau, x, y, t) = (-i)^n e^{2 \pi i t} \eta (\tau)^{n-m} \frac{\prod_{i =1}^{m} \vartheta_{10} (\tau, x_i)}{\prod_{i=1}^{n} \vartheta_{10} (\tau, y_i)},
\end{equation}
\begin{equation}
\label{eq6.29}
(\ch^-_{\La_{n+m}}  - \ch^-_{\La_{n+m-1}}) (\tau, x, y, t) = (-1)^n i^m e^{2 \pi i t} \eta (\tau)^{n-m} \frac{\prod_{i =1}^{m} \vartheta_{11} (\tau, x_i)}{\prod_{i=1}^{n} \vartheta_{11} (\tau, y_i)},
\end{equation}
The modular transformations of these normalized supercharacters are as follows:
\begin{eqnarray*}
(\ch^-_{\La_0} + \ch^-_{\La_1}) \big{|}_S & = & \ch^-_{\La_0} + \ch^-_{\La_1} \\
(\ch^-_{\La_0} - \ch^-_{\La_1}) \big{|}_S & = &  i^n   (\ch^-_{\La_{n+m}} + \ch^-_{\La_{n+m-1}}  ),\\
(\ch^-_{\La_{n+m}} + \ch^-_{\La_{n+m-1}}) \big{|}_S & = & (-i)^n (\ch^-_{\La_0} - \ch^-_{\La_1}),\\
(\ch^-_{\La_{n+m}}  - \ch^-_{\La_{n+m-1}}) \big{|}_S  & = & (-i)^{m-n} (\ch^-_{\La_{n+m}}  - \ch^-_{\La_{n+m-1}}), \\ 
\end{eqnarray*}
\noindent and the four normalized supercharacters are eigenvectors with respect to the transformation $ (\tau, x ,y, t) \mapsto (\tau +1, x, y, t) $ with the eigenvalues 
\[ e^{\frac{\pi i}{12}(n-m)}, \quad -e^{\frac{\pi i}{12}(n-m)}, \quad e^{-\frac{\pi i}{6}(n-m)}, \quad e^{-\frac{\pi i}{6}(n-m)} \]
\noindent respectively. 

\subsection{Modular invariance of characters from that of supercharacters} $  $ \\

We show here that the modular invariance of modified normalized supercharacters implies the modular invariance of modified normalized characters together with the modified normalized characters of supercharacters of twisted modules, i.e. of $ \wg^{\mathrm{tw}} $-modules, constructed in \cite{KW6}, Section 4. 

Recall that $ \ch^+_{L(\La)} $ is expressed via formula (\ref{eq4.4}) in terms of $ \ch^-_{L(\La)}, $ where $ \xi \in \fh $ is a fixed element, such that $ \al(\xi) \in \frac{p(\al)}{2} + \ZZ  $ for all $ \al \in \Delta. $ Hence we have
\begin{equation}
\label{eq6.30}
\ch^+_{\La} (h) = e^{-2 \pi i \La(\xi)} \ch^-_{\La} (h + 2 \pi i \xi).
\end{equation}
Also, we have, by \cite{KW6}, (4.7), the $ \wg^{\mathrm{tw}} $-modules, for which the normalized (super)characters are:
\begin{equation}
\label{eq6.31}
\ch^{\mathrm{tw}, \pm}_{\La} (h) = \ch^{\pm}_{\La} (t_{-\xi}(h))
\end{equation}
Since $ \gamma_i (\xi) \in \ZZ, \beta_i (\xi) \in \half + \ZZ, $ it follows from (\ref{eq6.30}) and (\ref{eq6.31}) the following equalities for modified characters in coordinates (\ref{eq1.5}):
\begin{equation}
\label{eq6.32}
\begin{split}
& \tilde{\ch}^+_{\La} (\tau, z, t) = e^{-2 \pi i \La (\xi)} \tilde{\ch}^-_{\La} (\tau, z + \xi, t), \\
& \tilde{\ch}^{\mathrm{tw}, -}_{\La} (\tau, z, t) = \tilde{\ch}^-_{\La} (\tau, z + \tau \xi, t + (\xi | z) + \frac{\tau}{2} |\xi|^2), \\
& \tilde{\ch}^{\mathrm{tw}, +}_{\La} (\tau, z, t) = e^{-2 \pi i \La (\xi)} \tilde{\ch}^-_{\La} (\tau, z+ \tau \xi + \xi, t + (\xi | z) + \frac{\tau}{2} |\xi|^2). \\
\end{split}
\end{equation}
It is straightforward to deduce from (\ref{eq6.32}) the modular invariance of the family  
\[ \{ \tilde{\ch}^+_{\La}, \ \tilde{\ch}^{\mathrm{tw}, +}_{\La}, \ \tilde{\ch}^{\mathrm{tw}, -}_{\La}\}_{\La \in \Omega},  \]
\noindent provided that the family $\{ \tilde{\ch}^-_{\La}\}_{\La \in \Omega}$
is modular invariant. 
\begin{proposition}
\label{Prop6.22}
Suppose that, for $ \la \in \Omega $ we have
\[ \tilde{\ch}^-_{\la} \bigg{|}_S = \sum_{\mu \in \Omega} S_{\la \mu} \tilde{\ch}^-_{\mu} \mbox{ and } \tilde{\ch}^-_{\la}\bigg{|}_{\left(\begin{smallmatrix}
1 & 1 \\
0 & 1 \\
\end{smallmatrix} \right)
} = e^{2 \pi i s_{\la}} \tilde{\ch}^-_{\la}.\]
\noindent Then we have for $ \la \in \Omega: $
\begin{enumerate}
\item[(a)]\[ \tilde{\ch}^+_{\la} \bigg{|}_S = \sum_{\mu \in \Omega} e^{-2 \pi i \la(\xi)} S_{\la \mu} \tilde{\ch}^{\mathrm{tw}, -}_{\mu}, \quad  \tilde{\ch}^{\mathrm{tw}, -}_{\la} \bigg{|}_S = \sum_{\mu \in \Omega} S_{\la, \mu}  e^{-2 \pi i \mu(\xi)} 
\tilde{\ch}^+_{\mu}, \]
\[  \tilde{\ch}^{\mathrm{tw}, +}_{\la} \bigg{|}_S = e^{-2 \pi i(k |\xi|^2 + \la (\xi))} \sum_{\mu \in \Omega} S_{\la, \mu} e^{-2 \pi i \mu (\xi)} \tilde{\ch}^{\mathrm{tw}, +}_{\mu}, \]
\item[(b)] \[ \tilde{\ch}^+_{\la}\bigg{|}_{\left(\begin{smallmatrix}
1 & 1 \\
0 & 1 \\
\end{smallmatrix} \right)
} = e^{2 \pi i s_{\la}} \tilde{\ch}^+_{\la}, \quad  \tilde{\ch}^{\mathrm{tw}, \pm}_{\la}\bigg{|}_{\left(\begin{smallmatrix}
1 & 1 \\
0 & 1 \\
\end{smallmatrix} \right)
} = e^{2 \pi i (s_{\la}+ \la(\xi))} e^{\pi i k |\xi|^2}  
\tilde{\ch}^{\mathrm{tw}, \mp}_{\la}. \]
\end{enumerate}
\qed
\end{proposition}

\end{document}